%% file: dirac_res_11.tex
\documentclass[a4paper,reqno]{amsart}

\input{commands.tex}      
\graphicspath{{../figures/}} 

\usepackage[plainpages=false,colorlinks,linkcolor=black,citecolor=black,urlcolor=black,breaklinks]{hyperref}

\usepackage{csquotes}
\usepackage{mathtools,dsfont}
\usepackage{subcaption}
\numberwithin{equation}{section}
\numberwithin{figure}{section}

\begin{document}
\title[Dirac operators with imaginary potentials]{Resolvent estimates for one-dimensional Dirac operators with imaginary potentials}

\author{Antonio Arnal}

\address[PS, AA]{Institute of Applied Mathematics, Graz University of Technology, Steyrergasse 30, 8010 Graz, Austria}

\email{aarnalperez01@qub.ac.uk}
\email{siegl@tugraz.at}

\author{Tho Nguyen Duc}

\address[TND]{Analytical and Algebraic Methods in Optimization Research Group, Faculty of Mathematics and Statistics, Ton Duc Thang University, Ho Chi Minh City, Vietnam}
\email{nguyenductho@tdtu.edu.vn}

\author{Petr Siegl}




\date{\today}

\begin{abstract}
We investigate massive one-dimensional Dirac operators perturbed by diagonal matrix potentials of the form $i V$ where the function $V$ is real-valued and unbounded at infinity. For such operators we find an $L^2$-realization with non-empty resolvent set using generalized coercivity and Schur complement dominance techniques. In the prototypical Airy-Dirac case $V(x)=x$, $x \in \R$, we derive the precise asymptotic behavior of the resolvent as the spectral parameter tends to infinity and also as the mass $m$ tends to $0$. Finally, we find the asymptotics of the resolvent norm for general potentials $V$ in terms of the Airy-Dirac resolvent, which in particular yields an asymptotic shape of $\eps$-pseudospectral curves and establishes the optimality of the pseudospectral region found in \cite{Krejcirik-2022-282}.
\end{abstract}

\maketitle


\section{Introduction}
\label{sec:intro}

We analyse massive one-dimensional Dirac operators in $L^2(\R) \oplus L^2(\R)$ with imaginary potentials such as
\begin{equation}
	\label{eq:H.gral.intro}
	H =  
	\begin{pmatrix}
		m & -i\partial_{x}\\
		-i\partial_{x} & -m
	\end{pmatrix} 
	+ i  \begin{pmatrix}
		V(x) & 0\\
		0 & V(x)
	\end{pmatrix},
\end{equation}
where $m > 0$ and $V \in C^1(\R,\R)$ is eventually increasing on $(0,+\infty)$ and unbounded at infinity (cf.~details in Assumption~\ref{asm:V.iR} below). Model cases include monomial potentials $V(x) = x^k$, $k \in \N$, $x \in \R$, for which the Dirac operator $H$ can be viewed as a relativistic counterpart of complex oscillators $-\partial_x^2 + i x^k$ in $L^2(\R)$ (see e.g.~\cite[Chap.~14]{Helffer-2013-book}, \cite[Chap.~14]{Davies-2007}, \cite{Sjoestrand-2019-14}). 

Our main goal is to obtain estimates for the resolvent norm $\| (H - \la)^{-1} \|$ when the spectral parameter $\la$ lies on the imaginary axis and in adjacent regions. In particular, we aim on finding the asymptotic shape of $\eps$-pseudosectral curves, i.e., the curves of the form $\la_b := a(b) + i b \in \C$, $b > 0$, for which
\begin{equation}\label{eps.curve.intro}
\| (H - \la_b)^{-1} \| = \frac{1}{\eps}(1+o(1)), \quad b \to + \infty, 
\end{equation}
where $\eps>0$ is given. 

In the non-relativistic Schr\"odinger case, such an analysis is closely related to the so-called Boulton's conjecture (see \cite{Boulton-2002-47}) resolved in \cite{Pravda-Starov-2006-73} and then further improved in \cite{BordeauxMontrieux-2013,ArSi-2023-284}; cf.~also \cite{Dencker-2004-57,Henry-2014,Almog-2016-48,Dondl-2016,Bellis-2018-9,Bellis-2019-277,Arnal-2026-263,Nguyen_25,Boulton_26}. We briefly summarise the known results and the ideas behind their proofs for odd or even $V \in C^2(\R,\R)$ which is eventually increasing on $(0, + \infty)$ (and satisfies further technical conditions in \cite[Asm.~3.1]{ArSi-2023-284}). 
In this case, \cite[Thm.~3.2]{ArSi-2023-284} yields  
\begin{equation}
	\label{eq:Hv.res.norm.intro}
	\| (-\partial_x^2 + i V(x) - i b)^{-1} \| = \| A^{-1} \| (V'(x_b))^{-\frac23} (1 + o(1)), \quad b \to +\infty,
\end{equation}
where $x_b>0$ is the turning point of $V$ defined by $V(x_b) = b$ for each sufficiently large $b>0$ and $A$ is the imaginary Airy operator in $L^2(\R)$
\begin{equation}\label{Airy.intro}
	A = -\partial_x^2 + ix.
\end{equation}
The Airy operator appears in numerous contexts, including the study of resonances \cite{Herbst-1979-64}, superconductivity \cite{Almog-2008-40,Almog-2010-300,Rubinstein-2010-195}, spectral instability \cite{Henry-2012-350} or spectral approximation in domain truncations \cite{Semoradova-2022-54}. It is well-known that the spectrum of $A$ is empty, the resolvent norm satisfies 
\begin{equation}\label{Airy.res.intro}
	\| (A - \la)^{-1} \| = \sqrt{\frac{\pi}{2}} \frac1{ \la^{\frac14}} \exp \left(\frac43 \la^{\frac32} \right) (1 + o(1)), \quad \la \to +\infty,
\end{equation}
see \cite[Cor.~1.4]{BordeauxMontrieux-2013}, \cite[Chap.~14]{Helffer-2013-book} or \cite{ArSi-2025}, and the numerics yields $\|A^{-1}\| = 0.7497\dots$. 

The asymptotic relation \eqref{eq:Hv.res.norm.intro} reflects that
\begin{equation}
V(x) - b = V(x) - V(x_b) = V'(x_b)(x-x_b)(1+o(1)), \quad b \to + \infty	
\end{equation}
if $x$ is sufficiently close to $x_b$ and $V$ together with $V'$ and $V''$ can be suitably controlled at infinity (here the technical conditions in \cite[Asm.~3.1]{ArSi-2023-284} are employed). Thus for $x$ near $x_b$, one expects that the operator $-\partial_x^2 + i V(x) - i b $ can be approximated by $-\partial_x^2 + i V'(x_b)(x-x_b)$ and a subsequent shift $x \mapsto x+x_b$ and scaling $x \mapsto V'(x_b)^{-\frac 13} x$ leads to a unitarily equivalent operator $V'(x_b)^\frac23 A$, cf.~\eqref{eq:Hv.res.norm.intro}. The proof of the latter is a rigorous justification of these steps, see \cite[Thm.~3.2]{ArSi-2023-284}. 

An analogous approximation by the Airy operator can be implemented also for the spectral parameter with a non-zero real part, which leads to 
\begin{equation}
	\label{eq:Hv.res.norm.a.intro}
	\| (-\partial_x^2 + i V(x) - (a+ i b))^{-1} \| = \| (A - a V'(x_b)^{-\frac 23})^{-1} \| (V'(x_b))^{-\frac23} (1 + o(1)), 
\end{equation}
as $b \to +\infty$ (if $a V'(x_b)^{-\frac 23}$ remains sufficiently small, for details see \cite[Prop.~5.1]{ArSi-2023-284} and \cite{ArSi-2025}). Hence equating $1/\eps$ and the l.h.s.~of \eqref{eq:Hv.res.norm.a.intro}, where \eqref{Airy.res.intro} is used in the r.h.s.,~yields that  
$\| (-\partial_x^2 + i V(x) - \la_b)^{-1} \| = \eps^{-1}(1+o(1))$ as $b \to + \infty$ holds on curves satisfying
\begin{equation}\label{eps.c.Schr.intro}
a(b)^\frac 32 = V'(x_b) \log \left(\frac{V'(x_b)^\frac 12}{\eps^\frac34} \right).
\end{equation}
(This is justified in \cite[Sec.~5.1]{ArSi-2023-284} if $V'(x) \to + \infty$ and $V(x)$ grows slower than $x^4$ as $x \to + \infty$ or $V(x) = x^k$ in \cite{BordeauxMontrieux-2013}).

Our analysis of the Dirac operator $H$ is inspired by the strategy above, however, several new effects appear here. We show in Theorem~\ref{thm:iR} that 
\begin{equation}\label{H.iR.intro}
\|(H - i b)^{-1} \| 
 = \| A_{m_b} ^{-1} \| (V'(x_b))^{-\frac12} 
		\left(1 + o(1)\right),
	\quad b \to \infty,
\end{equation}
where $A_m$ is the imaginary Airy-Dirac operator
\begin{equation}\label{Am.intro}
A_m = -i\Ntp \sigma_1 + m \sigma_3 + i x I_2 
= \begin{pmatrix}
	m & -i\partial_{x}\\
	-i\partial_{x} & -m
\end{pmatrix} 
+ i  
\begin{pmatrix}
	x & 0\\
	0 & x
\end{pmatrix}
\end{equation}
and $m_b$ is the transformed mass 
\begin{equation}\label{mb.intro}
	m_b = \frac{m} {V'(x_b)^{\frac12}},
\end{equation}
(which originates from a scaling when $H-ib$ is approximated near $x_b$). The properties of the Airy-Dirac operator $A_m$ are analogous to those of $A$ in \eqref{Airy.intro}, namely, if $m>0$, then the spectrum of $A_m$ is empty (see Proposition~\ref{prop:Airy.basic}) and one has (with $\la \in \R$, see Theorem~\ref{thm:airy.res.a})
\begin{equation}\label{Am.inv.la}
	\Vert (A_m-\la)^{-1} \Vert = \frac{ \Gamma(\frac{m^2}{2}+1)}{2^{\frac{m^2}{2}}m} \frac {\exp(\la^2)}{|\la|^{m^2}} 	\left(
	1 + \BigO(|\la|^{-2})
	\right), \quad \la \to \pm\infty,
\end{equation}
cf.~\eqref{Airy.res.intro}. However, a new feature is the behavior of $A_m^{-1}$ as $m \to 0+$, where 
(see Proposition~\ref{prop:Airy.basic})
\begin{equation}\label{Am.inv.intro}
	\|A_m^{-1}\| = \frac1m, \quad m>0,
\end{equation}
and the spectrum of $A_0$ is the whole $\C$ (see Proposition~\ref{prop:A0}).

The results \eqref{eq:Hv.res.norm.intro} and \eqref{H.iR.intro} seem analogous, however, if $V'(x) \to + \infty$ as $x \to \infty$, then the term $\|A_{m_b}^{-1}\|$ is unbounded as $b  \to +\infty$ and \eqref{H.iR.intro} in fact simplifies to  
\begin{equation}
\|(H - i b)^{-1} \|	= \frac 1m 
	\left(1 + o(1) \right),
	\quad b \to +\infty,
\end{equation}
(see Corollary~\ref{cor:iR} for details and also other cases of behavior of $V'$). This reveals an interesting difference between the Schr\"{o}dinger and Dirac settings. Along the positive imaginary axis, the resolvent norm decays to zero in the Schrödinger case, whereas in the Dirac case it converges to $\frac{1}{m}$ as the spectral parameter tends to infinity.

Similarly as for the Schr\"odinger case in \eqref{eq:Hv.res.norm.a.intro}, to find the curves for which \eqref{eps.curve.intro} holds, an approximation by Airy-Dirac operator leads to 
\begin{equation}\label{eq:resnorm.gen.intro}
	\|(H - (a+i b))^{-1} \| = \| (A_{m_b} - a V'(x_b)^{-1/2})^{-1} \| (V'(x_b))^{-\frac12} 
	\left(1 + o(1)\right),
\end{equation}
as $b \to +\infty$ if $a V'(x_b)^{-1/2}$ stays bounded (see Theorem~\ref{thm:iR}). For this case we also have an extension of \eqref{Am.inv.intro}, namely,  
\begin{equation}\label{Am.inv.intro.la}
\| (A_{m} - \la)^{-1} \| =  \frac{\exp(\la^2)}{m} (1 + \BigO(m)), \quad m \to 0+,
\end{equation}
for $\la \in [-R,R]$ with a fixed $R>0$, see Theorem~\ref{thm:airy.res.m}. Hence equating $1/\eps$ with the l.h.s.~of \eqref{eq:resnorm.gen.intro} yields that \eqref{eps.curve.intro} is satisfied for the curves (see Corollary~\ref{cor:iR})
\begin{equation}\label{ps.curv.intro}
a(b)^2 = V'(x_b) \log \frac m \eps;
\end{equation}
cf.~\eqref{eps.c.Schr.intro} and Figure~\ref{fig:main} for an illustration in the model cases $V(x)=x^k$ with $k=2,3$.
\begin{figure}[htbp]
	\centering
	\begin{subfigure}{0.45\textwidth}
		\centering
		\includegraphics[width=\linewidth]{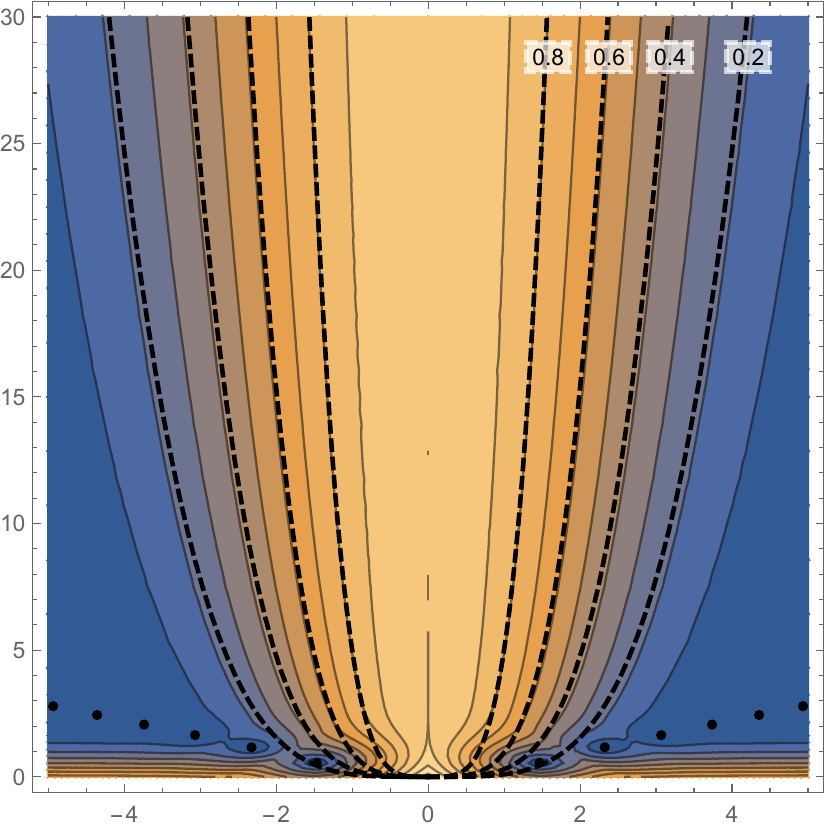}
		\caption{$V(x)=x^2$.}
		\label{fig:example x2}
	\end{subfigure}
	\hfill
	\begin{subfigure}{0.45\textwidth}
		\centering
		\includegraphics[width=\linewidth]{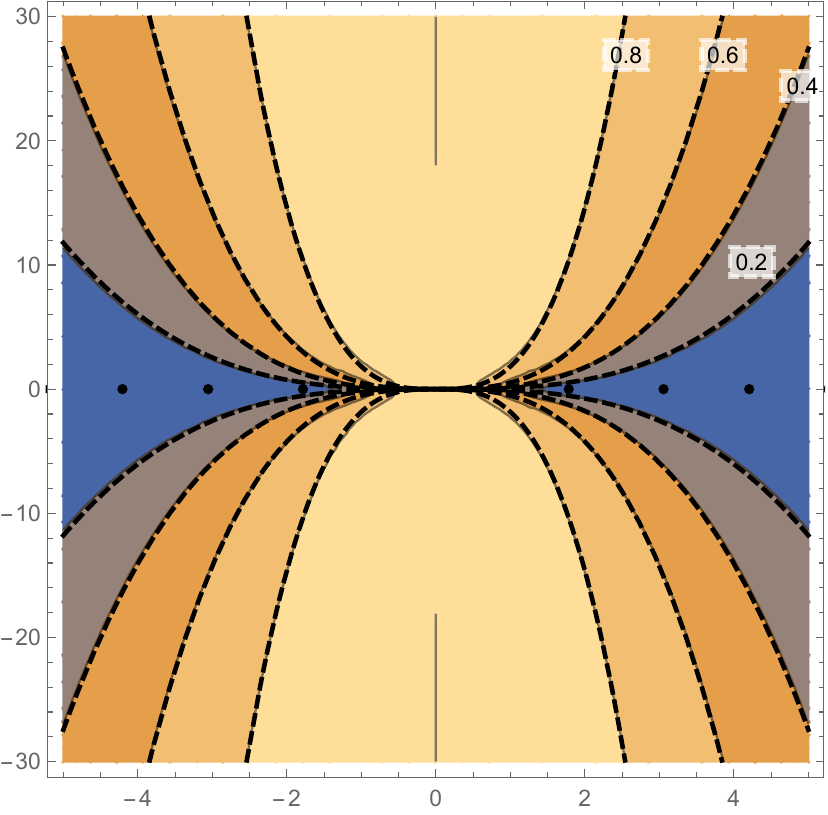}
		\caption{$V(x)=x^3$.}
		\label{fig:example x3}
	\end{subfigure}
	\caption{Numerical computation of pseudospectra associated with the polynomial potentials shown. Each operator in the operator matrix $H$ is expressed in the orthonormal basis of Hermite functions and truncated to a $400 \times 400$ matrix. The dashed curves are the pseudospectral curves \subref{fig:example x2} $a(b)^2 = 2 b^{\frac{1}{2}} \log\left(\frac{m}{\varepsilon}\right)$ and \subref{fig:example x3} $a(b)^2 = 3 \vert b \vert^{\frac{2}{3}} \log\left(\frac{m}{\varepsilon}\right)$ for $\varepsilon=0.2, 0.4, 0.6, 0.8$ and $m=1$. The black dots correspond to the respective numerical eigenvalues.}
	\label{fig:main}
\end{figure}

We also note that \eqref{eq:resnorm.gen.intro} shows that the pseudospectral region found in \cite{Krejcirik-2022-282} is optimal (see Remark~\ref{rmk:res}~\ref{rem opt} for more details and \cite{Krejcirik-2019-276,Arifoski-2020-52,Nguyen-2023-55} for non-semiclassical pseudomodes for Schr\"odinger, wave and biharmonic operators).

To obtain the results listed above (and also the existence of infinitely many eigenvalues and completeness of eigensystem for $V(x) \gs \langle x \rangle^\gamma$ with $\gamma>1$, see Corolarry~\ref{cor:V.bd.below}), several new ingredients are needed. In particular, we find a densely defined realisation of $H$ in $L^2(\R) \oplus L^2(\R)$ with non-empty resolvent set and give sufficient conditions on $V$ for the compactness of the resolvent, see Theorem~\ref{thm:H.def}. To this end, we employ the Schur complement dominance developed in \cite{Gerhat-2024-286} and the generalised coercivity of Almog-Helffer introduced in \cite{Almog-2015-40}, see Appendix~\ref{app:Schur.dom}.

A second key element is a detailed spectral analysis of the imaginary Airy-Dirac operator $A_m$ in \eqref{Am.intro}. The resolvent estimates for $A_m$ in the regime both $m \to 0+$, cf.~\eqref{Am.inv.intro}, \eqref{Am.inv.intro.la} and $\la \to \pm \infty$, cf.~\eqref{Am.inv.la}, are based on an explicit representation of the resolvent kernel. In fact, it is related to the resolvent of the harmonic oscillator
$\sL = -\partial_x^2 + x^2$ and that of the conjugated harmonic oscillator $\sL_\la = -(\partial_x+\la)^2 + x^2$ (see Section~\ref{ssec:Am.res} and  Appendix~\ref{app:HO}). The estimate for $\la \to \pm \infty$ is inspired by \cite{ArSi-2025} and it relies on the Schur test with non-trivial weights and a related choice of a pseudomode.

As indicated above, the proof of Theorem~\ref{thm:iR}, cf.~\eqref{eq:resnorm.gen.intro}, is inspired by \cite{ArSi-2023-284} and it consists of the following steps. First, in Proposition~\ref{prop:away.iR}, $\|(H-\la)u\|$ is estimated from below for functions $u$ having support outside of a neighbourhood of $x_b$. Second,	in Proposition~\ref{prop:local.iR}, a lower estimate of $\|(H-\la)u\|$ for $u$ supported in a neighborhood of $x_b$ is obtained and here the approximation by the Airy-Dirac operator $A_{m_b}- a V'(x_b)^{-\frac 12}$ is the crucial point. Third, in Proposition~\ref{prop:lbound.iR}, we construct a pseudomode, localised at $x_b$, which shows that the resolvent estimate for $H-\la$ cannot be improved; this step exploits the localisation technique and the approximation by the Airy-Dirac operator used in the second step. In the final fourth step, we combine the results from the previous ones with the aid of commutator estimates and a suitably constructed partition of unity.

Finally, we remark that the singularity of $A_m^{-1}$ at $m = 0$ also accounts for the more restrictive conditions on $V$ in Assumption~\ref{asm:V.iR} as compared to the Schr\"odinger case (where both logarithmically and exponentially growing potentials are allowed, see \cite[Asm.~3.1]{ArSi-2023-284}).

The remainder of our paper is structured as follows. In Section~\ref{sec:results}, we summarise our main results. In Sections~\ref{sec:H.def.proofs},~\ref{sec:Airy.proofs}~and~\ref{sec:resnorm.est.iR}, we prove the claims concerning the definition of $H$, the Airy-Dirac operators and the Dirac operators with general potentials, respectively. Appendices~\ref{app:HO}~and~\ref{app:Schur.dom} gather some facts on harmonic and conjugates oscillators and on generalized coercivity and Schur complement dominance, respectively.

\subsection{Notation}
\label{ssec:notation}

We write $\N_0 = \N \cup \{0\}$. If $\cX_1, \cX_2$ are two Banach spaces, $\cB(\cX_1, \cX_2)$ denotes the Banach space of bounded linear operators from $\cX_1$ to $\cX_2$; as usual, $\cB(\cX)= \cB(\cX, \cX)$. 

For a closed linear operator $T$ on a Banach space $\cX$, we denote its numerical range by $\Num(T)$, its spectrum by $\sigma(T)$, its point spectrum by $\spp(T)$, and its resolvent set by $\rho(T)$. The essential spectra are denoted by $\se{k}(T)$, $k = 1,\dots,5$, as in \cite[Chap.~IX]{EE}. 

The singular values of a compact operator $T$ in a Hilbert space $\cH$ are denoted by $s_k(T)$, $k \in \N$; They are listed in a non-increasing order and repeated according to their multiplicities, for details, see~\eg~\cite{Gohberg-1969,BS-1987}. For $0 < p < \infty$, let $\cS_{p,\infty}\coloneqq  \left\{\opT: \mathcal{H}\to \mathcal{H} \text{ is compact and }  s_{k}(T)=\mathcal{O}(k^{-1/p}) \text{ as }k\to+\infty\right\}$ denote the Schatten-Lorentz class and $\cS_{p}\coloneqq \left\{\opT : \mathcal{H}\to \mathcal{H} \text{ is compact and }  \sum_{k=1}^{\infty} (s_{k}(T))^p <\infty\right\}$ denote the Schatten class.

Unless stated otherwise, our underlying Hilbert space is $\Lt(\R, \C^2)$. Its elements are written as $u = (u_1, u_2)^t$ with $u_1, u_2 \in \Lt(\R) := \Lt(\R, \C)$. Its inner product and norm are denoted by $\langle\cdot,\cdot\rangle$ and $\| \cdot \|$, respectively, unless a subscript is needed to avoid ambiguity. We denote by $I_2$ both the $2 \times 2$ identity matrix and the identity operator on $\Lt(\R, \C^2)$. The Pauli matrices are
\begin{equation}
	\sigma_1 = \begin{pmatrix}
		0 & 1 \\
		1 & 0
	\end{pmatrix}, 
	\qquad	
	\sigma_2 = \begin{pmatrix}
		0 & -i \\
		i & 0
	\end{pmatrix},
	\qquad 
	\sigma_3 = \begin{pmatrix}
		1 & 0 \\
		0 & -1
	\end{pmatrix}.
\end{equation}
If $M: \R \to \C$ is a measurable function, we denote by $M$ the corresponding multiplication operator in $\Lt(\R)$ on the maximal domain,~\ie
\begin{equation}
\Dom(M) = \{f \in \Lt(\R) : M f \in \Lt(\R)\}.
\end{equation}
Similarly, for the multiplication operator by $M I_2$ in $L^2(\R,\C^2)$ we use the maximal domain
\begin{equation}
	\Dom(M I_2) = \{u \in \Lt(\R,\C^2) : M I_2 u \in \Lt(\R,\C^2)\}.
\end{equation}
The commutator of two operators $A$, $B$ is denoted by $[A,B]:=AB - BA$.

To avoid introducing multiple constants whose exact value is inessential for our purposes, we write $a \lesssim b$ to indicate that, given $a,b \ge 0$, there exists a constant $C>0$, independent of any relevant variable or parameter, such that $a \le Cb$. The relation $a \gtrsim b$ is defined analogously whereas $a \approx b$ means that $a \lesssim b$ \textit{and} $a \gtrsim b$.

\subsection*{Acknowledgements}

This research was funded in part by the Austrian Science Fund (FWF) 10.55776/P 33568-N.

\section{Summary of main results}
\label{sec:results}

We present our main results here, the proofs are given in further sections.

\subsection{\texorpdfstring{$L^2$}{L2}-realization with non-empty resolvent set}

Let $m > 0$ and assume $V \in C^1(\R,\R)$. Our first goal is to find a densely defined realization of the Dirac operator 
\begin{equation}
H = -i\Ntp \sigma_1 + m \sigma_3 + i V I_2 =  
	\begin{pmatrix}
		m & -i\partial_{x}\\
		-i\partial_{x} & -m
	\end{pmatrix} 
	+ i  \begin{pmatrix}
		V(x) & 0\\
		0 & V(x)
	\end{pmatrix},
\end{equation}
with a non-empty resolvent set in the Hilbert space $\Lt(\R, \C^2)$.  Theorem~\ref{thm:H.def} comprises a collection of results on basic spectral properties of $H$ obtained by these methods.

\begin{theorem}\label{thm:H.def}
Let $m > 0$, let $V \in C^1(\R,\R)$ and suppose that
\begin{equation}
	\label{eq:V.separ.cond}
	\exists \eps \in (0,1), \ \exists M_\eps >0 \, : \, |V'(x)| \le \eps |V(x)|^2 + M_\eps, \quad x \in \R.
\end{equation}
Let $H$ be the Dirac operator in $L^2(\R, \C^2)$ 
\begin{equation}\label{H.def}
H := -i\Ntp \sigma_1 + m \sigma_3 + i V I_2,
\quad
\Dom(H) := H^1(\R,  \C^2) \cap \Dom(V I_2).
\end{equation}
Then the following claims hold.

\begin{enumerate}[\upshape (i)]
\item \label{thm:H.def.i)} The adjoint operator $H^*$ reads
\begin{equation}\label{H*.def}
H^* = -i\Ntp \sigma_1 + m \sigma_3 - i V I_2,
\quad	\Dom(H^*)  = \Dom(H).
	\end{equation}
\item \label{thm:H.def.ii)}
There exists $C>0$ such that for all $u \in \Dom(H)$
\begin{equation}\label{H.gr.norm.sep}
	\begin{aligned}
		\| H u \|^2 + \| u \|^2 &\geq C \left(
		 \| u' \|^2 + \| V I_2 u \|^2 + \| u \|^2 \right),
		\\
		\| H^* u \|^2 + \| u \|^2 &\geq C \left( \| u' \|^2 + \| V I_2 u \|^2 + \| u \|^2  \right).
	\end{aligned}		
\end{equation}
Moreover, $C_{c}^{\infty}(\R,\C^2)$ is a core of $H$ and $H^*$.
\item \label{thm:H.def.iii)} The resolvent set of $H$ is non-empty, moreover,
\begin{equation}\label{H.rho.strip}
\{ z \in \C \, :\, \Re z \in (-m,m) \} \subset	\rho(H). 
\end{equation}
\item \label{thm:H.def.iv)} If, in addition to \eqref{eq:V.separ.cond}, $V$ satisfies $|V(x)| \to \infty$ as $|x| \to \infty$, then the resolvent of $H$ is compact. 
In particular, if for some $\gamma>0$ and $x_0 \geq 0$,
\begin{equation}\label{V.x.gamma}
|V(x)| \gs \langle x \rangle^\gamma,	\quad |x| \geq x_0,
\end{equation}
then the singular values of $H^{-1}$ satisfy
\begin{equation}
s_k(H^{-1}) = \BigO(k^{-1/p_\gamma}), \quad k \to + \infty, \quad p_\gamma = \frac{\gamma+1}{\gamma}.
\end{equation}
\item \label{thm:H.def.v)} Suppose that \eqref{V.x.gamma} is satisfied. For every $\delta>0$, there exists $c>0$ such that for all $\la \in \rho(H)$ with $|\la|\geq e$ and $\dist(\la, \sigma(H)) \geq \delta |\la|$  
\begin{equation}\label{res.est.Schatten}
\log \|(H-\la)^{-1}\| \leq c |\la|^{p_\gamma}.
\end{equation}
\end{enumerate}
\end{theorem}
\begin{remark}
	\label{rmk:D.massless}
	We note that the condition $m>0$ in Theorem~\ref{thm:H.def} is important for the spectral properties of $H$. Namely, if $m=0$, $V \in C(\R,\R)$ with 
	\begin{equation}
		\lim_{x \to \pm \infty} V(x) = \pm \infty,
	\end{equation}
	and $H$ is defined as in \eqref{H.def}, then
	\begin{equation}
		\sigma(H) = \spp(H) = \C.
	\end{equation}
	Indeed, for any $\la \in \C$, we have an eigenfunction $(-f_\la, f_\la) \in \Dom(H)$, where
	\begin{equation}\label{f.la.def}
		f_\la(x) := \exp \left(-\int_0^x V(t) \, \dd t - i \la x \right).
	\end{equation}
\end{remark}
The proof of Theorem~\ref{thm:H.def} is given in Section~\ref{sec:H.def.proofs}  (for some claims in Theorem~\ref{thm:H.def}, the assumption \eqref{eq:V.separ.cond} can be relaxed to \eqref{V.nabla.form}, see Section~\ref{sec:H.def.proofs} for details).

\subsection{Imaginary Airy-Dirac operator}

An important special case in our analysis is the Airy-Dirac operator
\begin{equation}\label{eq:airy.dirac.def}
\begin{aligned}
A_m &= -i\Ntp \sigma_1 + m \sigma_3 + i x I_2 
	= \begin{pmatrix}
		m & -i\partial_{x}\\
		-i\partial_{x} & -m
\end{pmatrix} 
+ i  
\begin{pmatrix}
		x & 0\\
		0 & x
\end{pmatrix},
\\
\Dom(A_m) & = H^1(\R, \C^2) \cap \Dom(x I_2).
\end{aligned}
\end{equation}
Similarly as for Schr\"odinger and wave operators (see \cite{ArSi-2023-284,Arnal-2026-263}) the precise analysis of $A_m$ allows us to establish resolvent estimates for Dirac operators with general potentials (see Theorem~\ref{thm:iR} below). Like the imaginary Airy operator $-\partial_x^2 + ix$ in $L^2(\R)$, $A_m$ has peculiar spectral properties. 
\begin{proposition}\label{prop:Airy.basic}
Let $m>0$ and let $A_m$ be as in \eqref{eq:airy.dirac.def}. Then the following claims hold.

\begin{enumerate}[\upshape (i)]
	\item For all $u \in \Dom(A_m) = \Dom(A_m^*)$
	\begin{equation}\label{Separation Airy}
		\begin{aligned}
			&\Vert A_{m} u \Vert^2 + \Vert u \Vert^2 \geq \Vert u' \Vert^2 + \Vert x I_2 u\Vert^2 +m^2 \Vert u \Vert^2 ,\\
			&\Vert A_{m}^{*} u \Vert^2 + \Vert u \Vert^2 \geq \Vert u' \Vert^2 + \Vert x I_2 u\Vert^2 +m^2 \Vert u \Vert^2.
		\end{aligned}
	\end{equation}
	\item \label{Spec 0} The resolvent of $A_{m}$ is compact,
	\begin{equation}
		\sigma(A_{m})=\emptyset	
	\end{equation}
	and 
	\begin{equation}\label{res.Airy.unit.eq}
		\Vert (A_{m}-\lambda)^{-1} \Vert= \Vert (A_{m}-\Re \lambda)^{-1} \Vert, \quad \la \in \C.	
	\end{equation}	 	 	
\item \label{Singular} The singular values of $A_m^{-1}$ read
\begin{equation}
\label{eq:Am.inv.sv}
s_{2k} = \frac{1}{\sqrt{m^2 + 2k}}, \quad  s_{2k+1} = \frac{1}{\sqrt{m^2 + 2k +2}}, \quad  k \in \N_0,
\end{equation} 
thus in particular
\begin{equation}\label{Am.inv.norm}
\Vert A_{m}^{-1} \Vert=\frac{1}{m}.
\end{equation}
\item There exists $c>0$ such that
	\begin{equation}\label{Carleman est 2}
	\log \Vert (A_{m}-\lambda)^{-1} \Vert \leq c (\Re\lambda)^2, \qquad \lambda \in \C.
	\end{equation}
\end{enumerate}
\end{proposition}
The proof of Proposition~\ref{prop:Airy.basic} is given in Sub-Section \ref{ssec:Airy.basic} and like Theorem~\ref{H.def} it relies mostly on more general operator-theoretic methods. The precise singular values in \eqref{eq:Am.inv.sv} are obtained by noticing that $A_m^*A_m$ is related via conjugation to the harmonic oscillator $\sL = -\partial_x^2 + x^2$ in $L^2(\R)$, see \eqref{Conjugate by U}.

In fact, the links between the resolvent of $A_m$ and those of the harmonic oscillator and the related conjugated one $\sL_\la = -(\partial_x+\la)^2 + x^2$ (see  Appendix~\ref{app:HO.conj}) are further exploited in Sub-section~\ref{ssec:Am.res} where an explicit representation of $(A_m-\la)^{-1}$ is found (see Lemma~\ref{lem:res.form}). Based on this and a formula for the resolvent of the harmonic oscillator in terms of parabolic cylinder functions (see Appendix~\ref{sec:app.ho}), we establish the asymptotic behaviour of the resolvent of $A_m$ as $\la \to \pm  \infty$. 
\begin{theorem}\label{thm:airy.res.a}
Let $m>0$ be fixed and $A_m$ be the imaginary Airy-Dirac operator defined in \eqref{eq:airy.dirac.def}. 
Then (with $\la \in \R$) 
	\begin{equation}\label{Airy.la.inf}
	\Vert (A_m-\la)^{-1} \Vert = 
	\frac{ \Gamma(\frac{m^2}{2}+1)}{2^{\frac{m^2}{2}}m} \frac{\exp(\la^2)}{|\la|^{m^2}}
		\left(1+\mathcal{O}\left(\frac{1}{|\la|^2}\right)\right), \quad \la \to \pm \infty.
	\end{equation}
\end{theorem}

In particular, \eqref{Airy.la.inf} shows that the exponential order in the upper resolvent bound \eqref{Carleman est 2} provided by abstract operator methods is sharp. 

The relation of the resolvent of $A_m$ to the conjugated oscillator $\sL_\la$ can be used also to analyse the resolvent norm of $A_m$ for varying $m$, in particular as $m \to 0 +$.

\begin{theorem}\label{thm:airy.res.m}
Let $A_m$, $m>0$, be the Airy-Dirac operators in \eqref{eq:airy.dirac.def}. 
Let $R >0$ be fixed and let $\la \in [-R,R]$. Then the following claims hold.
\begin{enumerate}[\upshape (i)]
\item If $m \to 0+$, then
\begin{equation}\label{eq:Am.small.m}
	\| (A_m - \la)^{-1} \| = \frac{\exp(\la^2)}{m} (1 + \BigO(m)), 
\end{equation}
where the remainder estimate is uniform for $\la \in [-R,R]$.
\item If $m \in [m_1,m_2]$ for some fixed $m_2>m_1>0$, then
\begin{equation}\label{Am.norm.m.1}
	\| (A_m - \la)^{-1} \| \approx 1.	
\end{equation}
\item If $m \to + \infty$, then
\begin{equation}\label{Am.norm.m.inf}
	\|(A_m - \la)^{-1}\| = \frac{1}{m}\left(1 + \BigO(m^{-1}) \right).
\end{equation}
\end{enumerate}
In particular,
\begin{equation}\label{Am.res.sum}
\|(A_m - \la)^{-1}\| \approx \frac{1}{m}, \quad m >0.
\end{equation}
\end{theorem}

The proof of Theorem~\ref{thm:airy.res.m} can be found in Sub-Section~\ref{ssec:Am.m0}.

Finally, for $m=0$, the spectrum of $A_0$ is very different from that of $A_m$, $m > 0$ (for the definitions of essential spectra, see \cite[Chap.~IX]{EE}). We recall that, if $T$ is a closed linear operator in a Hilbert space $\cH$ and $\la \in \C$ is an eigenvalue of $T$, then the \textit{geometric multiplicity} of $\la$ is the dimension of the eigenspace associated with $\la$, \ie~$\dim\Ker(T - \la)$. The \textit{algebraic multiplicity} of $\la$ is the dimension of the \textit{generalised} eigenspace corresponding to $\la$, \ie~$\dim\cM_\la$, with $\cM_\la := \overset{\infty}{\underset{n=1}{\cup}}\Ker((T-\la)^n)$ (see \eg~\cite[Def.~1.4.19]{Locker-2000-73}).
\begin{proposition}\label{prop:A0}
Let
\begin{equation}
\begin{aligned}
A_0 & = -i\Ntp \sigma_1 + i x I_2 
= \begin{pmatrix}
i x & - i \partial_x 
\\
- i \partial_x  & i x 
\end{pmatrix},
\\
\Dom(A_0) & = H^1(\R,\C^2) \cap \Dom(x I_2).
\end{aligned}
\end{equation}
Then $A_0$ is closed, $A_0^* = -i\Ntp \sigma_1 - i x I_2$, $\Dom(A_0^*) = \Dom(A_0)$ and 
\begin{equation}\label{A0.spe}
	\sigma(A_0)=  \spp(A_0)= \se{5}(A_{0}) = \C, \quad \se{k}(A_{0}) =\emptyset, \qquad k\in \{1,2,3,4\}.
\end{equation}
The spectral results in \eqref{A0.spe} also hold for $A_0^*$. Furthermore, each eigenvalue $\la\in \C$ has geometric multiplicity equal to $1$ and algebraic multiplicity equal to $\infty$ for both $A_0$ and $A_0^*$.	
\end{proposition}

The proof is given in Sub-Section~\ref{ssec:A0}.

\subsection{Resolvent estimates for general \texorpdfstring{$V$}{V}}

Our final goal is to establish resolvent bounds on the imaginary axis and adjacent regions for $H$ with a general potential $V$. More precisely, we assume that $V$ satisfies the following conditions (mostly asymptotic ones at $+ \infty$). 
\begin{asm-sec}
\label{asm:V.iR}
Let $V \in C^1(\R,\R) \cap C^2((x_0, \infty),\R)$ for some $x_0 \ge 0$. Suppose further that \eqref{eq:V.separ.cond} holds and that the following conditions are satisfied:
	\begin{enumerate}[\upshape (i)]
		\item\label{itm:V.unbd.rhs} $V$ is unbounded and eventually increasing on $(0,\infty)$:
		\begin{equation}\label{V.unbd}
		\lim_{x \to +\infty} V(x) = +\infty, \qquad V'(x) > 0, \qquad x > x_0;
		\end{equation}
		\item\label{itm:V.bd.lhs} $V$ is bounded above on $(-\infty,0)$: there exists $M \in \R$ such that
		\begin{equation}
			V(x) \le M, \qquad x < 0;
		\end{equation}
		\item\label{itm:nu.asm} $V$ has controlled derivatives: 
		\begin{equation}\label{V'.nu.asm}
		V'(x) \approx V(x) x^{-1}, \qquad |V''(x)| \ls V'(x) x^{-1}, \qquad x > x_0.
		\end{equation}
	\end{enumerate}
\end{asm-sec}

We remark that the conditions at $+\infty$ in Assumption~\ref{asm:V.iR} are satisfied by functions such as $V(x) = x^p$, $p > 0$, $x \gs 1$, or $V(x) = x^p \log x$, $p > 0$, $x \gs 1$, see also Example~\ref{ex:odd} below. Notice that by \eqref{V'.nu.asm} 
\begin{equation}
\frac{V'(x)}{V^2(x)} \approx \frac{1}{x V(x)} = o(1), \quad x \to + \infty,
\end{equation}
thus the condition \eqref{eq:V.separ.cond} is satisfied for $x>0$.

As a consequence of Gronwall inequality and the condition \eqref{V'.nu.asm}, the functions $V$ and $V'$ cannot grow faster than a polynomial. In particular, there exists $n_0 \in \N$ such that 
\begin{equation}\label{n0.def}
\frac{V'(x)}{	x^{2 n_0}} = o(1), \qquad x \to +\infty.
\end{equation}
Assuming that $b>0$ is sufficiently large, we denote by $x_b \in (0,\infty)$ the \textit{unique} solution (see Assumptions~\ref{asm:V.iR}~\ref{itm:V.unbd.rhs},~\ref{itm:V.bd.lhs}) to the equation
\begin{equation}\label{eq:xb.def}
	V(x_b)=b
\end{equation}
(sometimes called a turning point of $V$). Clearly, by Assumption~\ref{asm:V.iR}~\ref{itm:V.unbd.rhs}, we have $x_b \to +\infty$ as $b \to +\infty$. With $b$ positive, we define the curves
\begin{equation}\label{eq:la.def}
(0,\infty) \ni b \mapsto \la_b := a(b) + i b \in \C.
\end{equation}
\begin{theorem}\label{thm:iR}
Let $H$ be the Dirac operator in \eqref{H.def} with $m>0$ and $V$ satisfying Assumption~\ref{asm:V.iR}. Let $x_b$ and $\la_b$ be defined as in \eqref{eq:xb.def} and \eqref{eq:la.def}, respectively. Suppose furthermore that $\la_b$ satisfies
\begin{equation}\label{mu.b.def}
\mu_b := \frac{a(b)}{V'(x_b)^{\frac12}} = \BigO(1), \quad b \to +\infty.
\end{equation}
Then as $ b \to +\infty$
\begin{equation}\label{eq:resnorm.iR}
\|(H - \la_b)^{-1} \| = \| (A_{m_b} - \mu_b)^{-1} \| (V'(x_b))^{-\frac12} 
\left(1 + \BigO(x_b^{-\frac 1 2} + b^{-\frac12})\right),
\end{equation}
where $A_{m_b}$ denotes the Airy-Dirac operator, defined as in \eqref{eq:airy.dirac.def}, with mass 
\begin{equation}\label{mb.def}
m_b := \frac{m} {V'(x_b)^{\frac12}}.
\end{equation}
\end{theorem}
The proof of Theorem~\ref{thm:iR} is given in Section~\ref{sec:resnorm.est.iR}. 

Combining \eqref{eq:resnorm.iR} and \eqref{eq:Am.small.m}, we can further simplify the main term in \eqref{eq:resnorm.iR} and also obtain the asymptotic shape of the pseudospectral curves of $H$.

\begin{corollary}\label{cor:iR}
Let the assumptions of Theorem~\ref{thm:iR} be satisfied.
\begin{enumerate}[\upshape (i)]
\item We have 
\begin{equation}\label{eq:resnorm.iR.exact}
\|(H - i b)^{-1} \| = \frac 1m 
\left(1 + \BigO(x_b^{-\frac 1 2} + b^{-\frac12})\right),
 \quad b \to \infty.
\end{equation}
\item Suppose in addition that $V'(x) \to +\infty$ as $x \to +\infty$. Then as $b \to +\infty$
	\begin{equation}\label{eq:resnorm.iR.expl}
		\|(H - \la_b)^{-1} \| = \frac1m \exp \left(\frac{a(b)^2}{V'(x_b)} \right) \left(1 + \BigO(x_b^{-\frac 12}) + \BigO(V'(x_b)^{-\frac12}) \right).
	\end{equation}
	Moreover, for the curve $b \mapsto \la_b = a(b) + i b$, where
	\begin{equation}
		a(b)^2 = V'(x_b) \log \frac m \eps,
	\end{equation}
	with a fixed $\eps>0$, we have as $b \to +\infty$
	\begin{equation}\label{H.res.ps}
		\|(H-\la_b)^{-1}\| = \frac1 \eps \left(1 + \BigO(x_b^{-\frac 12}) + \BigO(V'(x_b)^{-\frac12}) \right).	
	\end{equation}
		\item Suppose in addition that  $V'(x) \approx 1$ as $x \to +\infty$. Then 
	\begin{equation}\label{H.lb.V'.bdd}
	\|(H - \la_b)^{-1} \| \approx 1, \quad b \to + \infty.
	\end{equation}
	(By \eqref{mu.b.def}, we have $a(b) = \BigO(1)$ as $b \to + \infty$ in this case.)
	\item Suppose in addition that  $V'(x) \to 0$ as $x \to +\infty$. Then 
\begin{equation}\label{H.lb.V'.dec}
\|(H - \la_b)^{-1} \| = \frac 1m   \left(1 + \BigO( b^{-\frac12}) +\BigO(V'(x_{b})^{\frac12}) \right), \quad b \to +\infty.
\end{equation}
	(By \eqref{mu.b.def}, we have $a(b) = o(1)$ as $b \to + \infty$ in this case.)
\end{enumerate}
\end{corollary}
\begin{proof}
\begin{enumerate}[\upshape (i), wide]
	\item An immediate consequence of \eqref{eq:resnorm.iR} and \eqref{Am.inv.norm}.
	\item We have $m_b \to 0+$ as $b \to + \infty$, thus to show \eqref{eq:resnorm.iR.expl}, we employ \eqref{eq:resnorm.iR} and \eqref{eq:Am.small.m}; \eqref{H.res.ps} then follows immediately.
	\item As $m_b \approx 1$ and $|\mu_b| \ls 1$ as $b \to +\infty$ in this case, \eqref{H.lb.V'.bdd} follows by \eqref{eq:resnorm.iR} and \eqref{Am.norm.m.1}.
	 \item In this case $m_b \to +\infty$, so \eqref{H.lb.V'.dec} follows by \eqref{Am.norm.m.inf} and \eqref{eq:resnorm.iR}.
	  \qedhere
\end{enumerate}
\end{proof}

\begin{remark}\label{rmk:res}
\begin{enumerate}[\upshape (i), wide]
\item\label{rem odd} If $V$ satisfies Assumption \ref{asm:V.iR} and it is in addition an odd function (\ie~$V(-x) = -V(x)$, $x \in \R$), then Theorem~\ref{thm:iR} and Corollary~\ref{cor:iR} also hold for $\lambda_{b}=a(b)+ib$ as $b \to -\infty$. Indeed, this easily follows from the identity
\begin{equation}
\Vert (H-\lambda_{b})^{-1} \Vert = \Vert (H^{*}-\overline{\lambda_{b}})^{-1} \Vert,
\end{equation}
together with $H^* = -i\Ntp \sigma_1 + m \sigma_3 - i V I_2$ (see~\eqref{H*.def}) and $\overline{\lambda_{b}}=a - ib$. 
\item \label{rem opt} In \cite{Krejcirik-2022-282}, the pseudospectrum of $H$ was found by applying a non-semi-classical pseudomode construction. By \eqref{eq:resnorm.iR.expl} and Assumption~\ref{asm:V.iR}~\ref{itm:nu.asm}, for all $\varepsilon>0$ and $\varepsilon'\in (0,1)$, there exists $b_{0}>0$ such that, for all $\lambda_{b}$ in the region determined by
\begin{equation} 
\Lambda=\left\{ a+ib\in \C: \, b>b_{0}, \, a^2\ls b x_{b}^{-1} \log\frac{m(1-\varepsilon')}{\varepsilon} \right\},
\end{equation}
we have $\Vert (H-\lambda_{b})^{-1} \Vert\leq \frac{1}{\varepsilon}$. This shows that the condition \cite[Eq.~(4.14)]{Krejcirik-2022-282}, with $\nu=-1$, is sharp.
\end{enumerate}
\end{remark}

\begin{example}\label{ex:odd}
\begin{enumerate}[\upshape (i), wide]
\item Polynomial potential:
\begin{align}
V(x)=\left\{\begin{aligned}
& x\vert x \vert^{\gamma-1},\qquad &&\text{if }|x|\geq 1,\\
& V_{0}(x), &&\text{if } -1< x< 1,\\
\end{aligned}
\right.
\end{align}
where $\gamma>0$ and  $V_{0}:(-1,1)\to \R$ is any odd function such that $V\in C^{1}(\R, \R)$. It is not hard to check that $V(x)$ satisfies Assumption~\ref{asm:V.iR}. For $b>0$ and sufficiently large, the turning point $x_{b}$ is determined by $x_{b}=b^{\frac{1}{\gamma}}$. Applying Corollary \ref{cor:iR}, the asymptotic behavior of the resolvent norm on the curve $\lambda_{b}=a(b)+ib\in \C$ as $b\to+\infty$ is given by
\begin{align}
&\|(H-\la_b)^{-1}\| \nonumber \\
=& \left\{
\begin{aligned}
& \frac1 \eps \left(1 + \BigO(b^{-\frac{1}{2\gamma}})+\BigO\left(b^{\frac{1-\gamma}{2\gamma}}\right) \right), &&\text{ for }a(b)^2 = \gamma b^{\frac{\gamma-1}{\gamma}} \log\frac{m}{\varepsilon} \text{ if }\gamma>1,\\
&\mathcal{O}(1), &&\text{ for }\vert a(b) \vert=\mathcal{O}(1) \text{ if }\gamma=1,\\
&\frac{1}{m}\left(1+\BigO\left(b^{-\frac{1}{2}}\right)+\BigO\left(b^{\frac{\gamma-1}{2\gamma}}\right) \right), && \text{ for } \vert a(b)\vert=o(1) \text{ if }0<\gamma<1.
\end{aligned}
\right. \label{Eq norm res}
\end{align}
The reader is invited to compare the above results with \cite[Example~5]{Krejcirik-2022-282}, where the resolvent norm blows up in the regions: $\{\alpha + i\beta \in \C : \, \vert\alpha\vert \gs \beta^{\frac{1}{2}\frac{\gamma-1}{\gamma}+\eta}\}$, for any $\eta > 0$ and $\gamma \ge 1$; and $\{\alpha + i\beta \in \C : \, \vert\alpha\vert \geq m\}$, for $0<\gamma<1$. Furthermore, since $V(x)$ is odd, it follows that  \eqref{Eq norm res} holds for $|b| \to +\infty$ replacing $b$ with $\vert b \vert$ in the r.h.s of the equality (see Remark~\ref{rmk:res}~\ref{rem odd}). For example, when $V(x)=x^3$, \eqref{Eq norm res} yields 
\begin{equation}
	\|(H-\la_b)^{-1}\| = \frac1 \eps \left(1 + \BigO(\vert b \vert^{-\frac{1}{6}})\right),\qquad \vert b \vert\to+\infty,
\end{equation}
for the curves $\la_b$ s.t. $a(b)^2 = 3 \vert b \vert^{\frac{2}{3}} \log\left(\frac{m}{\varepsilon}\right)$. This case is illustrated in Fig.~\ref{fig:main}\subref{fig:example x3}.
\item Logarithmic-linear potential:
\begin{align}
V(x)=\left\{\begin{aligned}
& x\log(\vert x \vert),\qquad &&\text{if }|x|\geq 1,\\
& V_{0}(x), &&\text{if } -1< x< 1,\\
\end{aligned}
\right.
\end{align}
where $V_{0}:(-1,1)\to \R$ is any odd function such that $V\in C^{1}(\R, \R)$. For $b>0$ and sufficiently large, the turning point $x_{b}$ solving $x\ln(x) =b$ has the form
\begin{equation} 
x_{b}=\frac{b}{W(b)},
\end{equation}
where $W$ is the Lambert function which solves the equation $W e^{W}=b$, see \cite[Sec.~4.13]{DLMF}. Using the asymptotic expansion of the Lambert function at infinity (cf. \cite[Eq.~4.13.10]{DLMF}), we have as $b\to+\infty$
\begin{align}
V'(x_{b})=\log(x_{b})+1=W(b)+1=\log b\left(1+\mathcal{O}\left(\frac{\log \log b}{\log b}\right) \right).
\end{align}
Then, applying \eqref{H.res.ps}, the level sets of the resolvent norm are given by
\begin{equation}
\|(H-\la_b)^{-1}\| = \frac1 \eps \left(1  + \BigO((\log b)^{-\frac12}) \right),\qquad b\to+\infty, 
\end{equation}
on the curves 
\begin{equation} \lambda_{b}=a(b)+ib \text{ where } a(b)^2 = (W(b)+1) \log\frac{m}{\varepsilon}.
\end{equation}
\end{enumerate}
\end{example}
\subsubsection{The case of $V$ bounded from below}

The claim of Theorem~\ref{thm:iR} can be extended for potentials with 
\begin{equation}
\lim_{x \to \pm \infty } V(x) = + \infty	
\end{equation}
which satisfy analogous conditions to those in Assumption~\ref{asm:V.iR} also at $-\infty$. %
\begin{asm-sec}	\label{asm:V.iR.ext}
Let $V \in C^1(\R,\R) \cap C^2((-\infty, -x_0) \cup (x_0, \infty),\R)$ for some $x_0 > 1$. Suppose that the following conditions are satisfied:
	\begin{enumerate}[\upshape (i)]
		\item\label{itm:V.unbd.rhs.ext} $V$ is unbounded and eventually increasing on $(0,\infty)$/decreasing on $(-\infty,0)$:
		\begin{equation}
			\begin{aligned}
				&\underset{x \to +\infty}{\lim} V(x) = +\infty, \qquad  V'(x) > 0, \qquad x > x_0,\\
				&\underset{x \to -\infty}{\lim} V(x) = +\infty, \qquad  V'(x) < 0, \qquad x < -x_0;
			\end{aligned}
		\end{equation}
		\item\label{itm:nu.asm.ext} $V$ has controlled derivatives: 
		\begin{equation}
			\begin{aligned}
				|V'(x)| \approx V(x) |x|^{-1}, \qquad |V''(x)| \ls |V'(x)| |x|^{-1}, \qquad |x| > x_0.
			\end{aligned}
		\end{equation}
	\end{enumerate}
\end{asm-sec}
Similarly to our preparation for Theorem~\ref{thm:iR}, for sufficiently large $b>0$, we define the turning points $x_{b,\pm}$ by
\begin{equation}
	\label{eq:xb.def.ext}
	V(x_{b, \pm}) = b, \quad \text{with} \quad x_{b,+} > x_0, \quad \quad x_{b,-} < -x_0.
\end{equation}

In the following, we use the notation $\max\{a_\pm\} := \max\{a_+, a_-\}$.
\begin{theorem}
\label{thm:iR.ext}
Let $H$ be the Dirac operator in \eqref{H.def} with $m>0$ and $V$ satisfying Assumption~\ref{asm:V.iR.ext}. Let $x_{b,\pm}$ and $\la_b=a(b)+i b$ be defined as in \eqref{eq:xb.def.ext} and \eqref{eq:la.def}, respectively. Suppose furthermore that $\la_b$ satisfies
	\begin{equation}
		\label{eq:lab.gral.curves.ext}
		\mu_{b,\pm} := \frac{a(b)}{|V'(x_{b,\pm})|^{\frac12}} = \BigO(1), \quad b \to +\infty.
	\end{equation}
	Then as $b \to +\infty$
	\begin{equation}
		\label{eq:resnorm.iR.ext}
		\begin{aligned}
			\|(H - \la_b)^{-1} \| &= \max \{ \| (A_{m_{b,\pm}} - \mu_{b,\pm})^{-1} \| |V'(x_{b, \pm})|^{-\frac12}\} \times
			\\
			&\qquad \qquad \qquad \qquad \qquad \qquad \left(
			1 + \BigO(\max\{x_{b,\pm}^{-\frac 12}\} + b^{-\frac12})
			\right),
		\end{aligned}
	\end{equation}
	where $A_{m_{b,\pm}}$ denotes the Airy-Dirac operator, defined as in \eqref{eq:airy.dirac.def}, with mass 
	\begin{equation}
	m_{b,\pm}:= \frac{m}{|V'(x_{b, \pm})|^{\frac12}}.	 
	\end{equation}
	In particular, as $b \to + \infty$, 
	\begin{equation}
		\label{eq:resnorm.iR.ext.exact}
			\|(H - i b)^{-1} \| = \frac1m \left(1 + \BigO(\max\{x_{b,\pm}^{-\frac 12}\} + b^{-\frac12}) 
			\right).
	\end{equation}
	
\end{theorem}
The proof of \eqref{eq:resnorm.iR.ext} relies on introducing a suitable partition of unity to separately estimate (via Theorem~\ref{thm:iR}) the asymptotic behaviour of $H$ on $(0,\infty)$ and $(-\infty,0)$. Details of the procedure in a similar result for Schr\"odinger operators with complex potentials can be found in \cite[Sub-sec.~5.3]{ArSi-2023-284}. The necessary modifications are straightforward and we leave them to the reader. The special case \eqref{eq:resnorm.iR.ext.exact} then follows by \eqref{Am.inv.norm}. Like for Theorem~\ref{thm:iR}, one can also formulate an analogue of the other statements in Corollary~\ref{cor:iR}.

\begin{example}\label{ex:even}
For application of Theorem~\ref{thm:iR.ext}, let us consider the following potential $V(x)$ which satisfies the Assumption~\ref{asm:V.iR.ext}:
 \begin{align}
V(x)=\left\{\begin{aligned}
& \vert x \vert^{\gamma},\qquad &&\text{if }|x|\geq 1,\\
& V_{0}(x), &&\text{if } -1< x< 1,\\
\end{aligned}
\right.
\end{align}
where $\gamma>0$ and  $V_{0}:(-1,1)\to \R$ is any function such that $V\in C^{1}(\R,\R)$. Applying Theorem \ref{thm:iR.ext}, \eqref{Eq norm res} is obtained. For example, when $V(x)=x^2$, we have
\begin{equation}
\|(H-\la_b)^{-1}\| = \frac1 \eps \left(1 + \BigO( b^{-\frac{1}{4}})\right),\qquad \text{ as } b \to+\infty, 
\end{equation}
t
where  $a(b)^2 = 2 b^{\frac{1}{2}} \log\frac{m}{\varepsilon}$. This case is illustrated in Fig.~\ref{fig:main}\subref{fig:example x2}.
\end{example}

Finally, under Assumption~\ref{asm:V.iR.ext} and faster than a linear growth of $V$, our resolvent analysis in Theorem~\ref{thm:iR.ext} leads to the following result on the spectrum and spectral projections of $H$. The proof is based on  \cite[Cor.~XI.9.31]{DS2}.
\begin{corollary}
\label{cor:V.bd.below}
Let $H$ be the Dirac operator in \eqref{H.def} with $m>0$ and $V$ satisfying Assumption~\ref{asm:V.iR.ext}. Suppose in addition that
	\begin{equation}
		V(x) \gs \langle x \rangle^\gamma, \quad |x|>x_0
	\end{equation}
	with some $x_0 \geq 0$ and $\gamma>1$. Then the spectrum of $H$ comprises infinitely many (distinct) discrete eigenvalues of finite algebraic multiplicities and the vector space spanned by generalised eigenfunctions is complete in $L^2(\R,\C^2)$.
\end{corollary}
\begin{proof}
We verify the assumptions of \cite[Cor.~XI.9.31]{DS2}, from which the claims follow. 

Since $V(x) \geq v_0 \langle x \rangle^\gamma$, $x \in \R$, with some $v_0>0$ and $\gamma >1$, Theorem~\ref{thm:H.def}~\ref{thm:H.def.iv)} yields that $H^{-1} \in \cS_{p_\gamma,\infty}$ where $p_\gamma \in (1,2)$. Fix $\varepsilon>0$ such that $p_{\gamma}+\varepsilon<2$. Since $\cS_{p_\gamma,\infty}\subset \cS_{p_\gamma+\varepsilon}$, it follows that $H^{-1}\in \cS_{p_{\gamma}+\varepsilon}$.

For $u \in \textup{Dom}(H)$, by straightforward calculation, we have
\begin{equation}
\Im \langle H u ,u\rangle = \int_{\R} V(x) |u(x)|^2 \, \dd x \geq v_0  \|u\|^2,	
\end{equation}
thus $\Num(H) \subset \{z \in \C: \Im z \geq v_0\}$. We consider four rays in $\C$ 
\begin{equation}
\gamma_k= \left\{ r e^{i \frac\pi 2 k} \, : \, r>0 \right\}, \quad k = 0,1,2,3,
\end{equation}
for which the angle between any two adjacent rays equals $\frac{\pi}{2} < \frac{\pi}{p_{\gamma}+\varepsilon}$. Finally,   
\begin{equation}
\|(H- \lambda)^{-1}\| = \BigO(1), \quad \lambda \to  \infty \text{ on }\gamma_{k}, \quad k = 0,1,2,3;
\end{equation}
the cases $k=0,2,3$ follow from the standard estimate using the distance of $\la$ to $\Num(H)$ and the case $k=1$ is covered by Theorem~\ref{thm:iR.ext}. 
\end{proof}

\section{\texorpdfstring{$L^2$}{L2}-realization with non-empty resolvent set -- proofs}
\label{sec:H.def.proofs}

\subsection{Schur complements}
\label{ssec:Sch.compl}

Recall that we assume that $m>0$ and $V \in C^1(\R,\R)$.
The Schur complements of $H$ are the following operator families in $L^2(\R)$
\begin{equation}\label{S12.def}
	\begin{aligned}
		S_1(\la) &= -\partial_x \frac{1}{m - i V +\la} \partial_x + m + i V  - \la, &\la &\notin \ov{ \{ i V(x) - m \, : \,  x \in \R \}},
		\\
		S_2(\la) &= \partial_x \frac{1}{m + i V - \la} \partial_x - m + i V - \la, &\la &\notin \ov{ \{ i V(x) + m \, : \,  x \in \R \}}.
	\end{aligned}
\end{equation}
In this Sub-section, we find realizations of $S_j$, $j=1,2$, in $L^2(\R)$ with non-empty resolvent set. To this end, we employ the generalized coercivity of Almog-Helffer of the associated sesquilinear forms (see Appendix~\ref{app:AH} and Lemma \ref{lem:s.coer} \ref{itm:s.coer.ii}); this coercivity is obtained under the assumption
\begin{equation}\label{V.nabla.form}
	\forall \delta > 0, \ \exists C_\delta > 0 
	\,: \,  |V'(x)|^2 \leq \delta V(x)^8 + C_\delta, \quad x \in \R.
\end{equation}
\begin{remark}
	\label{rmk:3.2.wkr.th.2.2}
	We note that \eqref{V.nabla.form} is a weaker condition than \eqref{eq:V.separ.cond}. To see this, assume that \eqref{eq:V.separ.cond} holds. Then
	\begin{equation}
		|V'(x)|^2 \leq 2 \eps^2 V(x)^4 + 2 M_\eps^2, \quad x \in \R,
	\end{equation}
	for some $\eps \in (0,1), \; M_\eps > 0$. Let $\delta > 0$ be arbitrary. Then, applying Young's inequality for products with $p=q=2$ to $(2 \eps^2 V(x)^4 \hat{\delta})\hat{\delta}^{-1}$ with $\hat{\delta} = \sqrt{\delta/2}/\eps^2$, we obtain
	\begin{equation}
		2 \eps^2 V(x)^4 \leq \frac{4 \eps^4 V(x)^8 \hat{\delta}^2}{2} + \frac{\hat{\delta}^{-2}}{2} = \delta V(x)^8 + \frac{\eps^4}{\delta}.
	\end{equation}
	Hence \eqref{V.nabla.form} follows with $C_\delta = \eps^4/\delta + 2 M_\eps^2$.
\end{remark}

We show that a suitable form-domain is the space
\begin{equation}\label{cV.def}
	\cV := \Big\{ f \in L^2(\R): \left(m^2+V^2\right)^{-\frac14} f' \in L^2(\R), \, \vert V \vert^{\frac{1}{2}}f\in L^2(\R)\Big\}
\end{equation}
equipped with the scalar product $\langle \cdot, \cdot \rangle_{\mathcal{V}}$
which induces the norm
\begin{equation}\label{cV.norm.def}
	\|f \|_{\cV}^2 =  \Vert (m^2+V^2)^{-\frac14} f' \Vert^2+ \Vert (m^2+V^2)^{\frac{1}{4}} f \Vert^2 + \Vert f \Vert^2, \quad f \in \cV.
\end{equation}
\begin{lemma}\label{lem:cV}
Let $m>0$ and $V \in C^1(\R,\R)$. Let $(\cV,\langle \cdot,\cdot \rangle_\cV)$ be as in \eqref{cV.def} and \eqref{cV.norm.def}. Then $(\cV,\langle \cdot,\cdot \rangle_\cV)$ is a Hilbert space and $\CcR$ is  dense in it.
\end{lemma}
\begin{proof}
The first claim follows by standard arguments. For the density of $\CcR$, the usual cut-off and mollification like e.g.~in~the proof of~\cite[Thm.~8.2.1]{Davies-1995} can be used.
\end{proof}
Let $\frs$ be a sesquilinear form on $\cV \times \cV$ defined via its quadratic form as
\begin{equation}\label{s.form.def}
\frs[f] := \int_\R \frac{1}{m-iV(x)} |f'(x)|^2 \, \dd x + \int_\R (m+ iV(x))|f(x)|^2 \, \dd x,
\quad 
\Dom(\frs) := \cV.
\end{equation}
\begin{lemma}\label{lem:s.coer}
Let $m>0$, $V \in C^1(\R,\R)$ and let $\frs$ be as in \eqref{s.form.def}. 
\begin{enumerate}[\upshape (i)]
\item The form $\frs$ is bounded on $\cV$, in particular,
	\begin{equation}
	|\frs[f]| \leq \|f\|_\cV^2, \qquad f \in \cV.
	\end{equation}
\item\label{itm:s.coer.ii} Suppose that $V$ satisfies \eqref{V.nabla.form}. Then there exists $c_\frs>0$ such that for all sufficiently large $t>0$
\begin{equation}\label{s.AH.coer}
\begin{aligned}
|(\frs+t)[f]| + |(\frs+t)(f,\Phi f)| &\geq c_\frs \|f\|_\cV^2, 
\\		
|(\frs+t)[f]| + |(\frs+t)(\Phi f, f)| &\geq c_\frs \|f\|_\cV^2, \quad f \in \cV,
\end{aligned}
\end{equation}
where 
\begin{equation}
\Phi:= \frac{V}{(m^2+V^2)^\frac12}.
\end{equation}
The multiplication operator by $\Phi$ is bounded on $\cV$, thus the form $\frs + t$ is coercive on $\cV$ in the generalized sense of Almog-Helffer.
\end{enumerate}
\end{lemma}
\begin{proof}
\begin{enumerate}[\upshape (i), wide]
\item Follows immediately from \eqref{s.form.def} and \eqref{cV.norm.def}.
\item First we find that
\begin{equation}
\Phi' =  \frac{V'}{(m^2+V^2)^{\frac{1}{2}}} -\frac{V'V^2}{(m^2+V^2)^\frac 32}
=\frac{m^2V'}{(m^2+V^2)^{\frac{3}{2}}}.
\end{equation}
Thus, for any $f \in \cV$, using the assumption \eqref{V.nabla.form}, we arrive at
\begin{equation}
\begin{aligned}
\|\Phi f\|_{\cV}^2 & = \left\Vert \frac{(\Phi f)'}{(m^2+V^2)^{\frac{1}{4}}} \right\Vert^2 
+ \Vert (m^2+V^2)^{\frac{1}{4}} \Phi f \Vert^2 + \Vert \Phi f \Vert^2
\\
& \leq 
2 \left\Vert \frac{f'}{(m^2+V^2)^{\frac{1}{4}}} \right\Vert^2
+ 2 m^4 \left\Vert \frac{V' f }{(m^2+V^2)^{\frac{7}{4}}} \right\Vert^2 
+ \Vert (m^2+V^2)^{\frac{1}{4}} f \Vert^2 
\\ & \quad 
+ \Vert f \Vert^2
\\
& \leq 
2 \|f\|_\cV^2 + 2m^4\delta \Vert (m^2+V^2)^{\frac{1}{4}} f \Vert^2 + \frac{2 C_\delta}{m^3} \|f\|^2,
\end{aligned}
\end{equation}
and so the multiplication by $\Phi$ is bounded on $\cV$.

Next, for all $f \in \cV$,
\begin{equation}
\begin{aligned}
& \Re \left((\frs+t)[f]\right) + \Im \left((\frs+t)(f, \Phi f) \right) 
\\
&\ \geq 
\int_\R \left( \frac{m}{m^2+V(x)^2} + \frac{V(x) \Phi(x)}{m^2+V(x)^2} \right)|f'(x)|^2 \; \dd x
+  \int_\R V(x) \Phi(x) |f(x)|^2 \; \dd x\\
& \quad + (m+t)\|f\|^2 + \Im \int_\R \frac{f'(x) \Phi'(x) \ov{f(x)}}{m-iV(x)} \; \dd x.
\end{aligned}
\end{equation}
By Cauchy-Schwarz and Young inequalities, we obtain that for any $\eps>0$ and $\delta > 0$
\begin{equation}
\begin{aligned}
& \left|\Im \int_\R \frac{f'(x) \Phi'(x) \ov{f(x)}}{m-iV(x)} \; \dd x \right| 
\\
& \quad \leq 
\int_\R \frac{|f'(x)|}{(m^2+V(x)^2)^\frac14}  \frac{|\Phi'(x)||f(x)|}{(m^2+V(x)^2)^\frac14} \; \dd x
\\
& \quad \leq
\eps \int_\R \frac{|f'(x)|^2}{(m^2+V(x)^2)^\frac 12} \; \dd x + 
\frac1{4\eps} \int_\R \frac{|\Phi'(x)|^2|f(x)|^2}{(m^2+V(x)^2)^\frac 12} \; \dd x
\\
& \quad \leq
\eps \int_\R \frac{|f'(x)|^2}{(m^2+V(x)^2)^\frac 12} \; \dd x + 
\frac{m^4 \delta}{4\eps} \int_\R \frac{V(x)^8 |f(x)|^2}{(m^2+V(x)^2)^\frac 72}  \; \dd x 
+\frac{C_\delta}{4\eps m^3}\|f\|^2. 
\end{aligned}
\end{equation}
Thus 
\begin{equation}
\begin{aligned}
& \Re \left((\frs+t)[f]\right) + \Im \left((\frs+t)(f, \Phi f) \right) 
\\
&\quad \geq 
\int_\R \left( \frac{m}{m^2+V(x)^2} + \frac{V(x)^2}{(m^2+V(x)^2)^\frac32} - \frac{\eps}{(m^2+V(x)^2)^\frac 12}\right)|f'(x)|^2 \; \dd x
\\
&\qquad  +  \int_\R \left( \frac{V(x)^2}{(m^2+V(x)^2)^\frac12} - \frac{m^4 \delta}{4\eps} \frac{V(x)^8 }{(m^2+V(x)^2)^\frac 72}  \right)|f(x)|^2 \; \dd x
\\
&\qquad
+ \left(m+t - \frac{C_\delta}{4\eps m^3}\right)\|f\|^2.
\end{aligned}
\end{equation}
Notice that
\begin{equation}
\frac{m}{m^2+V^2} + \frac{V^2}{(m^2+V^2)^{\frac{3}{2}}} -\frac{\eps}{(m^2+V^2)^\frac 12} 
\geq \frac{1-\eps}{(m^2+V^2)^{\frac{1}{2}}},
\end{equation}
and
\begin{equation}
\begin{aligned}
& \frac{V^2}{(m^2+V^2)^\frac12} - \frac{m^4 \delta}{4\eps} \frac{V^8 }{(m^2+V^2)^\frac 72} 
\\
& \quad  \geq \left(1 - \frac{m^4 \delta}{4\eps} \right) \frac{V^2}{(m^2+V^2)^\frac12}
\geq \left(1 - \frac{m^4 \delta}{4\eps} \right) \left( (m^2+V^2)^\frac12 - m \right).
\end{aligned}
\end{equation}
By assumption we can select $\delta > 0$ such that $m^4 \delta/4 <1$, thus we can select $\eps \in (0,1)$ and afterwards $t>0$ so that
\begin{equation}
 \Re \left((\frs+t)[f]\right) + \Im \left((\frs+t)(f, \Phi f) \right)  \geq c_\frs \|f\|_\cV^2, \quad f \in \cV,
\end{equation}
with $c_\frs >0$. The proof of the second inequality in \eqref{s.AH.coer} is fully analogous.
\qedhere
\end{enumerate}
\end{proof}

\begin{remark}
If $V \geq 0$  (or $V \leq 0$), then $\frs$ is coercive on $\cV$, i.e.~there exists $c_\frs >0$ such that 
\begin{equation}\label{s.V.pos.coer}
	|\frs[f]| \geq c_\frs \|f\|_\cV^2, \quad f \in \cV.
\end{equation}
Indeed, for all $f \in \cV$,
\begin{equation}
\begin{aligned}
&	\Re \left(e^{-i\frac \pi4}\frs[f]\right) 
\\
& \quad = 
		\Re \left( e^{-i\frac \pi4} \int_\R \left(\frac{m}{m^2+ V(x)^2} |f'(x)|^2 + m|f(x)|^2\right) \, \dd x \right.
		\\ & \qquad \qquad
		\left. +
		e^{i\frac \pi4} \int_\R \left(\frac{V(x)}{m^2+ V(x)^2} |f'(x)|^2 + V|f(x)|^2\right) \, \dd x
		\right)
		\\
		& \quad = \frac{1}{\sqrt 2}
		\int_\R \left(\frac{m+V(x)}{m^2+ V(x)^2} |f'(x)|^2 + (m+V(x))|f(x)|^2\right) \, \dd x
		\\
		& \quad \geq c_{\frs} \int_\R \left(\frac{|f'(x)|^2}{(m^2+ V(x)^2)^\frac 12}  + \left((m^2+V(x)^2)^\frac12+1 \right)|f(x)|^2\right) \, \dd x,
	\end{aligned}
\end{equation}
where $c_{\frs} = c_{\frs}(m) > 0$. We do not discuss this special case further in the following. 

\end{remark}

\begin{remark}
	We note that the statement in Lemma~\ref{lem:s.coer}~\ref{itm:s.coer.ii} is still true if we replace condition~\eqref{V.nabla.form} with the weaker assumption
	\begin{equation}
		\exists C_\nabla \in [0, \tfrac{4}{m^4}), \ \exists C>0 
		\,: \,  |V'(x)|^2 \leq C_\nabla V^8(x) + C, \quad x \in \R.
	\end{equation}
\end{remark}

For $j=1,2$, let $\frs_j$ be the sesquilinear forms on $\cV \times \cV$ related to the Schur complements $S_j(\la)$ in \eqref{S12.def} defined via 
\begin{equation}\label{s12.form.def}
	\begin{aligned}
	\frs_1[f] &:= \int_\R \frac{1}{m-iV(x)+\la} |f'(x)|^2 \, \dd x + \int_\R (m+ iV(x)-\la)|f(x)|^2 \, \dd x,
		\\
	-\frs_2[f] &:= \int_\R \frac{1}{m+iV(x)-\la} |f'(x)|^2 \, \dd x + \int_\R (m- iV(x)+\la)|f(x)|^2 \, \dd x,
	\\
		\Dom(\frs_j) & := \cV, \quad j =1,2;
	\end{aligned}
\end{equation}
here $\frs_1$ is defined for 
\begin{equation}\label{s1.la.restr}
\la \notin \ov{ \{ i V(x) - m \, : \,  x \in \R \}}
\end{equation}
and $\frs_2$ for
\begin{equation}\label{s2.la.restr}
\la \notin \ov{ \{ i V(x) + m \, : \,  x \in \R \}}.
\end{equation}

Lemma \ref{lem:s.coer} and straightforward substitutions can be employed to establish a generalized coercivity of $\frs_j$.
\begin{lemma}\label{lem:s12.coer}
Let $m>0$ and let $V \in C^1(\R,\R)$ satisfy \eqref{V.nabla.form}. Let further $j=1,2$ and let $\frs_j$ be as in \eqref{s12.form.def}, \eqref{s1.la.restr} and \eqref{s2.la.restr}.  
Then $\frs_j$ are bounded on $\cV$ and moreover, 
\begin{enumerate}[\upshape (i)]
\item If $\Re \la> -m$, then there exists $c_{\frs_1}>0$ such that for all sufficiently large $t>0$, 
	\begin{equation}\label{s1.AH.coer}
		\begin{aligned}
			|(\frs_1+t)[f]| + |(\frs_1+t)(f,\Phi_1 f)| &\geq c_{\frs_1} \|f\|_\cV^2, 
			\\		
			|(\frs_1+t)[f]| + |(\frs_1+t)(\Phi_1 f, f)| &\geq c_{\frs_1} \|f\|_\cV^2, \quad f \in \cV,
		\end{aligned}
	\end{equation}
	where 
	\begin{equation}
		\Phi_1:= \frac{V-\Im \la}{((m+\Re \la)^2+(V- \Im \la)^2)^\frac12}.
	\end{equation}
\item If $\Re \la < m$, then there exists $c_{\frs_2}>0$ such that for all sufficiently large $t>0$, 
	\begin{equation}\label{s2.AH.coer}
		\begin{aligned}
			|(-\frs_2+t)[f]| + |(-\frs_2+t)(f,\Phi_2 f)| &\geq c_{\frs_2} \|f\|_\cV^2, 
			\\		
			|(-\frs_2+t)[f]| + |(-\frs_2+t)(\Phi_2 f, f)| &\geq c_{\frs_2} \|f\|_\cV^2, \quad f \in \cV,
		\end{aligned}
	\end{equation}
	where 
	\begin{equation}
		\Phi_2:= \frac{V - \Im \la}{((m-\Re \la)^2+(V- \Im \la)^2)^\frac12}.
	\end{equation}
\end{enumerate}
\end{lemma}
\begin{proof}
Let $\la = a + i b$ where $a,b \in \R$ and $a+m > 0$. The claim for $\frs_1$ follows from Lemma~\ref{lem:s.coer} since
\begin{equation}
\begin{aligned}
(\frs_1+t)[f] &:= \int_\R \frac{1}{(m+a)-i(V(x)-b)} |f'(x)|^2 \, \dd x 
\\
 & \quad  + \int_\R ((m+a)+ i(V(x)-b))|f(x)|^2 \, \dd x + (t-2a)\|f\|^2,
\end{aligned}
\end{equation}
i.e.~with the substitutions $ \wt m:=m+a$, $\wt V : =V-b$ and $\wt t:=t-2a$ and the assumptions on $a,b$ and $V$, we can apply the already obtained claims for $\frs$. Notice also that \eqref{V.nabla.form} implies that for any $\delta>0$, there is $C_\delta>0$ such that
\begin{equation}
	|\wt V'(x)|^2 = |V'(x)|^2 \leq \delta (\wt V(x) + b)^8 + C_\delta \leq \delta C (\wt V(x)^8 + 1) + C_\delta, \quad x \in \R,
\end{equation}
with some $C = C(b) > 0$. Thus $\wt V$ satisfies \eqref{V.nabla.form} as well.
(Notice that applying Lemma~\ref{lem:s.coer} yields the coercivity for the space $\wt{\cV}$ associated with $\wt m, \wt V$. However, we have $\wt\cV = \cV$ and moreover the corresponding norms are equivalent). The statement for $\frs_2$ can be shown using similar arguments. 
\end{proof}
\begin{proposition}
Let $m>0$ and let $V \in C^1(\R,\R)$ satisfy \eqref{V.nabla.form}. 
Define the operators, for $\Re \la>-m$,
	\begin{equation}
		\begin{aligned}
	 S_1(\la) & := -\partial_x \frac{1}{m - i V +\la} \partial_x + m + i V - \la,
	 \\
	 \Dom(S_1(\la)) & := \left\{ f \in \cV \, : \, -\partial_x \frac{f'}{m - i V +\la} + i V f \in L^2(\R) \right\}, 
		\end{aligned}
	\end{equation}
and, for $\Re \la < m$, 
\begin{equation}
\begin{aligned}
S_2(\la) & := \partial_x \frac{1}{m + i V -\la} \partial_x - m + i V - \la,
\\
\Dom(S_2(\la)) & := \left\{ g \in \cV \, : \, \partial_x \frac{g'}{m + i V -\la} + i V g \in L^2(\R) \right\}.
\end{aligned}
\end{equation}
Both $S_1(\la)$, $\Re \la > -m$, and $S_2(\la)$, $\Re \la < m$, are densely defined, have non-empty resolvent set. Moreover,
\begin{equation}\label{S12.strip}
0 \in \rho(S_j(\la)), \quad \Re \la \in (-m,m), \quad j =1,2.
\end{equation}
\end{proposition}
\begin{proof}
The forms $\frs_j$, $j=1,2$, are coercive in the generalized sense, see Lemma~\ref{lem:s12.coer}, and $\CcR$ is dense in $\cV$, hence all claims except \eqref{S12.strip} follow from the generalized Lax-Milgram theorem, see Theorem~\ref{thm:lm.1} in Appendix~\ref{app:AH}. 

To justify \eqref{S12.strip}, let $\la= a+ib$ with $a \in (-m,m)$ and $b \in \R$. Then
\begin{equation}
\begin{aligned}
\Re \langle S_1(\la) f,f \rangle & =  
\int_\R \Re \left( \frac{m+a + i(V(x)-b)}{(m+a)^2 + (V(x) - b)^2 } \right) |f'(x)|^2 \, \dd x   
\\ & \qquad + 
\int_\R \Re
\left( m+iV(x)-\la \right) |f(x)|^2 \, \dd x
\\
& \geq (m-a) \|f\|^2 
\end{aligned}
\end{equation}
for all $f \in \Dom(S_1(\la))$. Hence $0$ is a regular point of $S_1(\la)$, thus $S_1(\la)$ is injective and has closed range. To show that the range of $S_1(\la)$ is dense, we use the fact that $S_1(\la)^*$ is associated with the adjoint form 
\begin{equation}
\begin{aligned}
\frs_1^*[g] &= \int_\R \frac{1}{m+iV(x)+ \ov \la} |g'(x)|^2 \, \dd x + \int_\R (m - iV(x) - \ov \la)|g(x)|^2 \, \dd x, 
\\
\Dom(\frs_1^*) &= \cV.
\end{aligned}
\end{equation}
Thus, for all $g \in \Dom(S_1(\la)^*)$, 
\begin{equation}
\begin{aligned}
\Re \langle S_1(\la)^* g,g \rangle
&= 
\int_\R \Re  \left( \frac{m+a - i(V(x)-b)}{(m+a)^2 + (V(x) - b)^2 } \right) |g'(x)|^2 \, \dd x
\\
& \qquad + \int_\R \Re \left(m-iV(x)- \ov \la \right) |g(x)|^2 \, \dd x
\\
& \geq (m-a) \|g\|^2. 
	\end{aligned}
\end{equation}
Hence $S_1(\la)^*$ is injective and therefore the range of $S_1(\la)$ is dense.

The proof for $S_2(\la)$ is fully analogous.
\end{proof}

\subsection{Definition of \texorpdfstring{$H$}{H} }

We define $H$ via the strategy in \cite{Gerhat-2024-286}, i.e.~via the restriction of a distributional operator $\wh H$, see Appendix~\ref{app:Schur.dom} for details. To this end, we first introduce the spaces and operators in Assumption~\ref{asm:schur.dom} and then verify the required mapping properties. 

Suppose that $m>0$, $V \in C^1(\R,\R)$ and define the Hilbert spaces
\begin{equation}\label{spaces.hat.def}
\begin{aligned}
\cD_S &:= \cV, \quad \cD_{-S}:=\cV^*, 
\\
\cD_2 & := \{ f : \R \to \C \ \text{measurable} \, : \, (m^2+V^2)^\frac14 f \in L^2(\R) \}, 
\\ 
& \qquad \|\cdot\|_{\cD_2} := \|(m^2+V^2)^\frac14 \cdot\|_{L^2},
\\
\cD_{-2} & := \{ g : \R \to \C \ \text{measurable} \, : \, (m^2+V^2)^{-\frac14} g \in L^2(\R) \}, 
\\ 
& \qquad \|\cdot\|_{\cD_{-2}} := \|(m^2+V^2)^{-\frac14} \cdot\|_{L^2}.
\end{aligned}
\end{equation}
By construction, we have 
\begin{equation}
	\cD_S \subset L^2(\R) \subset \cD_{-S}, \qquad \cD_2 \subset L^2(\R) \subset \cD_{-2}
\end{equation}
where the embeddings are bounded and have dense ranges (we omit details). 

Next, we define the operators
\begin{equation}\label{op.hat.def}
\begin{aligned}
\wh A : \cD_S \to \cD_{-S}:&& (\wh A \phi,\psi)_{\cV^* \times \cV} &= \intR(m+i V(x)) \phi(x) \ov{\psi(x)} \, \dd x, 
\ \ \phi, \psi \in \cD_S, 
\\
\wh B : \cD_2 \to \cD_{-S}:&& (\wh B f,\phi)_{\cV^* \times \cV} &= \intR f(x) \ov{(-i \phi'(x))} \, \dd x, 
\ \ f \in \cD_2, \; \phi \in \cD_S,
\\
\wh C : \cD_S \to \cD_{-2}:&& \wh C \phi &= -i \phi',
\ \ \phi \in \cD_S,
\\
\wh D : \cD_2 \to \cD_{-2}:&& \wh D f &= (-m + i V) f, \ \ f \in \cD_2.
\end{aligned}
\end{equation}
\begin{lemma}\label{lem:Op.hat.bdd}
Let the spaces $\cD_S$, $\cD_{-S}$, $\cD_2$ and $\cD_{-2}$ be as in \eqref{spaces.hat.def} and let the operators $\wh A$, $\wh B$, $\wh C$ and $\wh D$ be as in \eqref{op.hat.def}. Then the following claims hold.
\begin{enumerate}[\upshape (i)]
	\item The conditions \eqref{op.hat.cond.bdd} are satisfied.
	\item For every $f \in \cD_2 \subset L^1_{\rm loc}(\R) \hookrightarrow \cD'(\R)$, the functional $\wh Bf$ is the unique extension of $-i f' \in \cD'(\R)$ to $\cV^*$. 
	\item For every $\la \notin \ov{\{ i V(x) -m \,: \, x \in \R\}}$, 
\begin{equation}
(\wh D-\la)^{-1} \in \cB(\cD_{-2},\cD_2).
\end{equation}
\end{enumerate}
\end{lemma}
\begin{proof}
\begin{enumerate}[\upshape (i), wide]
\item For every $\phi, \psi \in \cD_S$, 
\begin{equation}
|(\wh A \phi,\psi)_{\cV^* \times \cV}| \leq \|(m^2+V^2)^\frac14 \phi\| \|(m^2+V^2)^\frac14 \psi \| \leq \|\phi\|_{\cD_S}\|\psi\|_{\cD_S},
\end{equation}
thus $\wh A \phi \in \cV^* = \cD_{-S} $ and $\wh A$ is a bounded operator from $\cD_S$ to $\cD_{-S}$. Similarly, for every $f \in \cD_2$ and $\phi \in \cD_S$
\begin{equation}
|(\wh B f,\phi)_{\cV^* \times \cV}| \leq \|(m^2+V^2)^\frac14 f\| \|(m^2+V^2)^{-\frac14} \phi' \| \leq \|f\|_{\cD_2}\|\phi\|_{\cD_S},
\end{equation}
thus $\wh B f \in \cV^*$ and $\wh B$ is a bounded operator from $\cD_2$ to $\cD_{-S}$. It is immediate that for every $\phi \in \cD_S$
\begin{equation}
\|\wh C \phi\|_{\cD_{-2}} \leq \|\phi\|_{\cD_S}
\end{equation}
and for every $f \in \cD_2$
\begin{equation}
\|\wh D f\|_{\cD_{-2}} = \|f\|_{\cD_2},
\end{equation}
thus the claimed boundedness in \eqref{op.hat.cond.bdd} holds.
\item The claim follows since $\CcR$ is dense in $\cV$, see Lemma~\ref{lem:cV}.

\item Since $\dist(\la, \ov{\{ i V(x) -m \,: \, x \in \R\}}) \geq \delta>0$, it is easy to verify that $\wh D -\la$ is injective. For every $g \in \cD_{-2}$, we have $(-m+iV-\la)^{-1} g \in \cD_{2}$
since splitting the integration to $\cM:=\{ x \in \R \, : \, |V(x)| \leq 2(|\la|+m)\}$ and its complement yields (we omit the details of the needed straightforward estimates)
\begin{equation}
\begin{aligned}
&\intR \frac{(m^2+V(x)^2)^\frac12}{|-m+iV(x)-\la|^2} |g(x)|^2 \, \dd x 
\\   
& \quad = 
\int_{\cM} \frac{(m^2+V(x)^2)^\frac12}{|-m+iV(x)-\la|^2} |g(x)|^2 \, \dd x 
\\
& \quad \qquad  +
\int_{\cM^c} \frac{(m^2+V(x)^2)^\frac12}{|-m+iV(x)-\la|^2} |g(x)|^2 \, \dd x
\\
& \quad \ls \int_{\cM} \frac{|g(x)|^2}{\delta^2}  \, \dd x   + 
\int_{\cM^c} \frac{|g(x)|^2}{|V(x)|}  \, \dd x 
\ls \|g\|_{\cD_{-2}}^2.
\end{aligned}
\end{equation}
Notice that it also follows that $(\wh D -\la)^{-1} \in \cB(\cD_{-2},\cD_2)$.
\qedhere
\end{enumerate}
\end{proof}

We employ Theorem~\ref{thm:Schur.dom} and define an $L^2$-realization of the Dirac operator $H_0$ with non-empty resolvent set. In the next step, we show that, under the stronger assumption \eqref{eq:V.separ.cond}, the operator  $H_0$ is equal to $H$ from \eqref{H.def} and we also establish the remaining claims in Theorem~\ref{thm:H.def}.

\begin{proposition}\label{prop:H0.def}
Let $m>0$ and let $V \in C^1(\R,\R)$ satisfy \eqref{V.nabla.form}. Then the operator
\begin{equation}\label{H0.def}
\begin{aligned}
H_0 &:= \begin{pmatrix}
m & -i\partial_{x}\\
-i\partial_{x} & -m
\end{pmatrix} 
+ i  \begin{pmatrix}
V(x) & 0\\
0 & V(x)
\end{pmatrix}, 
\\
\Dom(H_0) & := \{ (u_1,u_2) \in \cV \oplus \Dom(|V|^\frac12) \,: \,  (m+iV) u_1 -i u_2' \in L^2(\R), 
\\
& \qquad \qquad  \qquad \qquad \qquad \qquad \qquad -i u_1' + (iV -m) u_2 \in L^2(\R)  \}.
\end{aligned}
\end{equation}
is densely defined in $L^2(\R,\C^2)$ and 
\begin{equation}\label{rho.strip}
\{ z \in \C \, :\, \Re z \in (-m,m) \} \subset	\rho(H_{0}). 
\end{equation}
Moreover, $C_{c}^{\infty}(\R,\C^2)$ is a core of $H_0$.
\end{proposition}
\begin{proof}
The intention is to apply Theorem~\ref{thm:Schur.dom}. To this end, one first verifies that, for every $\la \notin \ov{\{ i V(x) -m \,: \, x \in \R\}}$, $\wh{S_1}(\la)$ defined as in \eqref{eq:def.hat.S} satisfies (recall \eqref{s12.form.def})
\begin{equation}
(\wh{S_1}(\la) \phi, \psi)_{\cV^* \times \cV} = \frs_1(\phi, \psi), \qquad \phi, \psi \in \cV.
\end{equation}
We proved generalized coercivity for $\frs_1+t$, with $t>0$ sufficiently large, see Lemma~\ref{lem:s12.coer}, thus it follows that $(\wh{S_1}(\la)+t)^{-1} \in \cB(\cV^*,\cV)$, see Theorem~\ref{thm:lm.0}. By Lemma~\ref{lem:hat.ext} and \eqref{S12.strip}, we have that $\wh{S_1}(\la)^{-1} \in \cB(\cV^*,\cV) = \cB(\cD_{-S},\cD_S)$ for all $\la$ with $\Re \la \in (-m,m)$. So Theorem~\ref{thm:Schur.dom} is indeed applicable and it yields that $H_0$ is densely defined and \eqref{rho.strip} holds.

The claim on the density of $C_{c}^{\infty}(\R,\C^2)$ in $\Dom(H_0)$ follows by standard arguments which we only sketch. Let $u \in \Dom(H_0)$. Consider the cut-off
\begin{equation}
v_n := u \chi_n, \quad n \in \N,
\end{equation}
where $\chi_n(x) = \chi(x/n)$, $x \in \R$, with $\chi \in C_{c}^\infty(\R, [0,1])$ such that $\supp \chi \subset (-2,2)$ and $\chi = 1$ on $(-1,1)$. It can be verified in a straightforward way that $v_n \in \Dom(H_0)$ and, by the dominated convergence theorem, that $H_0 v_n \to H_0u$ and $v_n \to u$ in $L^2(\R, \C^2)$ as $n \to \infty$. By a subsequent mollification of $v_n$ for a fixed $n$, one shows the claimed density of $C_{c}^{\infty}(\R,\C^2)$.
\end{proof}

\begin{proposition}\label{prop:H.separation}
Let the assumptions of Theorem~\ref{thm:H.def} be satisfied and let $H_0$ be as in \eqref{H0.def}. Then
\begin{equation}\label{H0.gr.norm.sep}
\|H_0 u \|^2 + \| u \|^2 \approx \| u' \|^2 + \| V I_2 u \|^2 + \| u \|^2,
\quad  u \in \Dom(H_0).
\end{equation}
Consequently, $\Dom(H_0) = H^1(\R,  \C^2) \cap \Dom(V I_2)$ and hence $H_0 = H$ with $H$ defined in \eqref{H.def}.
\end{proposition}
\begin{proof}
Recall first that the assumption \eqref{eq:V.separ.cond} implies \eqref{V.nabla.form}, see Remark~\ref{rmk:3.2.wkr.th.2.2}, thus the statements of Proposition~\ref{prop:H0.def} hold. Moreover, it is easy to check that $H \subset H_0$.	
	
It suffices to show \eqref{H0.gr.norm.sep} for al $ \psi \in C_c^{\infty}(\R,\C^2)$.  Since the latter is a core of $H_0$, the inequality \eqref{H0.gr.norm.sep} extends to all $u \in \Dom(H_0)$. Thus the graph norm of $H_0$ is equivalent to $(\| \Ntp \cdot \|^2 + \| V I_2 \cdot \|^2 + \| \cdot \|^2)^\frac12$ which yields $\Dom(H_0) = \Dom(H)$.

It is immediate that 
\begin{equation}
\|H_0 \psi \|^2 + \| \psi \|^2 \ls \| \psi' \|^2 + \| V I_2 \psi \|^2 + \| \psi \|^2, \quad \psi \in C_c^{\infty}(\R,\C^2).
\end{equation}
To show the opposite inequality, we get for all $\psi\in C_c^{\infty}(\R,\C^2)$ that
\begin{equation}\label{H0.sep.1}
\begin{aligned}
\Vert H_0 \psi \Vert^2 = &\left\Vert
	\begin{pmatrix}
		(m+iV)\psi_{1}-i\psi_{2}'\\[1mm]
		(-m+iV)\psi_{2}-i\psi_{1}'
	\end{pmatrix}
	\right\Vert^2\\
	=& \Vert (m+iV)\psi_{1}\Vert^2 + \Vert \psi_{2}' \Vert^2 -2 \mathrm{Re}\, \langle (m+iV)\psi_{1}, i\psi_{2}'\rangle\\
	&+ \Vert (-m+iV)\psi_{2}\Vert^2  + \Vert \psi_{1}' \Vert^2-2 \mathrm{Re}\, \langle (-m+iV)\psi_{2}, i\psi_{1}'\rangle\\
	=& \Vert \psi' \Vert^2 + m^2 \Vert \psi \Vert^2 + \Vert V I_2 \psi\Vert^2  +2\mathrm{Re}\, \langle V'\psi_{1}, \psi_{2}\rangle,
\end{aligned}
\end{equation}
where the last equality follows by partial integration
\begin{equation}\label{H0.sep.2}
\begin{aligned}
	\Re\, \langle (m+iV)\psi_{1}, i\psi_{2}'\rangle =&-\mathrm{Re}\, \langle iV'\psi_{1}, i\psi_{2}\rangle-\mathrm{Re}\, \langle (m+iV)\psi_{1}',i\psi_2 \rangle\\
	=&-\mathrm{Re}\, \langle V'\psi_{1}, \psi_{2}\rangle-\mathrm{Re}\, \langle i\psi_2 , (m+iV)\psi_{1}'\rangle\\
	=& -\mathrm{Re}\, \langle V'\psi_{1}, \psi_{2}\rangle- \mathrm{Re}\, \langle (-m+iV)\psi_{2}, i\psi_{1}'\rangle.
\end{aligned}
\end{equation}
By \eqref{eq:V.separ.cond} and Cauchy-Schwarz inequality, we obtain
\begin{equation}\label{H0.sep.3}
\begin{aligned}
	\Vert H_0 \psi \Vert^2\geq &m^2 \Vert \psi \Vert^2 + \Vert V I_2 \psi\Vert^2 + \Vert \psi' \Vert^2 -2 \langle \vert V' \vert \vert\psi_{1}\vert, \vert\psi_{2}\vert\rangle
	\\
	\geq &m^2 \Vert \psi \Vert^2 + \Vert V I_2 \psi\Vert^2 + \Vert \psi' \Vert^2 -2 \varepsilon \langle \vert  V \vert \vert\psi_{1}\vert, \vert V \vert \vert\psi_{2}\vert\rangle 
	\\ & \quad -2M_\eps \langle  \vert\psi_{1}\vert, \vert\psi_{2}\vert\rangle
	\\
	\geq &m^2 \Vert \psi \Vert^2 + \Vert V I_2 \psi\Vert^2 + \Vert \psi' \Vert^2 -2 \varepsilon \Vert V \psi_{1}\Vert \Vert V \psi_{2}\Vert 
	-2M_\eps \Vert \psi_{1}\Vert \Vert\psi_{2}\Vert
	\\
	\geq &(m^2-M_\eps) \Vert \psi \Vert^2 + (1-\varepsilon)\Vert V I_2 \psi\Vert^2 + \Vert \psi' \Vert^2,
\end{aligned}
\end{equation}
from which the sought inequality follows (recall that $\eps \in (0,1)$ by assumption).
\end{proof}

\begin{proof}[Proof of Theorem~\ref{thm:H.def}]
We have already shown that $H$ satisfies \eqref{H.gr.norm.sep}, $C_c^\infty(\R, \C^2)$ is a core of $H$ and \eqref{H.rho.strip} holds, see Propositions~\ref{prop:H.separation} and \ref{prop:H0.def}. 

Let
\begin{equation}
 H^c := -i\Ntp \sigma_1 + m \sigma_3 - i V I_2, \quad \Dom(H^c) := \Dom(H),
\end{equation}
i.e.~$H^c$ coincides with the operator defined in \eqref{H*.def}. Our goal is to show that $H^c$ is indeed the adjoint of $H$. To this end, notice that Propositions~\ref{prop:H.separation} and \ref{prop:H0.def}, applied with $-i V$ instead of $i V$, yield in particular that $H^c$ satisfies \eqref{H.gr.norm.sep}, $0 \in \rho(H^c)$ and $C_c^\infty(\R, \C^2)$ is a core of $H^c$. It is straightforward to verify that
\begin{equation}
	H^c \restriction C_c^\infty(\R, \C^2) \subset H^* 
\end{equation}
thus, by taking the closures, $H^c \subset H^*$. Since $0 \in \rho(H^c) \cap \rho(H^*)$, it follows by a standard argument that $H^c = H^*$ as claimed.

In summary, the claims \ref{thm:H.def.i)}--\ref{thm:H.def.iii)} are proved.

To show \ref{thm:H.def.iv)}, notice first that the compactness of the resolvent follows from \eqref{H.gr.norm.sep} and Rellich's criterion, see e.g.~\cite[Thm.~XIII.65]{Reed4}. Finally, let $V$ satisfy \eqref{V.x.gamma}. Consider 
\begin{equation}
T_\gamma := - \partial_x^2 + |x|^{2 \gamma}, \quad \Dom(T_\gamma) := H^2(\R) \cap \Dom(|x|^{2\gamma}),
\end{equation}
which is self-adjoint in $L^2(\R)$ and whose eigenvalues $\{\mu_k\} \subset (0,\infty)$ satisfy (with some $C_\gamma>0$)
\begin{equation}
\mu_k = C_\gamma k^{\frac{2 \gamma}{\gamma+1}}(1+o(1)), \quad k \to \infty;
\end{equation} 
see~\cite{Titchmarsh-1954-5}, \cite[Sec.~7.7]{Titchmarsh-1962-book1} or \cite[Sec.~6]{Mityagin-2019-139}. From the second representation theorem, see \cite[Thm.~VI.2.23]{Kato-1966}, and \eqref{H.gr.norm.sep}, it follows that 

\begin{equation}
\cT_\gamma
H^{-1} \in \cB(L^2(\R,\C^2)), 
\qquad 
\cT_\gamma: = 
\begin{pmatrix}
T_\gamma^\frac12 & 0 \\
0 & T_\gamma^\frac12
\end{pmatrix}.
\end{equation}
Hence by writing $H^{-1}$ as
\begin{equation}
H^{-1} = \cT_\gamma^{-1} (\cT_\gamma H^{-1})
\end{equation}
we infer that the singular values $s_k(H^{-1})$ of $H^{-1}$ satisfy (see e.g.~\cite[Chap.II,\S2]{Gohberg-1969})
\begin{equation}
s_k(H^{-1}) \ls  s_k(\cT_\gamma^{-1}) \ls  \mu_k^{-\frac 12} \ls k^{-\frac{\gamma}{\gamma+1}}, \quad k \to \infty.
\end{equation}
Thus the claim \ref{thm:H.def.iv)} is proved.

Finally, the claim \ref{thm:H.def.v)} can be obtained from \cite{Sarihan-2021-610}, similarly as in \cite[Sec.~6]{ArSi-2025}.
\end{proof}

\section{Imaginary Airy-Dirac operator -- proofs}
\label{sec:Airy.proofs}

\subsection{Spectrum and basic resolvent estimates}
\label{ssec:Airy.basic}

\begin{proof}[Proof of Proposition~\ref{prop:Airy.basic} ]
\begin{enumerate}[\upshape (i), wide]
\item This is a special case of \eqref{H.gr.norm.sep}. The explicit constants can be obtained easily by following the steps in \eqref{H0.sep.1} -- \eqref{H0.sep.3}. In particular, for any $\psi \in C_{c}^{\infty}(\R,\C^2)$, we verify that
\begin{equation}
\begin{aligned}
\Vert A_{m} \psi \Vert^2 = \Vert \psi' \Vert^2+ m^2 \Vert \psi \Vert^2 + \Vert x I_2 \psi\Vert^2 + \langle \sigma_{1} \psi, \psi \rangle  .
\end{aligned}
\end{equation}
Notice that $\langle \sigma_{1} \psi, \psi \rangle =  \langle \psi_{1}, \psi_{2} \rangle+ \langle \psi_{2}, \psi_{1} \rangle$ and thus, we obtain
\begin{equation}\label{Ineq Am}
	\begin{aligned}
		\Vert A_{m} \psi \Vert^2+ \Vert \psi \Vert^2= &\Vert \psi' \Vert^2+ m^2 \Vert \psi \Vert^2 + \Vert x I_2\psi\Vert^2  + \Vert \psi_{1} + \psi_{2} \Vert^2\\
		\geq & \Vert \psi' \Vert^2+ m^2 \Vert \psi \Vert^2 + \Vert x I_2 \psi\Vert^2.
	\end{aligned}
\end{equation}
\item The standard argument for the imaginary Airy operator $-\partial_x^2 + ix$ can be used. In detail, for any $t\in \R$,
\begin{equation}\label{Translation}
	T_t A_{m}=(A_{m}-it)T_{t}
\end{equation}
where $T_{t}$ is the unitary translation operator on $L^2(\R,\C^2)$ defined by $(T_{t}u)(x)=u(x-t)$. 

Recall that the resolvent of $A_m$ is compact, see Theorem~\ref{thm:H.def}, hence the spectrum of $A_m$ is discrete. Assume that $\sigma(A_{m}) \neq \emptyset$ and take $\la \in \sigma(A_{m})$. Then there exists $\psi\in \Dom(A_{m})$ such that $A_{m}\psi =\lambda \psi$. Since $T_{t} \psi \in \Dom(A_m)$, we obtain from \eqref{Translation} that
\begin{equation}
	(A_{m}-\la -i t) T_{t} \psi = T_{t} (A_{m}- \la)\psi=0
\end{equation} 
for all $t \in \R$. Thus $\la +i\R \subset \sigma(A_{m})$, which contradicts the discreteness of $\sigma(A_{m})$. Therefore, $\sigma(A_{m})$ is empty.

To show \eqref{res.Airy.unit.eq}, notice that for all $\la \in \C$ we obtain from \eqref{Translation} that
\begin{equation}
T_{-\Im \lambda}(A_{m}-\lambda) T_{-\Im \lambda}^{-1} = A_{m}-\Re \lambda. 
\end{equation}
Hence $A_{m}-\lambda$ and $A_{m}-\Re \lambda$ are unitarily equivalent and the claim follows.
\item Notice that $\Dom(A_{m}A_{m}^{*})=\Dom(\sL)^2$ and
\begin{equation}
A_{m} A_{m}^{*}=\begin{pmatrix}
	\sL+m^2 & -1\\
	-1 & \sL+m^2
\end{pmatrix}
\end{equation}
where $\sL= -\partial_x^2 + x^2$ is the self-adjoint harmonic oscillator in $L^2(\R)$ (see also Appendix \ref{sec:app.ho}).
By conjugating with a unitary matrix 
\begin{equation}\label{U.def}
U=\frac{1}{\sqrt{2}}\begin{pmatrix}
	1 & 1\\
	1 & -1
\end{pmatrix},	
\end{equation}
we obtain
\begin{equation}\label{Conjugate by U}
	U \left(A_{m}A_{m}^{*}\right) U^{-1}=\begin{pmatrix}
		\sL+m^2-1 &0\\
		0 &\sL+m^2+1
	\end{pmatrix}.
\end{equation}
Thus
\begin{equation}
\sigma(A_{m}A_{m}^{*})= \{ 2n+m^2: n\in \N_0 \}\cup \{ 2n+2+m^2:n\in \N_0 \}
\end{equation}
and the claims follow.
\item The claim is a special case of Theorem~\ref{thm:H.def} \ref{thm:H.def.v)} with $\gamma=1$ and so $p_1 =2$. (Since $\sigma(A_m)=\emptyset$ the resolvent bound \eqref{res.est.Schatten} can be extended to $|\la| \leq e$. The claim follows also from \cite{Sarihan-2021-610} or as an application of \cite[Cor.~6.4]{ArSi-2025}.) 
\qedhere
\end{enumerate}
\end{proof}
\subsection{Resolvent of \texorpdfstring{$A_m$}{Am} and the harmonic oscillator}
\label{ssec:Am.res}
Recall that the Fourier transform $\cF: \cS(\R) \to \cS(\R)$ 
\begin{equation}
\cF \phi(y) =\frac{1}{\sqrt{2\pi}}\int_{\R} e^{-i x y} \phi(x)\, \dd x, \quad y \in \R,
\end{equation}
is bijective on the Schwartz space $\cS(\R)$, the inverse reads 
\begin{equation}
\cF^{-1} \phi (x) = \frac{1}{\sqrt{2\pi}}\int_{\R} e^{ ix y} \phi (y)\, \dd y,\quad  x \in \R,	
\end{equation}
and $\cF$ can be uniquely extended to a unitary operator on $L^2(\R)$. We keep the notation $\cF$ for this extension as well as for the Fourier transform on $L^2(\R,\C^2)$, \ie~$\cF u = (\cF u_1,\cF u_2)^t$ for $u \in L^2(\R,\C^2)$.

We introduce the operator (with $m>0$)
\begin{equation}\label{A.td.def}
\wt A_m = \cF U A_m U^{-1} \cF^{-1}, 
\end{equation}
where $U$ is the unitary matrix from \eqref{U.def}. Since both $U$ and $\cF$ are unitary in $L^2(\R,\C^2)$, we have 
\begin{equation}\label{Am.td.Am.norm}
\|(A_m-\la)^{-1}\|=\|(\wt A_m-\la)^{-1}\|, \quad \la \in \R.
\end{equation} 
In the next lemma, we relate the resolvent of $\wt A_m$ to the harmonic oscillator 
\begin{equation}\label{L.HO.def}
	\sL = -\partial_y^2 + y^2, \quad \Dom(\sL) = H^2(\R) \cap \Dom(y^2),
\end{equation}
and more precisely to the conjugated oscillator
\begin{equation}\label{HO.shift.F}
\sL_\la = - (\partial_y + \la)^2 + y^2, \quad \Dom(\sL_\la) = \Dom(\sL), \quad \la \in \R.
\end{equation}
Let $\cE_{\lambda}=\operatorname{span} \left\{h_{n,\lambda} : n\in \N_{0}\right\}$ where $h_{n,\lambda}$ are eigenfunctions of $\sL_\la$, see Appendix~\ref{app:HO.conj} for details.
\begin{lemma}\label{lem:res.form}
Let $m>0$, let $\wt A_{m}$ be as in \eqref{A.td.def} and let $\lambda\in \R$. Then 
\begin{equation}\label{A.hat.res}
	\begin{aligned}
		& ( \wt A_{m}-\lambda)^{-1} 
		\\
		& \quad =
		\begin{pmatrix}
			(\partial_{y}+y + \la)(\sL_\la +m^2-1)^{-1} & m (\sL_\la +m^2+1)^{-1}
			\\[1mm]
			m(\sL_\la +m^2-1)^{-1} & ( \partial_{y}-y +\la )(\sL_\la +m^2+1)^{-1}
		\end{pmatrix},
	\end{aligned}
\end{equation}
where $\sL_\la$ is the conjugated oscillator in \eqref{HO.shift.F}. Moreover, for all $u \in \cE_{\lambda}^2$
\begin{equation}\label{A.hat.res.S}
\begin{aligned}
& ( \wt A_{m}-\lambda)^{-1} u 
\\
& \quad =e^{-\lambda y}
\begin{pmatrix}
(\partial_{y}+y)( \sL +m^2-1)^{-1} & m ( \sL +m^2+1)^{-1}
\\[1mm]
m( \sL +m^2-1)^{-1} & ( \partial_{y}-y )( \sL +m^2+1)^{-1}
\end{pmatrix}
e^{\lambda y} u,
\end{aligned}
\end{equation}
where $\sL$ is the harmonic oscillator in \eqref{L.HO.def}.
\end{lemma}

\begin{proof}
It is straightforward to check that
\begin{equation}\label{A* hat}
\wt A_{m} 
= \begin{pmatrix}
	-\partial_{y}+y & m\\
	m & -\partial_{y}-y
\end{pmatrix},
\quad
\wt A_{m}^{*} 
= \begin{pmatrix}
	\partial_{y}+y & m\\
	m & \partial_{y}-y
\end{pmatrix}.
\end{equation}
Considering the operator $\wt T_{m}(\lambda) \coloneqq (\wt A_{m}-\lambda)(\wt A_{m}^{*}+\lambda)$, we can express the inverse of $\wt A_{m}-\lambda$ as
\begin{equation}\label{A.hat.inv}
(\wt A_{m}-\lambda)^{-1} = ( \wt A_{m}^{*}+\lambda )\wt T_{m}(\lambda)^{-1}.
\end{equation}
The inverse of $\wt T_{m}(\lambda)$ can be found by observing that
\begin{equation}
\begin{aligned}	
\wt T_{m}(\lambda) &=  \wt  A_{m} \wt A_{m}^{*} +\lambda (\wt A_{m}- \wt A_{m}^{*})-\lambda^2 I_2\\
&= \begin{pmatrix}
	\sL_\la +m^2-1 & 0\\
	0 & \sL_\la +m^2+1
\end{pmatrix}
\end{aligned}
\end{equation}
and moreover, since $m \neq 0$, that we also have $1-m^2, -1-m^2 \in \rho(\sL_\la)$ (refer to~\eqref{Spec Conj Os}). Therefore to justify \eqref{A.hat.res}, it suffices to compose $\wt A_{m}^{*}+\lambda$ and $\wt T_{m}(\lambda)^{-1}$.\\
From \eqref{Eigen Os}, \eqref{Eig Conj Os} and \eqref{Eig Conj Os HO}, it can be verified that
\begin{equation}
(\mathscr{L}_{\lambda}+m^2 \pm 1)^{-1} \phi= e^{-\lambda y} (\mathscr{L}+m^2 \pm 1)^{-1} e^{\lambda y} \phi, \qquad \phi\in \cE_{\lambda},
\end{equation}
Applying these formulae, we get that
\begin{equation} 
\wt T_{m}(\lambda)^{-1} u= e^{-\lambda y}
\begin{pmatrix}
(\sL +m^2-1)^{-1} &0\\
0 & (\sL +m^2+1)^{-1}
\end{pmatrix}
e^{\lambda y}u, \quad u \in \cE_{\lambda}^2.
\end{equation}
Finally, from the observation that
\begin{equation} (\partial_{y}+\lambda)  \phi = e^{-\lambda y} \partial_{y} e^{\lambda y}\phi, \quad \phi \in \cE_{\lambda},
\end{equation}
we get 
\begin{equation}
 (\wt A_m^* +\la) u =  e^{-\lambda y}\wt A_m^* e^{\lambda y} u, \quad u \in \cE_{\lambda}^2,	
\end{equation}
and \eqref{A.hat.res.S} follows.
\end{proof}

\subsection{Resolvent norm of \texorpdfstring{$A_m$}{Am} as \texorpdfstring{$\la \to \pm \infty$}{lambda -> infinity}  via Schur test}
\label{ssec:res.norm}

Employing \eqref{A.hat.res.S}, we first finding an explicit integral kernel of $(\wt A_{m}-\lambda)^{-1}$. 
\begin{lemma}\label{lem:res.norm.form}
	Let $m>0$ and $\wt A_{m}$ be as in \eqref{A.td.def} and $U(a,z)$ be the parabolic cylinder function in Appendix \ref{ssec:app.par.cyl}. Then, for each $\la \in \R$, the resolvent of $\wt A_{m}$ can be decomposed as 
	\begin{equation}
	(\wt A_{m}-\lambda)^{-1}= \mathcal{R}_{\lambda,1} + \mathcal{R}_{\lambda,2},
	\end{equation}
	where $\mathcal{R}_{\lambda,1}$ and $\mathcal{R}_{\lambda,2}$ are bounded integral operators on $L^2(\R,\C^2)$
	\begin{equation}\label{R.la12.op}
	(\mathcal{R}_{\lambda,1}f)(y) =
	\int_{\R} \mathcal{R}_{\lambda,1}(y,\xi)f(\xi)\, \dd \xi,\quad (\mathcal{R}_{\lambda,2}f)(y)=\int_{\R} \mathcal{R}_{\lambda,2}(y,\xi)f(\xi)\, \dd \xi, 
	\end{equation}
	whose kernels, also denoted by $\mathcal{R}_{\lambda,1}$ and $\mathcal{R}_{\lambda,2}$, read
	\begin{equation} \label{R.la12.def} 
		\begin{aligned}
			&\mathcal{R}_{\lambda,1}(y,\xi) =\frac{m}{\sqrt2} c_m
			e^{\lambda (\xi-y)} \mathds{1}_{\{y\geq \xi\}} \times \\
			&\begin{pmatrix}
				- U(\frac{m^2+1}{2},\sqrt{2}y)U(\frac{m^2-1}{2},-\sqrt{2}\xi) & \frac{m}{\sqrt{2}} U(\frac{m^2+1}{2},\sqrt{2}y) U(\frac{m^2+1}{2},-\sqrt{2}\xi)\\[1mm]
				\frac{\sqrt{2}}{m} U(\frac{m^2-1}{2},\sqrt{2}y)U(\frac{m^2-1}{2},-\sqrt{2}\xi) &-U(\frac{m^2-1}{2},\sqrt{2}y) U(\frac{m^2+1}{2},-\sqrt{2}\xi)
			\end{pmatrix}
			,\\
			&\mathcal{R}_{\lambda,2}(y,\xi) = \frac{m}{\sqrt2} c_m
			e^{\lambda (\xi-y)} \mathds{1}_{\{y\leq \xi \}} \times \\
			& \begin{pmatrix}
				U(\frac{m^2+1}{2},-\sqrt{2}y)U(\frac{m^2-1}{2},\sqrt{2}\xi) & \frac{m}{\sqrt{2}} U(\frac{m^2+1}{2},-\sqrt{2}y) U(\frac{m^2+1}{2},\sqrt{2}\xi)\\[1mm]
				\frac{\sqrt{2}}{m} U(\frac{m^2-1}{2},-\sqrt{2}y)U(\frac{m^2-1}{2},\sqrt{2}\xi) & U(\frac{m^2-1}{2},-\sqrt{2}y) U(\frac{m^2+1}{2},\sqrt{2}\xi)
			\end{pmatrix},
	\end{aligned} 
\end{equation}
where 
\begin{equation}\label{cm.def}
	c_m:=\frac{ \Gamma(\frac{m^2}{2}+1)}{m\sqrt{\pi}}.
\end{equation}
\end{lemma}
\begin{proof}
	By using the recurrence formulae \eqref{Recurrence}, we have
\begin{equation} 
\begin{aligned}
			&(\partial_{y}+y) U\left(-\frac{\eta}{2}, \sqrt{2} y\right)=-\frac{\sqrt{2}(1-\eta)}{2}U\left(-\frac{\eta}{2}+1, \sqrt{2} y\right),\\
			&(\partial_{y}+y) U\left(-\frac{\eta}{2}, -\sqrt{2} y\right)
			=\frac{\sqrt{2}(1-\eta)}{2}U\left(-\frac{\eta}{2}+1, -\sqrt{2} y\right),\\
			&(\partial_{y}-y) U\left(-\frac{\eta}{2}, \sqrt{2} y\right)=-\sqrt{2}U\left(-\frac{\eta}{2}-1, \sqrt{2} y\right),\\
			&(\partial_{y}-y) U\left(-\frac{\eta}{2}, -\sqrt{2} y\right)=\sqrt{2}U\left(-\frac{\eta}{2}-1, -\sqrt{2} y\right), \quad \eta, y \in \R.
\end{aligned} 
\end{equation}
Applying these formulae together with \eqref{A.hat.res.S}, \eqref{Prop HO Res} and the standard recurrence relation $\frac{m^2}{2} \Gamma(\frac{m^2}{2}) = \Gamma(\frac{m^2}{2}+1)$, we obtain
\begin{equation} 
(\wt A_{m}-\lambda)^{-1}u= \left(\mathcal{R}_{\lambda,1} + \mathcal{R}_{\lambda,2}\right)u, \qquad u \in \cE_{\lambda}^2.
\end{equation}
We show in Lemma \ref{lem:R1.est} below that $\mathcal{R}_{\lambda,1}$ and $\mathcal{R}_{\lambda,2}$ are bounded operators on $L^2(\R,\C^2)$.
Then, by the density of $\cE_{\lambda}^2$ in $L^2(\R,\C^2)$ (cf.~Appendix \ref{app:HO.conj}), the equality extends to $L^2(\R,\C^2)$.
\end{proof}
In the following, given a matrix $A\in \C^{2\times 2}$, we denote 
\begin{equation} 
|A|_{\C^{2\times 2}} = \sup_{ v \neq 0} \frac{|A v|_{\C^2}}{| v |_{\C^2}}.
\end{equation}
Notice that if $\det A=0$, then 
\begin{equation}\label{Operator norm}
\left\vert A \right\vert_{\C^{2 \times 2}}= \sqrt{\sum_{1\leq i,j \leq 2} \vert A_{ij} \vert^2}.
\end{equation}
Indeed, since $\det (A^* A)=0$, it follows that $0$ and $\Tr (A^* A)$ are two non-negative eigenvalues of $A^* A$, so we have
\begin{equation}
|A|_{\mathbb{C}^{2\times 2}}^2 = \lambda_{\max} (A^*A)= \Tr \left( A^* A\right)= \sum_{1\leq i,j \leq 2} \vert A_{ij} \vert^2,
\end{equation}
where $\lambda_{\max}(A^* A)$ represents the largest eigenvalue of $A^{*}A$.

From the formulas of kernels $\mathcal{R}_{\lambda,1}$ and $\mathcal{R}_{\lambda,1}$ (refer to~\eqref{R.la12.def}), it is easy to see that their determinants are zero and therefore
\begin{equation}\label{Matrix norm}
\begin{aligned}
|\mathcal{R}_{\lambda,1}(y,\xi)|_{\C^{2\times 2}} &= c_m e^{\lambda (\xi-y)}u_{m}(y) u_m(-\xi)\mathds{1}_{\{y\geq \xi\}},\\
|\mathcal{R}_{\lambda,2}(y,\xi)|_{\C^{2\times 2}} &= c_m e^{\lambda (\xi-y)}u_m(-y) u_m(\xi)\mathds{1}_{\{y\leq \xi\}},
\end{aligned}
\end{equation}
where $c_m$ is as in \eqref{cm.def} and 
\begin{equation}\label{um.def}
u_m(y)\coloneqq  \sqrt{\frac{m^2}{2}\left|U\left(\frac{m^2+1}{2},\sqrt{2}y\right)\right|^2+ \left|U\left(\frac{m^2-1}{2},\sqrt{2}y\right)\right|^2}, \quad y \in \R.
\end{equation}
It is also worth noticing that for all $y,\xi \in \R$
\begin{equation}\label{anti sym}
	\begin{aligned}
		\vert\mathcal{R}_{\lambda,1}(y,\xi)\vert_{\C^{2\times 2}}=&\vert \mathcal{R}_{\lambda,1}(-\xi,-y) \vert_{\C^{2\times 2}},
		\\
		\vert\mathcal{R}_{\lambda,2}(y,\xi)\vert_{\C^{2\times 2}}=&\vert \mathcal{R}_{\lambda,2}(-\xi,-y) \vert_{\C^{2\times 2}}.
	\end{aligned}
\end{equation}
The asymptotic formulae for $U$ in \eqref{Asymp U +} and \eqref{Asymp U -} yield
\begin{equation}\label{um.asym}
	\begin{aligned}
		u_m(y) =& e^{-\frac{y^2}{2}} (\sqrt{2}y)^{-\frac{m^2}{2}}\left(1+\mathcal{O}\left(\frac{1}{|y|^2}\right)\right),\qquad &&y \to +\infty,\\
		u_m(y) =& \frac{1}{c_m} e^{\frac{y^2}{2}} (-\sqrt{2}y)^{\frac{m^2}{2}}\left(1+\mathcal{O}\left(\frac{1}{|y|^2}\right)\right),\qquad &&y \to -\infty.
	\end{aligned}
\end{equation}
Given a matrix-valued kernel $K:\R^2 \to \C^{2\times 2}$, consider the associated integral operator $\mathcal{R}_{K}$ on $L^2(\R,\C^2)$ defined by
\begin{equation}\label{R.K.def}
\mathcal{R}_{K} f(x) = \int_{\R} K(x,y) f(y)\, \dd y, \qquad f \in L^2(\R,\C^2),\ x\in \R.
\end{equation}
Due to the inequality
\begin{equation}\label{R.K.norm}
\left\vert \int_{\R} K(x,y) f(y)\, \dd y \right\vert_{\C^2} \leq \int_{\R} \vert K(x,y) \vert_{\C^{2\times 2}} \vert f(y) \vert_{\C^2}\, \dd y,
\end{equation}
for all $f \in L^2(\R,\C^2)$ and $x \in \R$, 
we deduce that the operator $\mathcal{R}_{K}$ is bounded on $L^2(\R,\C^2)$ and
\begin{equation} 
	\Vert \mathcal{R}_{K} \Vert \leq  \Vert \mathcal{R}_{|K|} \Vert
\end{equation}
if the operator $\mathcal{R}_{|K|}$, the integral operator with the scalar kernel $\vert K(x,y) \vert_{\C^{2\times 2}}$, is bounded on $L^2(\R)$. Thus an upper bound from the Schur test for the scalar kernel also produces an upper bound for the matrix-valued kernel.
\begin{lemma}\label{lem:R1.est}
Let $m>0$ and let $\mathcal{R}_{\lambda,1}$ and $\mathcal{R}_{\lambda,2}$ be integral operators defined as in Lemma \ref{lem:res.norm.form}. Then, for each $\lambda\in \R$, $\mathcal{R}_{\lambda,1}$ and $\mathcal{R}_{\lambda,2}$ are bounded on $L^2(\R,\C^2)$. Furthermore, we have 
\begin{equation}
	\begin{aligned}
		\|\mathcal{R}_{\lambda,1}\|  =& \BigO(|\la|^{-1}), &&\lambda \to+\infty,\\
		\|\mathcal{R}_{\lambda,2}\| =& \BigO(|\la|^{-1}), &&\lambda \to-\infty.
	\end{aligned} 
\end{equation}
\end{lemma}
\begin{proof}
	Let $m>0$ and $\lambda \in \R$ be fixed. We first prove that the integral operator $\mathcal{R}_{\lambda,1}$ is bounded on $L^2(\R,\C^2)$. To this end, we apply the Schur test with trivial weight functions. More precisely, we establish 
	\begin{equation}\label{Schur 1} 
		\begin{aligned}
			&\int_{\R} \vert \mathcal{R}_{\lambda,1}(y,\xi) \vert_{\C^{2\times 2}} \, \dd \xi \lesssim 1,\qquad y\in \R,\\
			&\int_{\R} \vert \mathcal{R}_{\lambda,1}(y,\xi) \vert_{\C^{2\times 2}} \, \dd y \lesssim 1, \qquad \xi\in \R,
		\end{aligned}
	\end{equation}
	where the implicit constants depend only on $m$ and $\lambda$. By the Schur test, this proves the claim (see \eqref{R.K.def} -- \eqref{R.K.norm}).
	
	By \eqref{um.asym}, we may choose $\alpha>0$ sufficiently large so that
	\begin{equation}\label{Estimate hm}
		\begin{aligned}
			&u_m(y) \lesssim e^{-\frac{y^2}{2}} y^{-\frac{m^2}{2}},\qquad &&y\geq \alpha,\\
			&u_m(y) \lesssim e^{\frac{y^2}{2}} (-y)^{\frac{m^2}{2}},\qquad &&y\leq -\alpha,
		\end{aligned}
	\end{equation}
	and, in addition,
	\begin{equation}\label{Add Cond alpha}
		\lambda+\alpha>0.
	\end{equation}
	From \eqref{Matrix norm}, we have
	\begin{equation} 
		\begin{aligned}
			I_{\lambda}(y)\coloneqq\int_{\R} \vert \mathcal{R}_{\lambda,1}(y,\xi) \vert_{\C^{2\times 2}}\, \dd \xi =c_m \int_{-\infty}^{y} e^{\lambda(\xi-y)}u_m(y)u_m(-\xi)\, \dd \xi.
		\end{aligned} 
	\end{equation}
	We consider three cases.
	\begin{enumerate}[\upshape (i), wide]
		\item \label{first case} $y\in (-\infty,-\alpha)$. Using \eqref{Estimate hm}, we obtain
		\begin{equation}\label{Est 1}
			\begin{aligned}
				I_{\lambda}(y) =& c_me^{-\lambda y} u_m(y) \int_{-\infty}^{y} e^{\lambda \xi}u_m(-\xi)\, \dd \xi 
				\\
				\lesssim & e^{-\lambda y+\frac{y^2}{2}} (-y)^{\frac{m^2}{2}} \int_{-\infty}^{y} e^{\lambda \xi-\frac{\xi^2}{2}} (-\xi)^{-\frac{m^2}{2}}\, \dd \xi
				\\
				= & e^{-\lambda y+\frac{y^2}{2}} \int_{-\infty}^{y} 
				\left|\frac{y}{\xi}\right|^{\frac{m^2}{2}} \left[e^{\lambda \xi-\frac{\xi^2}{2}}\right]'\frac{1}{\lambda - \xi} \, \dd \xi
				\ls  \frac{1}{\lambda+\alpha}.
		\end{aligned} \end{equation}
		\item \label{second case}$y\in [-\alpha,\alpha]$. We decompose
		\begin{equation}
			\begin{aligned}
				I_{\lambda}(y) & = c_me^{-\lambda y} u_m(y) \int_{-\infty}^{-\alpha} e^{\lambda \xi}u_m(-\xi)\, \dd \xi 
				\\
				& \qquad +c_me^{-\lambda y} u_m(y) \int_{-\alpha}^{y} e^{\lambda \xi}u_m(-\xi)\, \dd \xi.
			\end{aligned}
		\end{equation}
		Since $e^{-\lambda y} u_m(y)$ is bounded on $[-\alpha,\alpha]$, the first term satisfies
		\begin{equation}
			\begin{aligned}
				c_me^{-\lambda y} u_m(y) \int_{-\infty}^{-\alpha} e^{\lambda \xi}u_m(-\xi)\, \dd \xi \lesssim & \int_{-\infty}^{-\alpha} e^{\lambda \xi-\frac{\xi^2}{2}} (-\xi)^{-\frac{m^2}{2}}\, \dd \xi\\
				\lesssim & e^{\frac{\lambda^2}{2}}\int_{-\infty}^{-\alpha} e^{-\frac{1}{2}(\xi-\lambda)^2}\, \dd \xi\\
				\lesssim &  e^{\frac{\lambda^2}{2}}\int_{-\infty}^{+\infty} e^{-\frac{1}{2}\xi^2}\, \dd \xi \lesssim  e^{\frac{\lambda^2}{2}}.
			\end{aligned}
		\end{equation}
		Furthermore, since the continuous function $u_{m}$ is bounded on the compact set $[-\alpha,\alpha]$,
		\begin{equation}
			c_me^{-\lambda y} u_m(y) \int_{-\alpha}^{y} e^{\lambda \xi}u_m(-\xi)\, \dd \xi \lesssim 1.
		\end{equation}
		%
		%
		\item  $y\in (\alpha,\infty)$. We write
		\begin{equation}
			\begin{aligned}
				I_{\lambda}(y) = & c_me^{-\lambda y} u_m(y) \int_{-\infty}^{-\alpha} e^{\lambda \xi}u_m(-\xi)\, \dd \xi 
				\\ & \qquad +c_me^{-\lambda y} u_m(y) \int_{-\alpha}^{\alpha} e^{\lambda \xi}u_m(-\xi)\, \dd \xi
				\\
				& \qquad +c_me^{-\lambda y} u_m(y) \int_{\alpha}^{y} e^{\lambda \xi}u_m(-\xi)\, \dd \xi .
			\end{aligned}
		\end{equation}
		By \eqref{Estimate hm},
		\begin{equation}
			\begin{aligned}
				c_me^{-\lambda y} u_m(y) \int_{-\infty}^{-\alpha} e^{\lambda \xi}u_m(-\xi)\, \dd \xi \lesssim & e^{-\lambda y-\frac{y^2}{2}} y^{-\frac{m^2}{2}} \int_{-\infty}^{-\alpha} e^{\lambda \xi-\frac{\xi^2}{2}} (-\xi)^{-\frac{m^2}{2}}\, \dd \xi\\
				\lesssim & e^{\lambda^2} e^{-\frac{1}{2}(y+\lambda)^2} \int_{-\infty}^{-\alpha} e^{-\frac{1}{2}(\xi-\lambda)^2}\, \dd \xi\\
				\lesssim & e^{\lambda^2}\int_{-\infty}^{+\infty} e^{-\frac{1}{2}\xi^2}\, \dd \xi \lesssim e^{\lambda^2}.
			\end{aligned}
		\end{equation}
		Similarly,
		\begin{equation}
			\begin{aligned}
				c_me^{-\lambda y} u_m(y) \int_{-\alpha}^{\alpha} e^{\lambda \xi}u_m(-\xi)\, \dd \xi & \lesssim  e^{-\lambda y-\frac{y^2}{2}} y^{-\frac{m^2}{2}} \lesssim  e^{\frac{1}{2}\lambda^2} e^{-\frac{1}{2}(y+\lambda)^2}
				\\
				&  \lesssim e^{\frac{1}{2}\lambda^2}.
			\end{aligned}
		\end{equation}
		Finally,
		\begin{equation}\label{Est 3}
			\begin{aligned}
				& c_me^{-\lambda y} u_m(y) \int_{\alpha}^{y} e^{\lambda \xi}u_m(-\xi)\, \dd \xi 
				\\& \quad \ls  e^{-\lambda y-\frac{y^2}{2}} y^{-\frac{m^2}{2}} \int_{\alpha}^{y} e^{\lambda \xi + \frac{\xi^2}{2}} \xi^{\frac{m^2}{2}} \, \dd \xi
				\\
				& \quad 
				=  e^{-\lambda y-\frac{y^2}{2}} \int_{\alpha}^{y} \left( \frac{\xi}{y} \right)^{\frac{m^2}{2}} \left[e^{\lambda \xi+\frac{\xi^2}{2}}\right]' \frac{1}{\lambda+\xi} \, \dd \xi
 \lesssim \frac{1}{\lambda+\alpha}.
			\end{aligned}
		\end{equation}
	\end{enumerate}
	Combining the above estimates, we conclude that  $I_{\lambda}(y)\lesssim 1$ uniformly for all $y\in\R$. The second estimate in \eqref{Schur 1} follows from \eqref{anti sym}, then we conclude that $\mathcal{R}_{\lambda,1}$ is bounded on $L^2(\R,\C^2)$. From ~\eqref{Matrix norm} and \eqref{anti sym}, we have
	\begin{equation}\label{Equalities} 
		\vert \mathcal{R}_{\lambda,2}(y,\xi) \vert_{\C^{2\times 2}} =\vert \mathcal{R}_{-\lambda,1}(-y,-\xi) \vert_{\C^{2\times 2}}=\vert \mathcal{R}_{-\lambda,1}(\xi,y) \vert_{\C^{2\times 2}}, \quad (y,\xi)\in \R^2.
	\end{equation}
	Then, by Schur test, $\mathcal{R}_{\lambda,2}$ is also bounded on $L^2(\R,\C^2)$.
	
	Now, fix $m>0$ and consider $\lambda>0$. In this case, condition \eqref{Add Cond alpha} is automatically satisfied, so $\alpha$ can be chosen independently of $\lambda$. As above, we distinguish three cases. The only difference is that the implicit constants appearing in the estimates below are independent of $\lambda$.
	\begin{enumerate}[\upshape (i), wide]
		\item  $y\in (-\infty,-\alpha)$. From \eqref{Est 1}, for $\lambda>0$,
		\begin{equation} \begin{aligned}
				I_{\lambda}(y) \lesssim \frac{1}{\lambda+\alpha}\lesssim \frac{1}{\lambda}.
		\end{aligned} 
		\end{equation}
		\item $y\in [-\alpha,\alpha]$. Since $u_m$ is bounded on $[-\alpha,\alpha]$ and from \eqref{Estimate hm}, $u_m(-\xi)$ is also bounded for all $\xi\in (-\infty,\alpha]$, we have, for $\lambda>0$,
		\begin{equation}\label{Est 2}
			I_{\lambda}(y) \lesssim  \int_{-\infty}^{y} e^{\lambda (\xi-y)}\, \dd \xi =\frac{1}{\lambda}.
		\end{equation}
		\item  $y\in (\alpha,\infty)$. We write
		\begin{equation} 
			I_{\lambda}(y) = c_m\int_{-\infty}^{\alpha} e^{\lambda(\xi-y)}u_m(y)u_m(-\xi)\, \dd \xi
			+
			c_m\int_{\alpha}^{y} e^{\lambda(\xi-y)}u_m(y)u_m(-\xi)\, \dd \xi.
		\end{equation}
		The first integral is controlled as in \eqref{Est 2} and the second integral is estimated as in \eqref{Est 3}, more precisely, for $\lambda>0$,
		\begin{equation} 
			\begin{aligned}
				\int_{-\infty}^{\alpha} e^{\lambda(\xi-y)}u_m(y)u_m(-\xi)\, \dd \xi
				&\lesssim  \int_{-\infty}^{\alpha} e^{\lambda (\xi-y)} \, \dd \xi =\frac{1}{\lambda} e^{\lambda(\alpha-y)}
				\lesssim \frac{1}{\lambda},
				\\
				\int_{\alpha}^{y} e^{\lambda(\xi-y)}u_m(y)u_m(-\xi)\, \dd \xi 
				& \lesssim \frac{1}{\lambda+\alpha}\lesssim \frac{1}{\lambda}.
		\end{aligned} \end{equation}
	\end{enumerate}
	Combining the above estimates, $\int_{\R} \vert \mathcal{R}_{\lambda,1}(y,\xi) \vert_{\C^{2\times 2}} \, \dd \xi\lesssim \frac{1}{\lambda}$ as $\lambda\to+\infty$. Using \eqref{anti sym}, it follows that $\int_{\R} \vert \mathcal{R}_{\lambda,1}(y,\xi) \vert_{\C^{2\times 2}} \, \dd y \lesssim \frac{1}{\lambda}$ as $\lambda\to+\infty$. An application of Schur test then yields the asserted upper bound on the operator norm of $\mathcal{R}_{\lambda,1}$. Finally, the result for $\mathcal{R}_{\lambda,2}$ follows immediately from \eqref{Equalities}.
\end{proof}

\begin{lemma}\label{lem:R2.est.up}
Let $m>0$ and let $\mathcal{R}_{\lambda,1}$ and $\mathcal{R}_{\lambda,2}$ be integral operators defined as in Lemma \ref{lem:res.norm.form}. We have (with $\la \in \R$)
\begin{equation}\label{R.la.Schur}
\begin{aligned}
&\Vert \mathcal{R}_{\lambda,2} \Vert \leq \frac{ \Gamma(\frac{m^2}{2}+1)}{2^{\frac{m^2}{2}}m} \frac{e^{\lambda^2}}{\lambda^{m^2}}\left(1+\mathcal{O}\left(\frac{1}{|\lambda|^2}\right)\right),\quad &&\lambda\to +\infty,\\
&\|\mathcal{R}_{\lambda,1}\| \leq  \frac{ \Gamma(\frac{m^2}{2}+1)}{2^{\frac{m^2}{2}}m} \frac{e^{\lambda^2}}{(-\lambda)^{m^2}}\left(1+\mathcal{O}\left(\frac{1}{|\lambda|^2}\right)\right), \quad &&\lambda \to-\infty.
\end{aligned} 
\end{equation}
\end{lemma}
\begin{proof}
Recall \eqref{R.K.def} -- \eqref{R.K.norm}. We start the proof by showing the upper bound of $\mathcal{R}_{\lambda,2}$ as $\lambda\to +\infty$ by Schur test with non-trivial weights (depending on $\lambda$)
\begin{equation}\label{pq.def} 
\begin{aligned}
p_{\lambda}(y) &:=
\begin{cases}
e^{-\lambda y} u_m(-y), & |y| \leq 2\lambda,\\
e^{-2\lambda^2}u_m(-2\lambda), & y> 2\lambda,\\
e^{2\lambda^2}u_m(2\lambda), & y<- 2\lambda,
\end{cases}
\\
q_{\lambda}(\xi) &:=
\begin{cases}
e^{\lambda \xi} u_m(\xi), & |\xi| \leq 2\lambda,\\
e^{2\lambda^2}u_m(2\lambda), & \xi> 2\lambda,\\
e^{-2\lambda^2}u_m(-2\lambda), &  \xi<- 2\lambda.
\end{cases}
\end{aligned}
\end{equation}
Notice that $U(a_{m},z)$ (where $a_m = (m^2 \pm 1)/2$) has no real zero (see \cite[Sec. 12.11(i)]{DLMF}), thus the function $u_m$ defined in \eqref{um.def} is positive and so are $p_{\lambda},q_{\lambda}$. Let us denote (see \eqref{cm.def} and \eqref{Matrix norm})
\begin{equation}
\begin{aligned}
J_{\lambda}(y) & \coloneqq \int_{\R} \vert \mathcal{R}_{\lambda,2}(y,\xi) \vert_{\C^{2\times 2}} q_{\lambda}(\xi)\, \dd \xi
\\
& =c_m e^{-\lambda y}u_m(-y) \int_{y}^{\infty} e^{\lambda \xi} u_m(\xi)q_{\lambda}(\xi)\, \dd \xi.
\end{aligned}	
\end{equation}

We assume further that $\lambda>2 \alpha$ with $\alpha>0$ appearing in \eqref{Estimate hm}. We split the estimate in the following cases (the estimates below are uniform in $y$).
\begin{enumerate}[\upshape (i), wide]
\item \label{J.it.i} $y\in [2\lambda,\infty)$. By \eqref{pq.def} and \eqref{Estimate hm}, we have
\begin{equation}
\begin{aligned}
J_{\lambda}(y)&=c_m e^{2\lambda^2}u_m(2\lambda) e^{-\lambda y}u_m(-y) \int_{y}^{\infty} e^{\lambda\xi}u_m(\xi)\, \dd \xi
\\
&\lesssim e^{2\lambda^2}u_m(2\lambda)  e^{\frac{y^2}{2}-\lambda y} y^{\frac{m^2}{2}}\int_{y}^{\infty} e^{\lambda \xi-\frac{\xi^2}{2}}\xi^{-\frac{m^2}{2}}\, \dd \xi
\\
& =
e^{2\lambda^2}u_m(2\lambda)  e^{\frac{y^2}{2}-\lambda y} 
\int_{y}^{\infty} 
\left(\frac{y}{\xi}\right)^\frac{m^2}{2}
\left[- e^{\lambda \xi-\frac{\xi^2}{2}}\right]'  \frac{1}{\xi-\lambda} \dd \xi
\\
& 
\lesssim  \frac{e^{2\lambda^2}u_m(2\lambda)}{\lambda }, \quad \la \to + \infty.
\end{aligned} \end{equation}
From \eqref{um.asym}, we obtain that, for all $y\in [2\lambda,\infty)$,
\begin{equation}\label{Case 1}
\begin{aligned}
	\frac{J_{\lambda}(y)}{p_{\lambda}(y)}\lesssim \frac{1}{\lambda } \frac{ e^{2\lambda^2}u_m(2\lambda)}{ e^{-2\lambda^2}u_m(-2\lambda)}
	\lesssim
	\lambda^{-m^2-1}, \qquad \lambda \to +\infty.
\end{aligned}
\end{equation}
\item  $y\in [ \lambda/2, 2\lambda )$. Using \eqref{pq.def} and splitting the integration, we get
\begin{equation}\label{J.ii}
\begin{aligned}
\frac{J_{\lambda}(y)}{p_{\lambda}(y)} &=  c_m
\int_{y}^{2\lambda} e^{2\lambda \xi} u_m^2(\xi)\, \dd \xi 
+
c_m \int_{2\lambda}^{\infty} e^{\lambda \xi} u_m(\xi)q_{\lambda}(\xi)\, \dd \xi .
\\ 
& = 
c_m \int_{y}^{2\lambda} e^{2\lambda \xi} u_m^2(\xi)\, \dd \xi 
+
\frac{J_\la(2\la)}{p_\la(2\la)}. 
\end{aligned}
\end{equation}
To estimate the integral, we change the integration variable to $\xi =\lambda s$ and use Laplace's method; notice that the maximum of the appearing function $2s-s^2$ is attained at $s=1$. In detail, employing also \eqref{um.asym}, we arrive at
\begin{equation} \label{I.ii.est}
	\begin{aligned}
 \int_{y}^{2\lambda} e^{2\lambda \xi} u_m^2(\xi)\, \dd \xi
& \leq  \int_{\lambda/2}^{2\lambda} e^{2\lambda \xi} u_m^2(\xi)\, \dd \xi
\\
& =\left(1+\mathcal{O}\left(\frac{1}{|\lambda|^2}\right)\right)
\int_{\lambda/2}^{2\lambda}e^{2\lambda \xi -\xi^2} (\sqrt{2}\xi)^{-m^2} \, \dd \xi
\\
& = \left(1+\mathcal{O}\left(\frac{1}{|\lambda|^2}\right)\right)
\frac{\la}{2^{\frac{m^2}{2}} \lambda^{m^2} } \int_{1/2}^{2} e^{\lambda^2(2s-s^2)}s^{-m^2} \, \dd s
\\
& =\left(1+\mathcal{O}\left(\frac{1}{|\lambda|^2}\right)\right) \frac{ \sqrt \pi }{ 2^{\frac{m^2}{2}}} \frac{e^{\lambda^2}}{\lambda^{m^2}}, \qquad \lambda \to +\infty.
\end{aligned} 
\end{equation}
Returning to \eqref{J.ii} and using \eqref{Case 1}, we obtain for all $y\in \left[ \lambda/2,2\lambda\right)$ that
\begin{equation}\label{Case 2}
\frac{J_{\lambda}(y)}{p_{\lambda}(y)} \leq \left(1+\mathcal{O}\left(\frac{1}{|\lambda|^2}\right)\right)
c_m \frac{ \sqrt \pi }{ 2^{\frac{m^2}{2}}}
\frac{e^{\lambda^2}}{\lambda^{m^2}}, \qquad \lambda \to +\infty.
\end{equation}
\item  $y\in [ \alpha, \lambda/2)$. Similarly as above, we have
\begin{equation}\label{J.iii}
\frac{J_{\lambda}(y)}{p_{\lambda}(y)} 
 = c_m 
\int_{y}^{\lambda/2} e^{2\lambda \xi} u_m^2(\xi)\, \dd \xi +  \frac{J_{\lambda}(\lambda/2)}{p_{\lambda}(\lambda/2)}.
\end{equation}
We estimate the integral by using \eqref{Estimate hm} and Laplace's method with maximum at the right endpoint
\begin{equation} \int_{a}^{b} e^{\mu f(x)} \, \dd x = \frac{e^{\mu f(b)}}{ \mu f'(b)} \left(1+\mathcal{O}\left(\frac{1}{\mu}\right)\right),\qquad \mu\to+\infty.
\end{equation}
Thus we have as $\lambda \to+\infty$ that
\begin{equation} 
\begin{aligned}
\int_{y}^{\lambda/2} e^{2\lambda \xi} u_m^2(\xi)\, \dd \xi 
&\lesssim \int_{0}^{\lambda/2}e^{2\lambda \xi -\xi^2} \, \dd \xi
\lesssim\int_{0}^{1/2} e^{\lambda^2(2s-s^2)} \lambda\, \dd s\lesssim \frac{e^{\frac{3}{4}\lambda^2}}{\lambda}.
\end{aligned} 
\end{equation}
Returning to \eqref{J.iii} and applying \eqref{Case 2}, we conclude that for all $y\in \left[ \alpha, \lambda/2\right)$
\begin{equation}\label{Case 3}
\frac{J_{\lambda}(y)}{p_{\lambda}(y)} \leq \left(1+\mathcal{O}\left(\frac{1}{|\lambda|^2}\right)\right)
c_m \frac{ \sqrt \pi }{ 2^{\frac{m^2}{2}}} 
\frac{e^{\lambda^2}}{\lambda^{m^2}}, \qquad  \lambda \to +\infty.
\end{equation}
\item $y\in \left[ -\alpha, \alpha\right)$. As above, we have
\begin{equation}
\frac{J_{\lambda}(y)}{p_{\lambda}(y)}= 
c_m
\int_{y}^{\alpha} e^{2\lambda \xi} u_m^2(\xi)\, \dd \xi + \frac{J_{\lambda}(\alpha)}{p_{\lambda}(\alpha)}.
\end{equation}
From the boundedness of $u_m$ on $[-\alpha,\alpha]$, we get
\begin{equation} 
c_m\int_{y}^{\alpha} e^{2\lambda \xi} u_m^2(\xi)\, \dd \xi\lesssim \int_{y}^{\alpha}e^{2\lambda \xi}\, \dd \xi\lesssim \frac{e^{2\alpha\lambda}}{\lambda}.
\end{equation}
Thus, employing \eqref{Case 3}, we obtain for all $y\in [ -\alpha, \alpha)$ that
\begin{equation}\label{Case 4}
\frac{J_{\lambda}(y)}{p_{\lambda}(y)}\leq \left(1+\mathcal{O}\left(\frac{1}{|\lambda|^2}\right)\right)
c_m \frac{ \sqrt \pi }{ 2^{\frac{m^2}{2}}} \frac{e^{\lambda^2}}{\lambda^{m^2}}, \qquad  \lambda \to +\infty.
\end{equation}
\item  \label{J.it.v} $y\in \left[ -2\lambda, -\alpha\right)$. As above, we have
\begin{equation} 
\begin{aligned}
\frac{J_{\lambda}(y)}{p_{\lambda}(y)} = c_m 
\int_{y}^{-\alpha} e^{2\lambda \xi} u_m^2(\xi)\, \dd \xi 
+
\frac{J_{\lambda}(-\alpha)}{p_{\lambda}(-\alpha)}.
\end{aligned} 
\end{equation}
Using \eqref{Estimate hm} and noticing that $2\lambda \xi +\xi^2 \leq 0$ for all $\xi\in [-2\lambda,0]$, we obtain
\begin{equation} 
\int_{y}^{-\alpha} e^{2\lambda \xi} u_m^2(\xi)\, \dd \xi\lesssim  \int_{-2\lambda}^{-\alpha} e^{2\lambda \xi+\xi^2} (-\xi)^{m^2}\, \dd \xi
\lesssim  \int_{\alpha}^{2\lambda} \xi^{m^2}\, \dd \xi
\lesssim \lambda^{m^2+1}.
\end{equation}
Thus, employing \eqref{Case 4}, we have for all $y\in [ -2\lambda, -\alpha)$ that
\begin{equation}\label{Case 5} 
\frac{J_{\lambda}(y)}{p_{\lambda}(y)}\leq \left(1+\mathcal{O}\left(\frac{1}{|\lambda|^2}\right)\right) 
c_m \frac{ \sqrt \pi }{ 2^{\frac{m^2}{2}}} \frac{e^{\lambda^2}}{\lambda^{m^2}}, \qquad  \lambda \to +\infty.
\end{equation}
\item\label{J.it.vi} $y\in \left(-\infty, -2\lambda\right)$. We write
\begin{equation} \begin{aligned}
J_{\lambda}(y)= c_m 
&
\int_{y}^{-2\lambda} e^{\lambda(\xi-y)}u_m(-y)u_m(\xi)q_{\lambda}(\xi)\, \dd \xi 
\\
& \quad 
+ c_m \int_{-2\lambda}^{\infty} e^{\lambda(\xi-y)}u_m(-y)u_m(\xi)q_{\lambda}(\xi)\, \dd \xi.
\end{aligned}
\end{equation}
For the first integral (recall \eqref{Estimate hm} and \eqref{pq.def})
\begin{equation}\label{Int.1.est}
\begin{aligned}
& \int_{y}^{-2\lambda} e^{\lambda(\xi-y)}u_m(-y)u_m(\xi)q_{\lambda}(\xi)\, \dd \xi
\\
& \quad \lesssim  e^{-2\lambda^2}u_m(-2\lambda)  e^{-\lambda y-\frac{y^2}{2}}(-y)^{-\frac{m^2}{2}}\int_{y}^{-2\lambda} e^{\lambda \xi+\frac{\xi^2}{2}}(-\xi)^{\frac{m^2}{2}}\, \dd \xi
\\
& \quad = e^{-2\lambda^2}u_m(-2\lambda)  e^{-\lambda y-\frac{y^2}{2}}
\int_{y}^{-2\lambda}
\left|\frac{\xi}{y}\right|^\frac{m^2}{2}
\left[-e^{\lambda \xi+\frac{\xi^2}{2}}\right]'
\frac1{ - \xi - \lambda}  \, \dd \xi
\\
& \quad 
\ls \frac{e^{-2\lambda^2}u_m(-2\lambda)}{\lambda}.
\end{aligned}
\end{equation}
For the second integral, notice that $-\lambda y -y^2/2\leq 0$ and $-y\geq 2\lambda$, thus by \eqref{um.asym} in the second step
\begin{equation}\label{Int.2.est}
\begin{aligned}
& c_m \int_{-2\lambda}^{\infty} e^{\lambda(\xi-y)}u_m(-y)u_m(\xi)q_{\lambda}(\xi)\, \dd \xi\\
& \qquad =J_{\lambda}(-2\lambda)\frac{e^{-\lambda y}u_m(-y)}{e^{2\lambda^2}u_m(2\lambda)}\\
& \qquad =\left(1+\mathcal{O}\left(\frac{1}{|\lambda|^2}\right)\right) J_{\lambda}(-2\lambda) e^{-\lambda y-\frac{y^2}{2}} 
\left(\frac{-y}{2\lambda}\right)^{-\frac{m^2}{2}}
\\
& \qquad \leq \left(1+\mathcal{O}\left(\frac{1}{|\lambda|^2}\right)\right)J_{\lambda}(-2\lambda), \quad \la \to + \infty.
\end{aligned} 
\end{equation}
Combining the estimates \eqref{Int.1.est} and \eqref{Int.2.est} and using \eqref{um.asym}, \eqref{Case 5}, we get
\begin{equation} 
\begin{aligned}
\frac{J_{\lambda}(y)}{p_{\lambda}(y)}\leq &\mathcal{O}\left(\frac{e^{-2\lambda^2}u_m(-2\lambda)}{\lambda e^{2\lambda^2}u_m(2\lambda)}\right)+\left(1+\mathcal{O}\left(\frac{1}{|\lambda|^2}\right)\right) \frac{J_{\lambda}(-2\lambda)}{p_{\lambda}(-2\lambda)}\\
\leq & \mathcal{O}\left(\lambda^{m^2-1}\right)+\left(1+\mathcal{O}\left(\frac{1}{|\lambda|^2}\right)\right)
c_m \frac{ \sqrt \pi }{ 2^{\frac{m^2}{2}}} \frac{e^{\lambda^2}}{\lambda^{m^2}}
\\
=&\left(1+\mathcal{O}\left(\frac{1}{|\lambda|^2}\right)\right)
c_m \frac{ \sqrt \pi }{ 2^{\frac{m^2}{2}}} \frac{e^{\lambda^2}}{\lambda^{m^2}},\qquad \lambda\to+\infty.
\end{aligned} 
\end{equation}
\end{enumerate}

Summarising the estimates in \ref{J.it.i}--\ref{J.it.vi}, we conclude that uniformly for all $y\in \R$  
\begin{equation}\label{Schur nontrivial 1}
\int_{\R} \vert \mathcal{R}_{\lambda,2}(y,\xi) \vert_{\C^{2\times 2}} q_{\lambda}(\xi)\, \dd \xi \leq p_{\lambda}(y)\left(1+\mathcal{O}\left(\frac{1}{|\lambda|^2}\right)\right)
c_m \frac{ \sqrt \pi }{ 2^{\frac{m^2}{2}}} 
\frac{e^{\lambda^2}}{\lambda^{m^2}}
\end{equation}
as $\lambda \to +\infty$.
Due to \eqref{anti sym} and $p_{\lambda}(y)=q_{\lambda}(-y)$ for all $y\in \R$, the estimate \eqref{Schur nontrivial 1} implies that for all $\xi \in \R$
\begin{equation}
\int_{\R} \vert \mathcal{R}_{\lambda,2}(y,\xi) \vert_{\C^{2\times 2}} p_{\lambda}(y)\, \dd y\leq q_{\lambda}(\xi)\left(1+\mathcal{O}\left(\frac{1}{|\lambda|^2}\right)\right)
c_m \frac{ \sqrt \pi }{ 2^{\frac{m^2}{2}}} 
\frac{e^{\lambda^2}}{\lambda^{m^2}}
\end{equation}
as $\lambda \to +\infty$. Thus the Schur test yields the upper bound of $\Vert \mathcal{R}_{\lambda,2} \Vert$ in \eqref{R.la.Schur}. 

We do not need to repeat the above procedure for estimating $\Vert \mathcal{R}_{\lambda,1} \Vert$ as $\lambda \to -\infty$. Indeed, by setting 
\begin{equation} 
\wt p_{\lambda}(y)= p_{\lambda}(-y),\qquad \wt q_{\lambda}(\xi) = q_{\lambda}(-\xi),
\end{equation}
and by the definition of the kernels $ \vert \mathcal{R}_{\lambda,1}(y,\xi) \vert_{\C^{2\times 2}}$ and $ \vert \mathcal{R}_{\lambda,2}(y,\xi) \vert_{\C^{2\times 2}}$ in \eqref{Matrix norm}, we have for every $y\in \R$
\begin{equation}
\begin{aligned}
\int_{\R} \vert \mathcal{R}_{\lambda,1}(y,\xi) \vert_{\C^{2\times 2}} \wt q_{\lambda}(\xi)\, \dd \xi & = \int_{\R} \vert \mathcal{R}_{-\lambda,2}(-y,-\xi) \vert_{\C^{2\times 2}}  q_{\lambda}(-\xi)\, \dd \xi\\
& \leq  \wt p_{\lambda}(y)\left(1+\mathcal{O}\left(\frac{1}{|\lambda|^2}\right)\right)
c_m \frac{ \sqrt \pi }{ 2^{\frac{m^2}{2}}} 
\frac{e^{\lambda^2}}{(-\lambda)^{m^2}},
\end{aligned} 
\end{equation}
as $\lambda \to -\infty$. Similarly, for every $\xi \in \R$
\begin{equation} \begin{aligned}
\int_{\R} \vert \mathcal{R}_{\lambda,1}(y,\xi) \vert_{\C^{2\times 2}} \wt p_{\lambda}(y)\, \dd y 
& = \int_{\R} \vert \mathcal{R}_{-\lambda,2}(-y,-\xi) \vert_{\C^{2\times 2}}  p_{\lambda}(-y)\, \dd y\\
& \leq  \wt q_{\lambda}(\xi)\left(1+\mathcal{O}\left(\frac{1}{|\lambda|^2}\right)\right)
c_m \frac{ \sqrt \pi }{ 2^{\frac{m^2}{2}}} 
\frac{e^{\lambda^2}}{(-\lambda)^{m^2}} 
\end{aligned} 
\end{equation}
as $\lambda \to -\infty$. Hence the upper bound of $\Vert \mathcal{R}_{\lambda,1} \Vert$ in \eqref{R.la.Schur} follows again by the Schur test.
\end{proof}

\begin{lemma}\label{lem:R2.est.low}
Let $m>0$ and let $\mathcal{R}_{\lambda,1}$ and $\mathcal{R}_{\lambda,2}$ be integral operators defined as in Lemma \ref{lem:res.norm.form}. We have (with $\la \in \R$)
\begin{equation}
\begin{aligned}
&\Vert \mathcal{R}_{\lambda,2} \Vert \geq \frac{ \Gamma(\frac{m^2}{2}+1)}{2^{\frac{m^2}{2}}m} \frac{e^{\lambda^2}}{\lambda^{m^2}}\left(1+\mathcal{O}\left(\frac{1}{|\lambda|^2}\right)\right),\qquad && \lambda\to +\infty,\\
&\|\mathcal{R}_{\lambda,1}\| \geq  \frac{ \Gamma(\frac{m^2}{2}+1)}{2^{\frac{m^2}{2}}m} \frac{e^{\lambda^2}}{(-\lambda)^{m^2}}\left(1+\mathcal{O}\left(\frac{1}{|\lambda|^2}\right)\right), \qquad && \lambda \to-\infty.
\end{aligned} 
\end{equation}
\end{lemma}
\begin{proof}
To find a lower bound for $\Vert \mathcal{R}_{2,\lambda} \Vert$ as $\lambda\to +\infty$, we consider the function
\begin{equation} \Psi_{\lambda}(\xi)\coloneqq c_m^{\frac 12}\begin{pmatrix}
U\left(\frac{m^2-1}{2},\sqrt{2}\xi\right)\\
\frac{m}{\sqrt{2}} U\left(\frac{m^2+1}{2},\sqrt{2}\xi\right)
\end{pmatrix}e^{\lambda \xi}\mathds{1}_{[\frac{\lambda}{2},2\lambda]}(\xi), \quad \xi \in \R,
\end{equation}
with $c_m$ from \eqref{cm.def}. By the definition \eqref{um.def} and the equalities in \eqref{I.ii.est}, we have
\begin{equation}\label{Norm of vector}
\Vert \Psi_{\lambda} \Vert^2 =c_m \int_{\lambda/2}^{2\lambda} e^{2\lambda \xi}u_m^{2}(\xi)\, \dd \xi
= 
c_m \frac{ \sqrt \pi }{ 2^{\frac{m^2}{2}}}
\frac{e^{\lambda^2}}{\lambda^{m^2}}\left(1+\mathcal{O}\left(\frac{1}{|\lambda|^2}\right)\right), 
\end{equation}
as $\lambda \to +\infty$. A straightforward calculation yields (recall \eqref{R.la12.op}--\eqref{R.la12.def}) 
\begin{equation}
 \left\vert \left(\mathcal{R}_{\lambda,2}\Psi_\la\right)(y)\right\vert_{\C^2}^2
 = c_m^ 3e^{-2\lambda y}u_m^{2}(-y) \left(\int_{y}^{\infty} e^{2\lambda \xi} u_m^2(\xi) \mathds{1}_{[\frac{\lambda}{2},2\lambda]}(\xi) \, \dd \xi\right)^2,
\end{equation}
for $y \in \R$. Then, for any $y \in [-2\lambda, -\lambda/2]$, we have
\begin{equation}\left\vert \left(\mathcal{R}_{\lambda,2}\Psi_\la\right)(y)\right\vert_{\C^2}^2
	= c_m e^{-2\lambda y}u_m^{2}(-y)\Vert \Psi_{\lambda} \Vert^4.
\end{equation}
This yields
\begin{equation} 
\begin{aligned}
\Vert \mathcal{R}_{\lambda,2}\Psi_{\lambda} \Vert^2 \geq & \Vert \Psi_{\lambda} \Vert^4 c_m \int_{-2\lambda}^{- \lambda/2} e^{-2\lambda y}u_m^{2}(-y)\, \dd y\\
=& \Vert \Psi_{\lambda} \Vert^4 c_m 
\int_{\lambda/2}^{2\lambda} e^{2\lambda y}u_m^{2}(y)\, \dd y = \Vert \Psi_{\lambda} \Vert^6,
\end{aligned} 
\end{equation}
and the lower bound of $\|\mathcal{R}_{\lambda,2}\|$ follows by \eqref{Norm of vector}. 

Next, to obtain a lower bound for $\mathcal{R}_{\lambda,1}$ as $\lambda\to -\infty$, we consider 
\begin{equation} \Phi_{\lambda}(\xi)\coloneqq  
c_m^\frac12 
\begin{pmatrix}
-U\left(\frac{m^2-1}{2},-\sqrt{2}\xi\right)\\
\frac{m}{\sqrt{2}} U\left(\frac{m^2+1}{2},-\sqrt{2}\xi\right)
\end{pmatrix}e^{\lambda \xi}\mathds{1}_{[2\lambda,\frac{\lambda}{2}]}(\xi), \quad \xi \in \R.
\end{equation}
By repeating the above procedure, we also have
\begin{equation}
\Vert \mathcal{R}_{\lambda,1}\Phi_{\lambda} \Vert^2 \geq  \Vert \Phi_{\lambda} \Vert^6.
\end{equation}
Noticing that $\Vert \Phi_{\lambda} \Vert^2 = \Vert \Psi_{-\lambda} \Vert^2$, the lower bound of $\Vert \mathcal{R}_{\lambda,1} \Vert$ follows from \eqref{Norm of vector}.
\end{proof}

\begin{proof}[Proof of Theorem~\ref{thm:airy.res.a}]
Putting together \eqref{Am.td.Am.norm} and Lemmas \ref{lem:res.norm.form}, \ref{lem:R1.est}, \ref{lem:R2.est.up} and \ref{lem:R2.est.low}, the resolvent norm estimate in \eqref{Airy.la.inf} is obtained by the triangle inequality.
\end{proof}

\subsection{Resolvent norm of \texorpdfstring{$A_m$}{Am} for varying \texorpdfstring{$m$}{m}}
\label{ssec:Am.m0}

\begin{proof}[Proof of Theorem~\ref{thm:airy.res.m}]
	
While \eqref{eq:Am.small.m} for $m \to 0+$ is the main claim, \eqref{Am.norm.m.1} for $m \approx 1$ follows by a compactness argument and \eqref{Am.norm.m.inf} for $m \to + \infty$ is an easy consequence of \eqref{Am.inv.norm} and the first resolvent identity -- see below for details.

Recall that, throughout the proof, we have $\la \in \R$ and $|\la| \leq R$ with some fixed $R>0$. 
	
\begin{enumerate}[\upshape (i), wide]
\item In this part we restrict the parameters to 
\begin{equation}\label{MR.def}
	\cM_R :=\left \{ (\la,m) \in \R^2 \, : \, |\la| \leq R, m^2\leq \frac 12 \right\}.
\end{equation}

For all $(\la,m) \in \cM_R$
\begin{equation}\label{L.la.sp}
\sigma(\sL_\la + m^2-1 ) = \{ 2n + m^2 \, : \, n \in \N_0\}
\end{equation}
and 
\begin{equation}
\begin{aligned}
(\sL_\la + m^2-1) h_{n,\la} &= (2n + m^2) h_{n,\la}, 
\\
(\sL_\la^* + m^2-1) h_{n,-\la} &= (2n + m^2) h_{n,-\la}, \quad n \in \N_0,		
\end{aligned}
\end{equation}
where $h_{n,\la}$ are the transformed normalized Hermite functions, in particular
\begin{equation}
h_{0,\la}(y) = \frac1{\pi^\frac14} \exp \left( - \frac 12 y^2 - \la y \right), \quad y \in \R;
\end{equation}
see Appendix~\ref{app:HO.conj} for details.

We start from \eqref{A.hat.res} and isolate the singular term. By the basic holomorphic functional calculus (see e.g.~\cite[Chap.~I]{Gohberg-1990} or \cite[Thm.~XII.5]{Reed4}), we decompose 
\begin{equation}\label{L.la.res.decomp}
( \sL_\la + m^2-1)^{-1} = \frac{1}{m^2} P_{\la} + T_{\la,m},
\end{equation}
where
\begin{equation}\label{P.T.def}
\begin{aligned}
P_\la &:= - \frac1{2\pi i} \int_{\Gamma} (\sL_\la + m^2-1-z)^{-1} \, \dd z, 
\\
T_{\la,m} &:= \frac1{2\pi i} \int_{\Gamma} \frac1z (\sL_\la + m^2-1-z)^{-1} \, \dd z, \qquad \Gamma := \{z \in \C \, : \, |z|=1\}.
\end{aligned}	
\end{equation}
Notice that the integration contour $\Gamma$ is selected such that both integrals are well-defined for all $(\la,m) \in \cM_R$. The decomposition \eqref{L.la.res.decomp} can be justified by $(\sL_\la + m^2-1) P_\la = m^2 P_\la$ and 
\begin{equation}\label{L.T.id}
\begin{aligned}
(\sL_\la + m^2-1) T_{\la,m} &  = \frac1{2\pi i} \int_{\Gamma} \frac1z (\sL_\la + m^2-1) (\sL_\la + m^2-1-z)^{-1} \, \dd z
\\
& =  \frac1{2\pi i} \int_{\Gamma} \frac{\dd z} {z} + \frac1{2\pi i} \int_{\Gamma}  (\sL_\la + m^2-1-z)^{-1} \, \dd z
\\
& = I - P_\la;
\end{aligned}
\end{equation}
the argument for the composition $T_{\la,m} (\sL_\la + m^2-1) $ is analogous.

We have $\Ran P_\la = \lspan \{ h_{0,\la}\}$, thus
\begin{equation}\label{P.RanP}
P_\la = \frac{h_{0,\la} \langle \cdot, h_{0,-\la} \rangle }{\langle h_{0,\la}, h_{0,-\la} \rangle } = h_{0,\la} \langle \cdot, h_{0,-\la} \rangle
\end{equation}
and
\begin{equation}\label{P.la.norm}
\|P_\la\| = \|h_{0,\la}\| \|h_{0,-\la}\| = \|h_{0,\la}\|^2 = 
\frac{e^{\la^2}}{\sqrt{\pi}} \int_\R  e^{ - (y+\la)^2} \, \dd y = e^{
		\la^2}.
\end{equation}
It is also important to notice that $(\partial_{y}+y + \la) h_{0,\la} = 0$, hence
\begin{equation}\label{a.P.anih}
(\partial_{y}+y + \la) P_\la = 0,.
\end{equation}
Thus using \eqref{A.hat.res}, \eqref{L.la.res.decomp} and \eqref{a.P.anih}, we obtain

\begin{equation}\label{Am.td.dec}
\begin{aligned}
( \wt A_{m}-\lambda)^{-1}  & = 
\frac1{m} 
\begin{pmatrix}
	0 & 0 \\
	P_\la  & 0
\end{pmatrix}
\\
& \quad +
\begin{pmatrix}
(\partial_{y}+y + \la)T_{\la,m} & m (\sL_\la +m^2+1)^{-1}
\\[1mm]
mT_{\la,m} & ( \partial_{y}-y +\la )(\sL_\la +m^2+1)^{-1}
\end{pmatrix}.
\end{aligned}
\end{equation}
We show below that the entries in the second operator matrix on the r.h.s.~of \eqref{Am.td.dec} are uniformly bounded on $\cM_R$, so the claim \eqref{eq:Am.small.m} follows from \eqref{P.la.norm}.

Let $\cD_{\sL} := \Dom(\sL) = \Dom(\sL_\la) = H^2(\R) \cap \Dom(x^2)$ be equipped with the norm 
\begin{equation}
\|f\|_{\cD_{\sL}} = (\|f'' \|^2 + \|x^2 f\|^2 + \|f\|^2)^\frac12, \quad f \in \cD_{\sL};	
\end{equation}
see Appendix~\ref{sec:app.ho}. Let  
$\iota_{\cD_{\sL} \to L^2}: \cD_{\sL} \to L^2(\R): f \mapsto f$ 
be the canonical embedding of $\cD_{\sL}$ in $L^2(\R)$, which is continuous and has dense range. 
We introduce the bounded operators
\begin{equation}\label{L.lmz.cont}
\wh \sL_{\la,m,z} :  \cD_{\sL} \to L^2(\R) : f \mapsto  (\sL_\la + m^2-1-z) f, \qquad \la, m \in \R, z \in \C;
\end{equation}
i.e., we view $\sL_\la$ as bounded operators acting from their ($\la$-independent) domain $\cD_\sL$ to $L^2(\R)$.

Let $(\la_0,m_0,z_0) \in \R^2 \times \C$ be such that $\sL_{\la_0} + m_0^2-1-z_0$ is boundedly invertible on $L^2(\R)$.  It follows from the graph norm estimate \eqref{D.norm.def} that $\wh \sL_{\la_0,m_0,z_0}$ is boundedly invertible. Indeed, for $f\in L^2(\R)$,
\begin{equation}
\begin{aligned}
\|\wh \sL_{\la_0,m_0,z_0}^{-1} f \|_{\cD_\sL} & \ls 
\|\sL_{\la_0} (\sL_{\la_0} + m_0^2-1-z_0)^{-1} f \|_{L^2} 
\\ & \qquad + \| (\sL_{\la_0} + m_0^2-1-z_0)^{-1} f \|_{L^2} 
\\
& \ls \left(1+ \|(\sL_{\la_0} + m_0^2-1-z_0)^{-1}\|\right) \|f\|_{L^2}.	
\end{aligned}
\end{equation}
It is straightforward to verify that there exists $C>0$ such that
\begin{equation}\label{Cont 1}
\Vert \widehat \sL_{\la,m,z}f - \widehat \sL_{\la_{0},m_{0},z_{0}}f \Vert_{L^2} \leq C \vert (\lambda,m,z)-(\lambda_{0},m_{0},z_{0}) \vert \Vert f \Vert_{D_{\sL}},
\end{equation}
for all $\vert (\lambda,m,z)-(\lambda_{0},m_{0},z_{0}) \vert<1$ and for all $f\in \cD_{\sL}$. By the stability of bounded invertibility (see~\cite[Thm.~IV.1.16]{Kato-1966}), we deduce that $\widehat \sL_{\lambda,m,z}$ is invertible and 
\begin{align}
&\vert (\lambda,m,z)-(\lambda_{0},m_{0},z_{0}) \vert < \min(1, 1/(2C \Vert \widehat\sL_{\la_{0},m_{0},z_{0}}^{-1}\Vert_{L^2\to \cD_{\sL}})) \implies\\
  &\Vert {\widehat\sL}_{\la,m,z}^{-1} -\widehat \sL_{\la_{0},m_{0},z_{0}}^{-1} \Vert_{L^2\to \cD_{\sL}} \leq C' \vert (\lambda,m,z)-(\lambda_{0},m_{0},z_{0}) \vert,
\end{align}
for  some $C' > 0$ (independent of $(\la,m,z)$). This implies that the function defined by $(\lambda,m,z)\mapsto \widehat \sL_{\la,m,z}^{-1}$ is continuous on the set where $\sL_{\la} + m^2-1-z$ is boundedly invertible.
By \eqref{L.la.sp}, $\sL_{\la} + m^2-1-z$ is boundedly invertible for all 
\begin{equation}
(\la,m,z) \in K := \cM_R \times \Gamma.
\end{equation}
Since $K$ is a compact set, there exists $C_K >0$ such that, for all $(\la,m,z) \in K$, we have
\begin{equation}\label{L.unif.bd}
\|\wh \sL_{\la,m,z}^{-1}\|_{L^2 \to \cD_\sL} \leq C_K.	
\end{equation}
Consequently, for all $(\la,m) \in \cM_R$, the operators
\begin{equation}
\wh T_{\la,m} := \frac1{2\pi i} \int_{\Gamma} \frac1z \wh \sL_{\la,m,z}^{-1} \, \dd z : L^2(\R) \to \cD_{\sL},
\end{equation}
are well-defined and uniformly bounded
\begin{equation}
\|\wh T_{\la,m}\|_{L^2 \to \cD_{\sL}} \leq C_K, \quad (\la,m) \in \cM_R.
\end{equation}
Since $T_{\la,m} = \iota_{\cD_{\sL} \to L^2} \wh T_{\la,m}$, we obtain
\begin{equation}
\begin{aligned}
\|m T_{\la,m}\| & \leq m \|\iota_{\cD_{\sL} \to L^2}\|_{{\cD_{\sL} \to L^2}} \|\wh T_{\la,m}\|_{L^2 \to \cD_{\sL}} \leq m C_K,
\\
\|(\partial_{y}+y + \la) T_{\la,m}\| & \leq \|(\partial_{y}+y + \la) \iota_{\cD_{\sL} \to L^2}\|_{{\cD_{\sL} \to L^2}} \|\wh T_{\la,m}\|_{L^2 \to \cD_{\sL}}
\\
& \leq \sqrt{3} (|\la|+1) C_K \leq \sqrt{3} (R+1) C_K,
\end{aligned}
\end{equation}
where, for the second bound, we have used
\begin{equation}
\begin{aligned}
\|(\partial_{y}+y + \la)\iota_{\cD_{\sL} \to L^2} f\|_{L^2} & \leq \|f'\| + \|y f\| + |\la|\|f\| 
\\
& \leq \|f''\|^\frac12 \|f\|^\frac12 + \|y^2 f\|^\frac12 \|f\|^\frac12 + |\la|\|f\|
\\
& \leq \frac12 \|f''\| + \frac12 \|y^2 f\| + (|\la|+1)\|f\|
\\
& \leq \sqrt{3} (|\la|+1) \|f\|_{\cD_{\sL}}, \qquad f \in \cD_{\sL}.
\end{aligned}
\end{equation}

The remaining uniform boundedness of $m (\sL_\la +m^2+1)^{-1}$ and $( \partial_{y}-y +\la )(\sL_\la +m^2+1)^{-1}$ can be similarly verified.

\item In this case, $m \in [m_1,m_2]$ with $m_2>m_1>0$. We sketch a continuity and compactness argument analogous to the one in \eqref{L.lmz.cont} -- \eqref{L.unif.bd}. Let $\cD_{A} : = \Dom(A_1)=H^1(\R, \C^2) \cap \Dom(x I_2)$ be equipped with the graph norm of $A_1$, see \eqref{Separation Airy}. The family of operators $\wh A_{m,\la}:=A_m-\la  \in \cB(\cD_A,  L^2(\R,\C^2))$, $(m,\la) \in (0,\infty) \times \R$, is continuous in these parameters. For every $(m,\la) \in (0,\infty) \times \R$, the operator $A_m- \la$ (and hence $\wh A_{m,\la}$) is boundedly invertible (for $\wh A_{m,\la}$ we use the graph norm estimate \eqref{Separation Airy}). By the stability of bounded invertibility  the function $(m,\la) \mapsto \|\wh A_{m,\la}^{-1}\|_{L^2 \to \cD_A}$ is continuous. Since $\iota_{\cD_A \to L^2}$, the canonical embedding of $\cD_A$ in $L^2(\R)$, is continuous (notice that we have $\|f\|_{L^2} \leq \|f\|_{\cD_A}$ for all $f \in \cD_A$), the function $(m,\la) \mapsto \|A_{m,\la}^{-1}\|_{L^2 \to L^2}$ (with $A_{m,\la}^{-1} := \iota_{\cD_A \to L^2} \wh A_{m,\la}^{-1} = (A_m - \la)^{-1}$) is also continuous and hence it attains a maximum and a (positive) minimum on the compact $[m_1,m_2] \times [-R,R]$. Therefore \eqref{Am.norm.m.1} follows.

\item Recall that $\|A_m^{-1}\| = 1/m$, see \eqref{Am.inv.norm}. As $|\la| \leq R$, the first resolvent identity yields
\begin{equation}
	(A_m - \la)^{-1} = A_m^{-1}\left(I_2 - \la A_m^{-1} \right)^{-1}, \quad m \to + \infty,
\end{equation}
and
\begin{equation}
\|(A_m - \la)^{-1}\| = \frac{1}{m}\left(1 + \BigO(\la m^{-1})\right), \quad m \to + \infty.
\end{equation}
\end{enumerate}

\end{proof}

\subsection{Spectrum of \texorpdfstring{$A_0$}{A0}}
\label{ssec:A0}

\begin{proof}[Proof of Proposition~\ref{prop:A0}]
Let $A = -i\Ntp \sigma_1 + i x I_2$ be the defined on the maximal domain, i.e.~$\Dom(A) = \{ u \in L^2(\R, \C^2) \, : \, A u \in L^2(\R, \C^2)\}$. It is straightforward to check that $A$ is densely defined and closed. 

We clearly have $A_0 \subset A$. We show below that $A \subset A_0$, hence $A = A_0$. To this end, a standard cut-off and mollification procedure yields that $C_{c}^{\infty}(\R,\C^2)$ is a core of $A$. On this core, the graph norm estimate \eqref{Ineq Am} for $A_m$ with $m>0$ reduces to
\begin{equation}\label{A0.gr}
\Vert A \psi \Vert^2+ \Vert \psi \Vert^2 \geq  \Vert \psi' \Vert^2 + \Vert x I_2 \psi\Vert^2, \quad \psi \in  C_{c}^{\infty}(\R,\C^2).
\end{equation}
The \enquote{missing} term with the $L^2$-norm of $\psi$ can be recovered by the Heisenberg inequality. In particular, integration by parts yields 
\begin{equation} 
\Vert \psi \Vert^2= -\langle x I_2 \psi, \psi' \rangle - \langle  \psi',x I_2 \psi \rangle, \qquad \psi \in C_{c}^{\infty}(\R,\C^2),
\end{equation}
thus we have
\begin{equation}\label{Heisenber Ineq}
\Vert x I_2\psi \Vert^2 + \Vert \psi' \Vert^2 
\geq 
2 \|xI_2 \psi\| \|\psi'\| 
\geq \Vert \psi \Vert^2,\qquad  \psi \in C_{c}^{\infty}(\R,\C^2). 
\end{equation}
Consequently,
\begin{equation}\label{A0.gr.1}
\Vert A \psi \Vert^2+ \Vert \psi \Vert^2 \geq \frac12 \left(
\Vert \psi' \Vert^2 + \Vert x I_2 \psi\Vert^2 + \|\psi\|^2
\right), \quad \psi \in C_{c}^{\infty}(\R,\C^2).
\end{equation}
Hence $\Dom(A) \subset H^1(\R^2, \C^2) \cap \Dom(x I_2) = \Dom(A_0)$ which proves $A \subset A_0$.

It is straightforward to see that $A_0^* =(A \restriction C_{c}^{\infty}(\R,\C^2))^* =  -i\Ntp \sigma_1 - i x I_2$ with $\Dom(A_0^*) = \Dom(A) = \Dom(A_0)$, where in the last step we use the equalities of domains shown above.

Let $\lambda\in \C$, then by solving the corresponding system of ODEs, we find that (see \eqref{f.la.def})
\begin{equation}\label{Ker.A0}
\Ker(A_0-\la)=\lspan \left\{ (-f_\la,f_\la) \right \}
\quad
\Ker(A_0^*-\la)= \lspan 
\left\{ 
(f_{-\la},f_{-\la})
\right\}
\end{equation}	
where 
\begin{equation}
f_\la(x) = e^{-\frac{x^2}{2}-i \la x}, \quad x \in \R.
\end{equation}
Since $f_\la \in L^2(\R)$, this implies that $\sigma(A_0)=  \spp(A_0)= \se{5}(A_{0}) = \C$ and the geometric multiplicity of each eigenvalue is $1$. Similar statements hold true for $A_0^*$.

To prove the claims on the remaining essential spectra in \eqref{A0.spe}, we construct a right approximate inverse for $A_0-\la$, with $\la \in \C$ arbitrary. This yields that $A_0-\la$ is Fredholm with deficiency equal to $1$ and \eqref{Ker.A0} justifies that its index is $0$ (see e.g.~\cite[Chap.~I.3]{EE} for details). We search for a right approximate inverse for the unitarily equivalent operator $\wt A_0 -\la$, where 
\begin{equation}
\wt A_0 = \begin{pmatrix}
	-\partial_{y}+y & 0 \\
	0 & -\partial_{y}-y
\end{pmatrix}
\end{equation}
is as in \eqref{A.td.def}, \eqref{A* hat}. In fact, the operator (with $T_{\la,0}$ as in \eqref{P.T.def})
\begin{equation}
B_\la :=
\begin{pmatrix}
	(\partial_{y}+y + \la)T_{\la,0} & 0
	\\[1mm]
	0 & ( \partial_{y}-y +\la )(\sL_\la +1)^{-1}
\end{pmatrix},
\end{equation}
is a right inverse of $\wt A_0 -\la$ as it satisfies (see \eqref{L.T.id})
\begin{equation}
\begin{aligned}
(\wt A_0 -\la) B_\la = 
\begin{pmatrix}
	(\sL_\la -1) T_{\la,0} & 0
	\\[1mm]
	0 & (\sL_\la +1)(\sL_\la +1)^{-1}
\end{pmatrix} = 
\begin{pmatrix}
I - P_\la & 0 
\\	
0 & I 
\end{pmatrix},
\end{aligned}
\end{equation}
with $P_\la$ as in \eqref{P.T.def} (recall also \eqref{P.RanP}). Analogously, one constructs a right approximate inverse for $A_0^*-\la$ to justify the remaining claims for $\se{k}(A_0^*)$.

Finally, the claim regarding the algebraic multiplicity of the eigenvalues follows by applying \cite[Thm.~2.2.16]{Locker-2000-73} to $\C$, the Fredholm set of $A_0$, and noting that, from the above arguments, there cannot be any eigenvalues with algebraic multiplicity equal to $0$.
\end{proof}

\section{Resolvent estimates for general \texorpdfstring{$V$}{V} -- proofs}
\label{sec:resnorm.est.iR}
In this section we prove Theorem~\ref{thm:iR}.
Let
\begin{equation}\label{eq:Hb.def}
	H_b := H - \la_b.
\end{equation}
The proof is structured in four steps. Firstly, we prove the claim \enquote{away} from the zero $x_b$ of $V - b$. Then we study the behaviour of the norm of the resolvent locally (\ie~near $x_b$). Next we establish a lower bound for the norm. Our final step, the theorem proof proper, combines the previously derived estimates. Throughout the proof we are chiefly concerned with the behaviour as $b \rightarrow +\infty$ and will therefore assume $b$ to be as large as needed for our assumptions to hold without further comment.

Let
\begin{equation}\label{eq:Deltab.def}
\Omega'_b := \left( x_b - \Delta_b, x_b + \Delta_b \right), \quad \Delta_b := \delta_b x_b,
\quad 0 < \delta_b :=
\frac{1}{(V'(x_b) x_b)^{\frac12}} 
\end{equation}
with $x_b$ as defined in \eqref{eq:xb.def}. Note that by \eqref{V'.nu.asm} and \eqref{V.unbd}, we have
\begin{equation}\label{del.b.asym}
\delta_b \approx \frac{1}{b^\frac12} = o(1), \quad b \to +\infty,
\end{equation}
and consequently,
\begin{equation}\label{eq:xb.db}
	x_b - 2 \Delta_b = x_b \left( 1 - 2 \delta_b \right) \gtrsim x_b \to + \infty, \qquad b \rightarrow +\infty.	
\end{equation}
In what follows, we shall also assume $b$ to be large enough so that $x_b - 2 \Delta_b > x_0$ and $V(x_b - 2 \Delta_b) > 0$. This ensures that $V(x) > 0 \text{ for all } x > x_b - 2 \Delta_b$.

Furthermore, in the proofs below, we will use that the above choice of $\Delta_b$ implies that $V(x)$ and $V'(x)$ are approximately equal to $V(x_b)$ and $V'(x_b)$, respectively, inside $\Omega_b'$ when $b$ is large enough. More precisely,
for all $x \in \Omega_b'$, we have
\begin{equation}\label{V.V'.approx}
\frac{V(x)}{V(x_b)} \approx 1, \quad \frac{V'(x)}{V'(x_b)} \approx 1, \quad b \to + \infty; 
\end{equation}
see \cite[Sec.~22.27]{Titchmarsh-1958-book2} or \cite[Lem.~4.1]{Mityagin-2022-28}. 

Finally, let $A_m$ be as in \eqref{eq:airy.dirac.def}. For any $\mu \in \R$, we extend the graph norm estimates in \eqref{Separation Airy} to
\begin{equation}\label{Am.gr.mu}
		\| (A_m - \mu) u \|^2 + \langle \mu \rangle^2 \| u \|^2 \geq \frac12 \left(\|  u' \|^2 + \left\| x u \right\|^2 + m^2 \| u \|^2 \right), \quad u \in \Dom(A_m).
	\end{equation}
Indeed, by Cauchy-Schwarz and Young inequalities	
\begin{equation}
\begin{aligned}
\| (A_m - \mu) u \|^2 + \langle \mu \rangle^2 \| u \|^2 & = 
\| A_m u \|^2 - 2 \mu \Re \langle A_m u,u \rangle + (2\mu^2+1) \| u \|^2 
\\
& \geq  
\| A_m u \|^2 - 2 \mu \|A_m u\| \|u\| + (2\mu^2+1) \| u \|^2 
\\
& \geq \frac 12 \| A_m u \|^2 + \|u\|^2,
\end{aligned}	
\end{equation}	
thus \eqref{Am.gr.mu} follows from \eqref{Separation Airy}. We deduce that
	\begin{equation}\label{eq:Am.inv.x.graphn}
		\| x (A_m - \mu)^{-1} \|^2 \leq 2 + 2 \langle \mu \rangle^2 \| (A_m - \mu)^{-1} \|^2.
	\end{equation}

\subsection{Step 1: estimate outside the neighbourhood of \texorpdfstring{$x_b$}{xb}}
\label{ssec:step.1.iR}
\begin{proposition}
	\label{prop:away.iR}
	Let $\Omega'_b$ be defined by~\eqref{eq:Deltab.def}, let the assumptions of Theorem~\ref{thm:iR} hold, let $H_b$ be as in \eqref{eq:Hb.def} and let $\delta_b$ be as in \eqref{eq:Deltab.def}. Then we have as $b \to +\infty$
	\begin{equation}
	b^\frac12 \approx V'(x_b) \Delta_b \ls \inf \left\{ \frac{\| H_b u \|}{\| u \|}: \; 0 \neq u \in \Dom(H), \; \supp u \cap \Omega'_b = \emptyset\right\}.
	\end{equation}
\end{proposition}
\begin{proof}
	Define $\chi_b(x) := \sgn(V(x) - b), x \in \R,$ and note that $\| \chi_b \|_{\infty} \le 1$ and $\chi'_b(x) = 0, \; x \in \R \setminus \Omega'_b$. Let $u = (u_1, u_2)^t \in \Dom(H)$ such that $\supp u \cap \Omega'_b = \emptyset$, then
	\begin{equation}
		\begin{aligned}
			\langle \chi_b H_b u, u \rangle &= \langle (m - a + i (V - b)) u_1 - i u_2', \chi_b u_1 \rangle\\
			&\quad - \langle i u_1' + (m + a - i (V - b)) u_2, \chi_b u_2 \rangle\\
			&= (m-a) \langle u_1, \chi_b u_1 \rangle - (m+a) \langle u_2, \chi_b u_2 \rangle + 2 \Im \langle u_1', \chi_b u_2 \rangle\\
			&\quad + i (\langle |V - b| u_1, u_1 \rangle + \langle |V - b| u_2, u_2 \rangle).
		\end{aligned}
	\end{equation}
	Therefore
	\begin{equation}\label{eq:lboundstep1}
		\langle |V - b| u_1, u_1 \rangle + \langle |V - b| u_2, u_2 \rangle = \Im \langle \chi_b H_b u, u \rangle \le \| H_b u \| \| u \|.
	\end{equation}
	Next we find a lower bound for $|V(x) - V(x_b)|$ in $\R \setminus \Omega'_b$. By Assumption~\ref{asm:V.iR}, we have that $V$ is bounded above in $(-\infty, x_0]$ and unbounded and increasing in $(x_0, +\infty)$. It follows that, for large enough $b$, we obtain for all $x \in \R \setminus \Omega'_b$ that 
	\begin{equation}
		|V(x) - V(x_b)| \ge \min\{V(x_b+\Delta_b) - V(x_b), V(x_b) - V(x_b-\Delta_b)\}.
	\end{equation}
	Applying the mean-value theorem to the first term inside the $\min$ with $\xi_b \in (x_b, x_b+\Delta_b)$, and noting that $V'(\xi_b) \approx V'(x_b)$ by \eqref{V.V'.approx}, we deduce (for $b \to +\infty$)
	\begin{equation}
		|V(x_b+\Delta_b) - V(x_b)| = V'(\xi_b) \Delta_b \approx V'(x_b) \Delta_b .
	\end{equation}
	A similar result can be found for $|V(x_b-\Delta_b) - V(x_b)|$. Therefore
	\begin{equation}\label{eq:ImVlbound}
		|V(x) - b| \gs V'(x_b) \Delta_b, \qquad x \in \R \setminus \Omega'_b, \qquad b \to +\infty.
	\end{equation}
	Hence, by combining \eqref{eq:ImVlbound} and \eqref{eq:lboundstep1} and applying \eqref{eq:Deltab.def}, \eqref{del.b.asym}, we conclude that for all $u \in \Dom(H)$ with $\supp u \cap \Omega'_b = \emptyset$
	\begin{equation}
	b^\frac12 \|u\| \approx V'(x_b) \Delta_b \| u \| 	\ls \| H_b u \|, \quad b \to +\infty,
	\end{equation}
	which completes the proof.
\end{proof}

\subsection{Step 2: estimate near \texorpdfstring{$x_b$}{xb}}
\label{ssec:step.2.iR}
\begin{proposition}
	\label{prop:local.iR}
	Let the assumptions of Theorem~\ref{thm:iR} hold, let $H_b$ be as in \eqref{eq:Hb.def} and define 
	\begin{equation}\label{eq:Omega.def}
		\Omega_b := \left( x_b - 2 \Delta_b, x_b + 2 \Delta_b \right),
	\end{equation}
	with $\Delta_b$ as defined in \eqref{eq:Deltab.def}. Then we have as $b \to + \infty$
	\begin{equation}\label{eq:Hb.res.norm.local}
		\begin{aligned}
			&(V'(x_b))^{\frac12} \| (A_{m_b} - \mu_b)^{-1} \|^{-1} \left(1 - \BigO(b^{-\frac12} +x_b^{-\frac12}) \right)  \le\\
			&\hspace{3cm} \inf \left\{ \frac{\left\| H_b u \right\|}{\|u\|}: \; 0 \neq u \in \Dom(H), \; \supp u \subset \Omega_b \right\},
		\end{aligned}	
	\end{equation}
	where $A_m$ is as in \eqref{eq:airy.dirac.def} and $m_b$ as in~\eqref{mb.def}.
\end{proposition}
\begin{proof}
	If $x \in \Omega_b$, the Taylor expansion of $V$ around $x_b$ yields
	\begin{equation}
		V(x) - V(x_b) = V'(x_b)\left( x - x_b \right) + \frac{1}{2} V''(x_b + s (x - x_b)) \left( x - x_b\right)^2,
	\end{equation}
	where $s = s(x,b) \text{ and } 0 < s < 1$. Let
	\begin{equation}\label{eq:V.hat.def}
		\widetilde V_b (x) := V'(x_b)\left( x - x_b \right) + \frac{1}{2} V''(x_b + s (x - x_b)) \left( x - x_b\right)^2 \chi_{\Omega_b}(x), \quad x \in \R,
	\end{equation}
	and consider the operator in $\Lt(\R, \C^2)$
	\begin{equation}\label{eq:Hb.hat.def}
		\widetilde H_b= -i \Ntp \sigma_1 + m \sigma_3 + (i \widetilde V_b - a) I_2, \quad \Dom(\widetilde H_b) = H^1(\R, \C^2) \cap \Dom(x).
	\end{equation}
	Given $\rho > 0$, we define a unitary operator on $L^2(\R, \C^2)$ by
	$
	(U_{b,\rho} u)(x) := \rho^{\frac{1}{2}} u(\rho x + x_b)$, $x \in \R$.
	Then for any $u \in U_{b,\rho} ( \Dom(\widetilde H_b) )$ and all $x \in \R$
	\begin{equation}
		( U_{b,\rho} \widetilde H_b U_{b,\rho}^{-1} u )(x) = -i \rho^{-1} \Ntp \sigma_1 u(x) + m \sigma_3 u(x) + (i \widetilde V_b(\rho x + x_b) - a) u(x).
	\end{equation}
	Choosing the value of $\rho$ for sufficiently large $b > 0$
	\begin{equation}
		\label{eq:rho.def}
		\rho := (V'(x_b))^{-\frac12},
	\end{equation}
	letting
	\begin{equation}\label{eq:Omegabrho.def}
		\Omega_{b,\rho} := (-2 \Delta_b \rho^{-1}, 2 \Delta_b \rho^{-1}) = (-2 x_b^\frac12, 2 x_b^\frac12)
	\end{equation}
	and taking $x \in \R$, then
	\begin{equation} \begin{aligned}
		V_b(x) &:= \rho \widetilde V_b(\rho x + x_b) 
		= V'(x_b) \rho^2 x + \frac{1}{2} V''(\tilde{s} \rho x + x_b) \rho^3 x^2 \chi_{\Omega_{b,\rho}}(x)\\
		&=  V'(x_b) \rho^2 \left(x + \frac{1}{2} \frac{V''(\tilde{s} \rho x + x_b)}{V'(x_b)} \rho x^2 \chi_{\Omega_{b,\rho}}(x)\right),
	\end{aligned} 
	\end{equation}		
	where $0 < \tilde s < 1$. Calling
	\begin{equation}\label{eq:Rb.def}
		R_b(x) := \frac12 \frac{V''(\tilde{s} \rho x + x_b)}{V'(x_b)} \rho x^2 \chi_{\Omega_{b,\rho}}(x), \quad x \in \R,
	\end{equation}
	we have
	\begin{equation}\label{eq:Vb.def}
		V_b(x) = x + R_b(x), \quad x \in \R.
	\end{equation}
	By Assumption~\ref{asm:V.iR}~\ref{itm:nu.asm}, and arguing as in the proof of \cite[Prop.~3.4]{ArSi-2023-284} (note that $\delta_b \le 1/4$ for sufficiently large $b$), we deduce
	\begin{equation}
		\left| \frac{V''(\tilde{s} \rho x + x_b)}{V'(x_b)} \right| \ls x_b^{-1}, \quad x \in \Omega_{b,\rho}, \quad  b \to +\infty.
	\end{equation}
	For all $x \in \Omega_{b,\rho}$ we have $|\rho x|\ls \delta_b x_b$ and therefore
	\begin{equation}\label{eq:Rbdecay}
		\| x^{-1} R_b \|_{\infty} \ls \delta_b, \qquad b \to +\infty.
	\end{equation}

	Let $S_b$ be the operator in $\Lt(\R, \C^2)$
	\begin{equation}\label{eq:Sb.def}
		S_b = -i \Ntp \sigma_1 + m_b \sigma_3 + i V_b I_2, \quad \Dom(S_b) = H^1(\R, \C^2) \cap \Dom(x I_2);
	\end{equation}
	notice that $m_b = \rho m$, $\mu_b = \rho a$ and
	\begin{equation}\label{eq:Sb.def.Am}
	S_b - \mu_b = A_{m_b} - \mu_b + i R_b I_2 = (I_2 + i R_b (A_{m_b} - \mu_b)^{-1}) (A_{m_b} - \mu_b).
	\end{equation}
	Our next aim is to prove that $\| (S_b - \mu_b)^{-1} - (A_{m_b} - \mu_b)^{-1} \| = o(1)$ as $b \to +\infty$.
	
	We begin by showing that  $\mu_b \in \rho(S_b)$ for any sufficiently large $b$.  Recall that $|\mu_b| \ls 1$ as $b \to + \infty$ by the assumption~\eqref{mu.b.def}, so from \eqref{eq:Rbdecay} and \eqref{eq:Am.inv.x.graphn}, we get 
	\begin{equation}
	\label{eq:Rb.x}
		\begin{aligned}
			\| R_b (A_{m_b} - \mu_b)^{-1} \| &\le \| x^{-1} R_b \|_{\infty}  \| x (A_{m_b} - \mu_b)^{-1} \| 
			\\
			&
			\ls \delta_b (1 + \langle \mu_b \rangle \| (A_{m_b} - \mu_b)^{-1} \|)
			\\
			& 
			\ls \delta_b (1 + \| (A_{m_b} - \mu_b)^{-1} \|), \quad b \to +\infty.
		\end{aligned}
	\end{equation}
Thus, recalling \eqref{Am.res.sum}, \eqref{mb.def}, \eqref{eq:Deltab.def} and \eqref{del.b.asym}, we obtain that
\begin{equation}\label{eq:Rb.x1}
\delta_b (1 + \| (A_{m_b} - \mu_b)^{-1} \|) 
\ls 
\delta_b + \delta_b m_b^{-1}
\ls 
b^{-\frac12} + x_b^{- \frac 12}, \quad b \to + \infty.
\end{equation}
Returning to \eqref{eq:Rb.x}, we get
\begin{equation}
\| R_b (A_{m_b} - \mu_b)^{-1} \| \leq \frac 12, \quad b \to +\infty,
\end{equation}
so $I_2 + i R_b (A_{m_b} - \mu_b)^{-1}$ is boundedly invertible and $\| (I_2 + i R_b (A_{m_b} - \mu_b)^{-1})^{-1} \| \leq 2$ as $b \to +\infty$.
It follows that $\mu_b \in \rho(S_b)$ and
\begin{equation}\label{eq:Sb.inv.local.iR}
(S_b - \mu_b)^{-1} = (A_{m_b} - \mu_b)^{-1} (I_2 + i R_b (A_{m_b} - \mu_b)^{-1})^{-1}, \quad b \to +\infty.
\end{equation}
In particular (see also \eqref{eq:Am.inv.x.graphn}), we have  
\begin{equation}
\label{eq:Sb.norm.est}
\begin{aligned}
  \| (S_b - \mu_b)^{-1} \| +  \| x (S_b - \mu_b)^{-1} \| &  \ls  1 + \langle \mu_b \rangle \| (A_{m_b} - \mu_b)^{-1} \| 
  \\ & \ls 1 + m_b^{-1}, \qquad b \to +\infty.
\end{aligned}
\end{equation}

Using the second resolvent identity, \eqref{eq:Sb.def.Am}, \eqref{eq:Rbdecay}, \eqref{eq:Sb.norm.est} and \eqref{eq:Rb.x1}, we obtain
\begin{equation}\label{eq:Sb.Sinf}
		\begin{aligned}
	&	\| (S_b - \mu_b)^{-1} - (A_{m_b} - \mu_b)^{-1} \| 
	\\
	& \quad = \| (A_{m_b} - \mu_b)^{-1}  R_b (S_b - \mu_b)^{-1}  \|
	\\
	& \quad \le \| (A_{m_b} - \mu_b)^{-1} \|  \| x^{-1} R_b \|_{\infty}  \| x (S_b - \mu_b)^{-1} \| 
	\\
	& \quad \ls \| (A_{m_b} - \mu_b)^{-1} \| \delta_b (1 + m_b^{-1}), 
	\\
	& \quad \ls \| (A_{m_b} - \mu_b)^{-1} \| ( b^{-\frac12} + x_b^{- \frac 12} ), \quad b \to +\infty.
	\end{aligned}
	\end{equation}
	Therefore
	\begin{equation}
		\label{eq:Sb.mub.inv.norm}
		\begin{aligned}
			\| (S_b - \mu_b)^{-1} \| = \| (A_{m_b} - \mu_b)^{-1} \| \left(1 + \BigO( b^{-\frac12} + x_b^{- \frac 12} ) \right), \quad b \to +\infty.
		\end{aligned}
	\end{equation}
	But $S_b - \mu_b = \rho U_{b,\rho} \widetilde H_b U_{b,\rho}^{-1}$ and hence
	\begin{equation}
		\rho^{-1} \| \widetilde H_b^{-1} \| = \| (A_{m_b} - \mu_b)^{-1} \| 
		\left(1 + \BigO( b^{-\frac12} + x_b^{- \frac 12} ) \right), \quad b \to +\infty.
	\end{equation}
	With $b$ sufficiently large, let $u \in \Dom(H)$ such that $\supp u \subset \Omega_b$. Then $u \in \Dom( \widetilde H_b)$ and $\|\widetilde H_b u \| = \| H_b u\|$. With $v := \widetilde H_b u \in \Lt(\R, \C^2)$, we have
	\begin{equation}
		\begin{aligned}
			\rho^{-1} \| u \| = \rho^{-1} \| \widetilde H_b^{-1} v \| \le \| (A_{m_b} - \mu_b)^{-1} \| \left(1 + \BigO( b^{-\frac12} + x_b^{- \frac 12} ) \right) \left\| H_b u \right\|,
		\end{aligned}
	\end{equation}
	which yields \eqref{eq:Hb.res.norm.local}.
\end{proof}

\subsection{Step 3: pseudomode}
\label{ssec:step.3.iR}
Throughout this sub-section, we shall retain the notation introduced in the proof of Proposition~\ref{prop:local.iR}. In particular, $\rho$ is as in \eqref{eq:rho.def}. 
\begin{lemma}
	Let the assumptions of Theorem~\ref{thm:iR} hold, let $S_b$ be as defined in \eqref{eq:Sb.def} and denote (recall \eqref{eq:Deltab.def}, \eqref{eq:rho.def})
	\begin{equation}\label{eq:Omegabrhop.def}
		\Omega_{b,\rho}' := (-\Delta_b \rho^{-1}, \Delta_b \rho^{-1}).
	\end{equation}
	Then for any $u \in \Dom(S_b)$ such that $\supp u \cap \Omega_{b,\rho}' = \emptyset$, we have
	\begin{equation}\label{eq:Sb.away.step3}
		\| u \| \ls x_b^{-\frac 12} \| (S_b - \mu_b) u \|, \quad b \to +\infty.
	\end{equation}
\end{lemma}
\begin{proof}
For $V_b(x)$ as in \eqref{eq:Vb.def}, define $\chi_b(x) := \sgn(V_b(x))$, $x \in \R$, and observe that $\| \chi_b \|_{\infty} \le 1$ and
\begin{equation}\label{chi'.0}
\chi'_b(x) = 0, \quad x \in \R \setminus \Omega_{b,\rho}', \quad b \to +\infty.	
\end{equation}
To see \eqref{chi'.0},
notice that for $x \notin \Omega_{b,\rho}$ (refer to~\eqref{eq:Omegabrho.def}), we have $V_b(x) = x$, so $\chi_b(x) =1 $ for $x> 2\Delta_b \rho^{-1}$ and $\chi_b(x) =- 1 $ for $x < - 2\Delta_b \rho^{-1}$.
For any $x \in \Omega_{b,\rho} \cap \Omega_{b,\rho}'^c$, we have by \eqref{eq:Vb.def} and \eqref{eq:Rbdecay} that
\begin{equation}
\begin{aligned}
V_b(x) & = x + R_b(x) = x(1 + x^{-1} R_b(x)) = x (1 + \BigO(\delta_b)), \quad b \to + \infty.
\end{aligned}
\end{equation}
Since $\delta_b = o(1)$, it follows that $\chi'_b(x) = 0$ as $b \to + \infty$.

Let $u = (u_1, u_2)^t \in \Dom(S_b)$ such that $\supp u \cap \Omega_{b,\rho}' = \emptyset$. Then, arguing as in the proof of Proposition~\ref{prop:away.iR}, we find
	\begin{equation}
	\begin{aligned}
		\langle \chi_b (S_b - \mu_b) u, u \rangle &= \rho ((m - a) \langle u_1, \chi_b u_1 \rangle - (m + a) \langle u_2, \chi_b u_2 \rangle) \\
			&\quad + 2 \Im \langle u_1', \chi_b u_2 \rangle + i (\langle |V_b| u_1, u_1 \rangle + \langle |V_b| u_2, u_2 \rangle).
		\end{aligned}
	\end{equation}
	Therefore
	\begin{equation}\label{eq:Vb.Sb.qf}
		\langle |V_b| u_1, u_1 \rangle + \langle |V_b| u_2, u_2 \rangle = \Im \langle \chi_b (S_b - \mu_b) u, u \rangle \le \| (S_b - \mu_b) u \| \| u \|.
	\end{equation}
	Using \eqref{eq:Vb.def} and \eqref{eq:Rbdecay} once more, we obtain for any $x \in \R \setminus \Omega_{b,\rho}'$
	\begin{equation}\label{eq:Vb.away.llim}
		|V_b(x)| \ge |x| (1 - |x^{-1} R_b(x)|) \gs |x| \ge \Delta_b \rho^{-1}, \quad b \to +\infty.
	\end{equation}
	Combining \eqref{eq:Vb.Sb.qf} and \eqref{eq:Vb.away.llim}, we conclude that for any $u \in \Dom(S_b)$ with $\supp u \cap \Omega_{b,\rho}' = \emptyset$
	\begin{equation}
	 x_b^\frac12 \| u \| =	\Delta_b \rho^{-1} \| u \| \ls \| (S_b - \mu_b) u \|, \quad b \to +\infty;
	\end{equation}
	for the first equality recall \eqref{eq:Deltab.def} and \eqref{eq:rho.def}.
\end{proof}
With $\Omega_{b,\rho}, \Omega_{b,\rho}'$ as in \eqref{eq:Omegabrho.def}, \eqref{eq:Omegabrhop.def}, respectively, let $\psi_b \in C_c^{\infty}(\Omega_{b,\rho})$, $0 \le \psi_b \le 1$, be such that
\begin{equation}\label{eq:psib.def}
	\psi_b(x) = 1, \; x \in \Omega_{b,\rho}' = (-x_b^{\frac 12},x_b^{\frac 12} ), \quad \| \psi_b^{(j)} \|_{\infty} \ls (\Delta_b \rho^{-1})^{-j} = x_b^{-\frac j2},
\end{equation}
where $ j \in \{1, \dots, n_0+1\}$ with $n_0$ from \eqref{n0.def}. Notice that $\psi_b \rightarrow 1$ pointwise in $\R$ as $b \to +\infty$.
\begin{lemma}
	\label{lem:Sb.psib.comm.est}
	Let the assumptions of Theorem~\ref{thm:iR} hold, let $\rho$ be as in \eqref{eq:rho.def}, let $S_b$ be as defined in \eqref{eq:Sb.def} and $\psi_b$ as in \eqref{eq:psib.def}. Then, for any $u \in \Dom(S_b)$ such that $\| (S_b - \mu_b) u \| \approx \rho \| u \|$ as $b \to +\infty$, we have
	\begin{equation}
		\| [S_b, \psi_b I_2] u \| \ls \rho x_b^{-\frac 12} \| u \|, \quad b \to +\infty.
	\end{equation}
\end{lemma}
\begin{proof}
It is straightforward to find that
\begin{equation}\label{eq:Sb.comm}
[S_b, \psi_b I_2] u = [-i \Ntp \sigma_1, \psi_b I_2] u +  [m_b \sigma_3,\psi_b I_2] u + [i V_b I_2,\psi_b I_2] u   = -i \sigma_1 \psi_b' u.
\end{equation}
Moreover, $\supp(\psi_b' u) \cap \Omega_{b,\rho}' = \emptyset$ and therefore, by \eqref{eq:Sb.away.step3}, we get
	\begin{equation}
		\| \psi_b' u \| \ls x_b^{-\frac 12} \| (S_b - \mu_b) \psi_b' u \|.
	\end{equation}
	Since
	\begin{equation}\label{eq:Sb.psib.comm.exp}
		(S_b - \mu_b) \psi_b' u = \psi_b' (S_b - \mu_b) u + [S_b, \psi_b' I_2] u = \psi_b' (S_b - \mu_b) u - i \sigma_1 \psi_b'' u,
	\end{equation}
	applying \eqref{eq:psib.def} and our assumption $\| (S_b - \mu_b) u \| \approx \rho \| u \|$, we deduce
	\begin{equation}\label{eq:Sb.psib.comm.order1}
		\begin{aligned}
			\| [S_b, \psi_b I_2] u \| = \| \psi_b' u \| & \ls x_b^{-\frac12} \left(
			x_b^{-\frac12} \| (S_b - \mu_b) u \| + \| \psi_b'' u \|
			\right) 
			\\ 
			& \ls x_b^{-1} \rho \| u \| + x_b^{-\frac 12} \| \psi_b'' u \|.
		\end{aligned}
	\end{equation}
	We note that $\supp(\psi_b^{(j)} u) \cap \Omega_{b,\rho}' = \emptyset$, $1 \le j \le n_0$, and we can therefore repeatedly apply the above process of using \eqref{eq:Sb.away.step3} and a commutator expansion as in \eqref{eq:Sb.psib.comm.exp} to find
	\begin{equation}
		\begin{aligned}
		\| \psi_b^{(j)} u \| &\ls x_b^{-\frac 12} \| (S_b - \mu_b) \psi_b^{(j)} u \| \ls x_b^{- \frac12} (\| \psi_b^{(j)} (S_b - \mu_b) u \| + \| \psi_b^{(j+1)} u \|)\\
			&\ls x_b^{-\frac{j+1}2 } \rho \| u \| + x_b^{-\frac 12} \| \psi_b^{(j+1)} u \|, \quad b \to +\infty,
		\end{aligned}
	\end{equation}
	where, in the last inequality, we have used the assumption $\| (S_b - \mu_b) u \| \approx \rho \| u \|$. Hence, returning to \eqref{eq:Sb.psib.comm.order1}, we obtain (recall \eqref{eq:rho.def})
	\begin{equation}
		\begin{aligned}
			\| [S_b, \psi_b I_2] u \| &\ls \left(\sum_{k = 1}^{n_0} \rho x_b^{- k }\right) \| u \| + x_b^{- \frac{n_0}2} \| \psi_b^{(n_0+1)} u \|
			\\
			&\ls \left(\rho x_b^{-1} + x_b^{-n_0 - \frac12}\right) \| u \|
			= \rho x_b^{- \frac12} \left( x_b^{-\frac 12 } +  \frac{V'(x_b)^{\frac12}}{x_b^{n_0}}  \right) \| u \| 
			\\
			&\ls \rho x_b^{-\frac 12} \| u \|, \quad b \to +\infty,
		\end{aligned}
	\end{equation}
	where in the last step we use \eqref{n0.def}.
\end{proof}
\begin{proposition}	
	\label{prop:lbound.iR}
	Let the assumptions of Theorem~\ref{thm:iR} hold and let $H_b$ be as in \eqref{eq:Hb.def}. Then there exist functions $0 \neq u_b \in \Dom(H)$ such that
	\begin{equation}
		\| H_b u_b \| = \| (A_{m_b} - \mu_b)^{-1} \|^{-1} (V'(x_b))^{\frac12} (1 + \BigO(x_b^{-\frac12} + b^{-\frac12})) \|u_b\|, \quad b \to + \infty.
	\end{equation}
\end{proposition}
\begin{proof}
	With a sufficiently large $b_0 > 0$, the operators $B_b := ( (S_b - \mu_b)^* (S_b - \mu_b))^{-1}$, $b \in (b_0,\infty)$, on $\Lt(\R)$, are compact, self-adjoint and non-negative. Let $0 < \varsigma_b^2 := \rad(B_b) = \max\{z: z \in \sigma(B_b)\}$ and let $g_b \in \Lt(\R, \C^2)$ be a corresponding normalised eigenfunction, \ie~$\| B_b \| = \varsigma_b^2$, $B_b g_b = \varsigma_b^2 g_b$ and $\| g_b \| = 1$. Note that $g_b \in \Dom((S_b - \mu_b)^*(S_b - \mu_b))$ and it is straightforward to verify that
	\begin{equation}\label{eq:la*g*}
		\| (S_{b} - \mu_b) g_b\| = \varsigma_b^{-1} = \| (S_b - \mu_b)^{-1} \|^{-1}= \| B_b \|^{-\frac12}, \quad b \in (b_0, \infty).
	\end{equation}
	Moreover, letting $\tilde\varsigma_b := \| (A_{m_b} - \mu_b)^{-1} \| \approx m_b^{-1} \approx \rho^{-1}$, see \eqref{Am.res.sum}, \eqref{mb.def}, \eqref{eq:rho.def}, and applying \eqref{eq:Sb.mub.inv.norm}, we obtain 
	\begin{equation}\label{eq:lab.conv}
	\varsigma_b = \tilde\varsigma_b \left(1 + \BigO(x_b^{-\frac12} + b^{-\frac12}) \right) \approx \rho^{-1}, \quad b \to +\infty.
	\end{equation}
	Note also that, similarly to \eqref{eq:Sb.norm.est}, we have
	\begin{equation}\label{eq:Sb*.gn}
		\| x (S_b^* - \mu_b)^{-1} \| \ls 1 + \rho^{-1}, \quad b \to + \infty.
	\end{equation}

	With $\psi_b$ as in \eqref{eq:psib.def}, then $\psi_b g_b \in \Dom(S_b)$ and we have
	\begin{equation}\label{eq:Sb.psib.comm}
		(S_b - \mu_b) \psi_b g_b = (S_b - \mu_b) g_b + (\psi_b-1) (S_b - \mu_b) g_b + [S_b, \psi_b I_2] g_b.
	\end{equation}
	The second term in \eqref{eq:Sb.psib.comm} can be  estimated using \eqref{eq:psib.def}, \eqref{eq:Sb*.gn} and the fact that $\| B_b^{-1} g_b \| = \varsigma_b^{-2}$ as follows
	\begin{equation}
		\begin{aligned}
			&\|(\psi_b-1) (S_b - \mu_b) g_b \|
			\\ & \quad \le \|(\psi_b-1) x^{-1}\|_\infty \|x (S_b^* - \mu_b)^{-1}\| \| (S_b^* - \mu_b) (S_b - \mu_b) g_b \|
			\\
			 & \quad \ls x_b^{-\frac12} (1 + \rho^{-1}) \varsigma_b^{-2} \ls \rho (x_b^{-\frac12} + b^{-\frac12}), \quad b \to +\infty,
		\end{aligned}
	\end{equation}
	where in the last step we have used \eqref{eq:lab.conv} and Assumption~\ref{asm:V.iR}~\ref{itm:nu.asm}. For the last term in \eqref{eq:Sb.psib.comm}, we apply Lemma~\ref{lem:Sb.psib.comm.est} (note that $\| (S_b - \mu_b) g_b \| = \varsigma_b^{-1} \approx \rho$) to find
	\begin{equation}
		\begin{aligned}
			\|[S_b, \psi_b I_2] g_b \| \ls \rho x_b^{- \frac12}, \quad b \to +\infty.
		\end{aligned}
	\end{equation}
	Hence 
	\begin{equation}
	\| (S_b - \mu_b) \psi_b g_b\| = \varsigma_b^{-1} + \BigO(\rho (x_b^{- \frac12} + b^{-\frac12})), \quad b \to +\infty.
	\end{equation}
	Similarly, writing $\psi_b g_b = g_b + (\psi_b - 1)g_b$, we obtain
	\begin{equation}
	\|\psi_b g_b\| = 1 + \BigO(x_b^{-\frac 12} + b^{-\frac12}), \quad  b \to +\infty.
	\end{equation}
	Thus using \eqref{eq:lab.conv}, we arrive at
	\begin{equation}
		\left| \frac{\left\| (S_b - \mu_b) \psi_b g_b \right\|}{\|\psi_b g_b\|} - \frac1{\tilde\varsigma_b} \right| 
		= \BigO \left(\rho (x_b^{- \frac12} + b^{-\frac12}) \right), \quad b \to +\infty.
	\end{equation}

	Recalling from the proof of Proposition~\ref{prop:local.iR} that $S_b - \mu_b = \rho U_{b,\rho} \widetilde H_b U_{b,\rho}^{-1}$ and letting $u_b := U_{b,\rho}^{-1} \psi_b g_b$, then $u_b \in \Dom(H)$ with $\supp u_b \subset \Omega_b$ and we conclude
	\begin{equation}
		\left| \frac{\left\| H_b u_b \right\|}{\|u_b\|} - \frac1{\rho \tilde\varsigma_b} \right| 
		= \BigO(x_b^{-\frac 12} + b^{-\frac12}), \quad b \to +\infty,
	\end{equation}
	from which the claim follows.
\end{proof}

\subsection{Step 4: combining the estimates}
\label{ssec:step.4.iR}
With $\Omega_b'$, $\Omega_b$ and $\Delta_b$ from \eqref{eq:Deltab.def}, \eqref{eq:Omega.def}, let $\phi_b \in C_c^{\infty}(\Omega_b)$,  $0 \le \phi_b \le 1$, be such that
\begin{equation}\label{eq:phib'}
	\phi_b(x) = 1, \; x \in \Omega'_b, \quad \|\phi_b^{(j)}\|_{\infty} \ls \Delta_b^{-j}, \; j \in \{1, 2, \dots, n_0+1\},
\end{equation}		
with $n_0$ as in \eqref{n0.def}, and define
\begin{equation}
	\label{eq:phikb}
	\phi_{b,0}(x) := 1 - \phi_b(x), \quad \phi_{b,1}(x) := \phi_b(x), \quad x >0.
\end{equation}
\begin{lemma}
	\label{lem:Hb.phibk.commutator}
	Let the assumptions of Theorem~\ref{thm:iR} hold, let $H_b$ be as in \eqref{eq:Hb.def} and let $\phi_{b,k}, \; k \in \{0, 1\},$ be as in \eqref{eq:phikb}. Then for all $u \in \Dom(H)$, we have
	\begin{equation}
		\label{eq:Hbphikb.commutator.norm.est}
		\| [H_b, \phi_{b,k} I_2] u \| \ls x_b^{-1} \| H_b u \| + x_b^{-\frac12} \| u \|, \quad b \to +\infty.
	\end{equation}
\end{lemma}
\begin{proof}
	Let $u \in \Dom(H)$ and $\phi \in \CcR$, then straightforward calculations as in \eqref{eq:Sb.comm} show that
	\begin{equation}\label{eq:Hb.comm.particular}
		\begin{aligned}
			[H_b, \phi I_2] u = -i \sigma_1 \phi' u.
		\end{aligned}
	\end{equation}
	Noting that $\supp(\phi_{b,k}^{(j)} u) \cap \Omega_b' = \emptyset$, $1 \le j \le n_0$, $k \in \{1, 2\}$, we shall proceed by successive applications of Proposition~\ref{prop:away.iR}
	\begin{equation}\label{eq:Hb.away.Step1.est}
		\| \phi_{b,k}^{(j)} u \| \ls \frac{1}{V'(x_b) \Delta_b} \| H_b \phi_{b,k}^{(j)} u \|
	\end{equation}
	followed by commutator expansions such as
	\begin{equation}\label{eq:Hb.phibp.jorder.comm}
		H_b \phi_{b,k}^{(j)} u = \phi_{b,k}^{(j)} H_b u + [H_b, \phi_{b,k}^{(j)} I_2] u = \phi_{b,k}^{(j)} H_b u - i \sigma_1 \phi_{b,k}^{(j+1)} u
	\end{equation}
	until the remainder term is sufficiently small. Thus applying \eqref{eq:Hb.away.Step1.est}, \eqref{eq:Hb.phibp.jorder.comm}, \eqref{eq:phib'} and \eqref{eq:Deltab.def} for $1 \le j \le n_0$, we find
	\begin{equation}\label{phi.j.est}
		\begin{aligned}
		\| \phi_{b,k}^{(j)} u \| & \ls  \frac{1}{V'(x_b) \Delta_b} (\| \phi_{b,k}^{(j)} H_b u \| + \| \phi_{b,k}^{(j+1)} u \|)
		\\
		& \ls \frac{1}{V'(x_b) \Delta_b^{1+j}} \| H_b u \| + \frac{1}{V'(x_b) \Delta_b} \| \phi_{b,k}^{(j+1)} u \|, \quad b \to +\infty.
		\end{aligned}
	\end{equation}
	For the next step, recall that $V'(x_b) \Delta_b^2 = x_b$, see \eqref{eq:Deltab.def}.
	Hence, iterating \eqref{phi.j.est} and using \eqref{eq:phib'}, we deduce as $b \to +\infty$
	\begin{equation}
		\begin{aligned}
			\| [H_b, \phi_{b,k} I_2] u \| &= \| \phi_{b,k}' u \| 
			\\ & \ls \sum_{n=1}^{n_0} \frac{1}{(V'(x_b) \Delta_b^2 )^n}  \| H_b u \| + \frac{1}{(V'(x_b) \Delta_b)^{n_0}} \| \phi_{b,k}^{(n_0+1)} u \|\\
			&\ls  x_b^{-1} \| H_b u \| + \frac{1}{(V'(x_b) \Delta_b^2)^{n_0} \Delta_b}   \| u \|
			\\
			& 
			=  x_b^{-1} \| H_b u \| + \frac{ V'(x_b)^\frac12 }{ x_b^{n_0} } x_b^{-\frac12}   \| u \|,
		\end{aligned}
	\end{equation}
Finally, employing \eqref{n0.def}, we obtain the claim \eqref{eq:Hbphikb.commutator.norm.est}.
\end{proof}
\begin{proof}[Proof of Theorem~\ref{thm:iR}]
	Let $u \in \Dom(H)$, with $\phi_{b,k}$, $k \in \{0, 1\}$, as in \eqref{eq:phikb}, and write $u = u_0 + u_1$ where $u_0 := \phi_{b,0} u$ and $u_1 := \phi_{b,1} u$. Then
	\begin{equation}\label{eq:hbuexpansion}
		H_b u_k = \phi_{b,k} H_b u + [H_b, \phi_{b,k} I_2] u, \quad k \in \{0, 1\},
	\end{equation}
	and therefore by \eqref{eq:Hbphikb.commutator.norm.est} we get for $k \in \{0, 1\}$
	\begin{equation}\label{eq:hbuexpansion.est}
		\|H_b u_k\| \le \left(1 + \BigO(x_b^{-1}) \right)\|H_b u\| + \BigO(x_b^{-\frac12}) \| u \|, \quad b \to +\infty.
	\end{equation}	

	Next, note that $\supp u_1 \subset \Omega_b$, hence by Proposition~\ref{prop:local.iR}
	\begin{equation}
	\| u_1 \| \le \| (A_{m_b} - \mu_b)^{-1}\| (V'(x_b))^{-\frac12} \left(1 + \BigO( x_b^{-\frac12} + b^{-\frac12} ) \right) \| H_b u_1 \|,
	\end{equation}
	as $b \to +\infty$. Recall that by \eqref{Am.res.sum} and \eqref{mb.def}, we have
	\begin{equation}\label{Am.V.1}
	\| (A_{m_b} - \mu_b)^{-1}\| (V'(x_b))^{-\frac12} \approx m_b^{-1} (V'(x_b))^{-\frac12} = m^{-1}, \quad b \to + \infty.
	\end{equation}
	Thus by \eqref{eq:hbuexpansion.est}, we get as $b \to +\infty$
	\begin{equation}\label{eq:u1upperbound}
		\begin{aligned}
		\| u_1 \| &\le \| (A_{m_b} - \mu_b)^{-1}\| (V'(x_b))^{-\frac12} \left(1 + \BigO( x_b^{-\frac12} + b^{-\frac12} ) \right) \| H_b u \|
		\\ 
		& \quad + \BigO(x_b^{-\frac12}) \| u \|.
		\end{aligned}
	\end{equation}

 	Furthermore, since $\supp u_0 \cap \Omega'_b = \emptyset$, by Proposition~\ref{prop:away.iR}
	\begin{equation}
		\| u_0 \| \ls b^{-\frac12} \| H_b u_0 \|, \quad b \to +\infty
	\end{equation}
	and applying \eqref{eq:hbuexpansion.est}, we have as $b \to +\infty$
	\begin{equation}\label{eq:u2upperbound}
		\| u_0 \| \ls b^{-\frac12} \| H_b u \| + b^{-\frac12} x_b^{-\frac12}   \| u \|, \quad b \to +\infty.
	\end{equation}

	Combining \eqref{eq:u1upperbound}, \eqref{eq:u2upperbound} and \eqref{Am.V.1}, we find that as $b \to +\infty$
	\begin{equation}
		\begin{aligned}
			\| u \| & \le \| u_0 \| + \| u_1 \| 
			\\
			& \le \| (A_{m_b} - \mu_b)^{-1} \| (V'(x_b))^{-\frac12} \left(1 + \BigO(x_b^{- \frac12} + b^{-\frac12}) \right) \| H_b u \|
			\\
			&\quad + \BigO(x_b^{- \frac12}) \| u \|
		\end{aligned}
	\end{equation}
	and hence
	\begin{equation}\label{eq:Hb.est.iR}
		\| u \| \le \| (A_{m_b} - \mu_b)^{-1} \| (V'(x_b))^{-\frac12} \left(1 + \BigO(x_b^{- \frac12} + b^{-\frac12}) \right) \| H_b u \|.
	\end{equation}
	Together with Proposition~\ref{prop:lbound.iR}, \eqref{eq:Hb.est.iR} completes the proof of the theorem.
\end{proof}
%

\appendix

\section{Harmonic and conjugated oscillator}
\label{app:HO}

\subsection{Parabolic cylinder functions}
\label{ssec:app.par.cyl}
We follow~\cite[Chap.~12]{DLMF}. The parabolic cylinder functions $U(a,z)$ and $U(a,-z)$ (with $a,z \in \C$) are solutions of the equation
\begin{equation}
\frac{\dd^2 w}{\dd z^2} -\left(\frac{1}{4}z^2+a\right)w=0.	
\end{equation}
These solutions are linearly independent and their Wronskian is given by
\begin{equation}\label{Wronskian U}
	\mathscr{W} \{U(a,z), U(a,-z)\} = \frac{\sqrt{2\pi}}{\Gamma\left(\frac{1}{2}+a \right)},
\end{equation}
see \cite[(12.2.11)]{DLMF}. We also recall the recurrence relations
	\begin{equation}\label{Recurrence}
		\begin{aligned}
			U'(a,z)+\frac{1}{2}z U(a,z)+\left(a+\frac{1}{2}\right) U(a+1,z)&=0,\\
			U'(a,z)-\frac{1}{2}z U(a,z)+U(a-1,z)&=0,
		\end{aligned}
	\end{equation}
see \cite[(12.8.2),(12.8.3)]{DLMF}, and the asymptotic expansions when $z$ is a large \textit{real variable}
	\begin{align}
		|U(a,z)|&= e^{-\frac{1}{4}z^2} z^{-a-\frac{1}{2}}\left(1+\mathcal{O}\left(\frac{1}{|z|^2}\right)\right), & z \to +\infty, \label{Asymp U +}\\
		|U(a,z)|&= \frac{\sqrt{2\pi}}{\Gamma\left(\frac{1}{2}+a\right)} e^{\frac{1}{4}z^2} (-z)^{a-\frac{1}{2}}\left(1+\mathcal{O}\left(\frac{1}{|z|^2}\right)\right), & z \to -\infty, \label{Asymp U -}
	\end{align}
see \cite[(12.9.1),(12.9.2),(12.2.15)]{DLMF}.

\subsection{Harmonic oscillator}
\label{sec:app.ho}
We recall some properties of the self-adjoint harmonic oscillator in $L^2(\R)$
\begin{equation}
\begin{aligned}
\sL&=-\partial_{x}^2+x^2,& \Dom (\sL) & =\{v\in L^2(\R):(-\partial_{x}^2+x^2)v\in L^2(\R) \}
\\ & & & = H^2(\R) \cap \Dom(x^2);	
\end{aligned}	
\end{equation}
the separation of the domain is a consequence of the graph norm estimate
\begin{equation}\label{D.norm.def}
\|\sL f\|^2 + \|f\|^2 \gs \|f''\|^2 + \|x^2 f\|^2 + \|f\|^2 =: \|f\|_{\cD_{\sL}}^2, \quad f \in \cD_{\sL}:=\Dom(\sL). 
\end{equation}
The spectrum consists of the discrete simple eigenvalues
\begin{equation}\label{Spec HO}
	\sigma(\sL)=\{2n+1: n\in \N_0\}.
\end{equation}
We denote by $h_n$, $n \in \N_0$, the normalized Hermite functions, i.e., 
\begin{equation}\label{Eigen Os}
	\sL h_n = (2n+1) h_n, \quad \|h_n\| = 1, \quad n \in \N_0.
\end{equation} 
The resolvent can be expressed as an integral operator with an explicit kernel (in terms of the parabolic cylinder functions), namely (with $\la \in \rho(\sL)$)
\begin{equation}\label{Prop HO Res}
\begin{aligned}
( \sL -\la )^{-1}f(x) &= \int_{\R} G_{\la}(x,y)\,f(y) \,\dd y, \quad x \in \R,
\\
G_{\la}(x,y) &=\frac{\Gamma(\frac{1-\la}{2})}{2 \sqrt{\pi}} \times
\begin{cases}
U\left(-\frac{\la}{2}, \sqrt{2}x\right)U\left(-\frac{\la}{2}, -\sqrt{2} y\right), & x\geq y,
\\
U\left(-\frac{\la}{2},- \sqrt{2}x\right)U\left(-\frac{\la}{2}, \sqrt{2} y\right), & x\leq y.
\end{cases} 
\end{aligned}
\end{equation}

The proof follows the standard procedure to find the inverse of a Sturm-Liouville operator, where the properties of the parabolic cylinder function $U$ are used. We omit the details (see \eg~\cite[Chaps.~III,IV]{Titchmarsh-1962-book1}, \cite[Chap.~4]{Levitan-1991}, \cite[Chap.~9]{Coddington1955theory} or \cite[Chap.~10]{Zettl-2005-121}). 

\subsection{Conjugated and shifted oscillators}
\label{app:HO.conj}
The conjugated oscillator 
\begin{equation}
\sL_\la = - (\partial_y+\la)^2 + y^2, \quad \Dom(\sL_\la) = \Dom(\sL), \quad \la \in \R,	
\end{equation}
satisfies
\begin{equation}\label{HO.conj.app}
\sL_\la \phi(y) = e^{-\la y} \sL e^{\la y} \phi(y), \quad \phi \in \C_{c}^{\infty}(\R),
\end{equation}
and it is unitarily equivalent via the Fourier transform to the shifted oscillator, namely
\begin{equation}\label{HO.shift.def}
	\sL_\la = \cF 
	\left(
	\Dt + (x + i \la)^2
	\right) \cF^{-1}.
\end{equation}
The adjoint operator reads
\begin{equation}
\sL_\la^* = \sL_{-\la}.	
\end{equation}
The conjugated oscillator satisfies a graph norm estimate like $\sL$, see \eqref{D.norm.def}. Specifically, for any $\la \in \R$, there exists $c_\la>0$ such that 
\begin{equation}\label{L.la.gr}
\|\sL_\la f\|^2 + \|f\|^2 \geq c_\la \|f\|_{\cD_{\sL}}^2, \quad f \in \Dom(\sL_\la). 
\end{equation}
The spectrum of $\sL_\la$ is independent of $\la$, \ie~it coincides with that of the harmonic oscillator (including multiplicities)
\begin{equation}\label{Spec Conj Os}
\sigma(\sL_\la) = \sigma(\sL) = 2\N_0 +1, \quad \la \in \R.
\end{equation}
The eigenfunctions of $\sL_\la$ can be easily found using \eqref{HO.conj.app} and they satisfy
\begin{equation}\label{Eig Conj Os}
\sL_\la h_{n,\la} = (2n+1) h_{n,\la}, \quad n \in \N_0,
\end{equation}
where (recall that $h_n$, $n \in \N_0$, are the normalized Hermite functions)
\begin{equation}\label{Eig Conj Os HO}
h_{n,\la}(y) = e^{-\la y} h_n(y), \quad y \in \R, \ n \in \N_0.
\end{equation}
For each $\lambda\in \R$, we define 
\begin{equation}\label{Gamma Span}
\cE_{\lambda}=\operatorname{span} \left\{h_{n,\lambda} : n\in \N_{0}\right\}.
\end{equation}
It is known that $\cE_{\lambda}$ is dense in $L^2(\R)$, see \cite[Lem. 2.5]{Mityagin-2017-272}. For details of the proofs of the facts summarised above, see \cite{Krejcirik-2015-56,Mityagin-2017-272,Mityagin-2021-22}.

\section{Generalized coercivity and Schur complement dominant operator matrices}
\label{app:Schur.dom}

\subsection{Generalized coercivity}
\label{app:AH}

We recall the generalized Lax-Milgram-type theorems of Almog-Helffer.

\begin{theorem}[{\cite[Thm.\ 2.1]{Almog-2015-40}}] \label{thm:lm.0}
Let $\cV$ be a Hilbert space and let $\fra$ be a bounded sesquilinear form on $\cV$. Assume there exist $\Phi_1,\Phi_2 \in \mathcal{B}(\mathcal{V})$ and $c> 0$ such that for all $f \in \cV$ we have 
	\begin{equation}\label{lm.gen.coev}
		\begin{aligned}
			|\fra(f,f)| + |\fra(\Phi_1f,f)| &\geq c \|f\|_\mathcal{V}^2,
			\\
			|\fra(f,f)| + |\fra(f,\Phi_2f)| &\geq c \|f\|_\mathcal{V}^2.
		\end{aligned}
	\end{equation}
	Then the operator
	\begin{equation}\label{T.hat}
		\widehat A \in \cB (\cV, \cV^*), \qquad \langle \widehat A f, g \rangle_{\cV^* \times \cV} := \fra(f,g), \qquad f,g \in \cV,
	\end{equation}
	is boundedly invertible.
\end{theorem}

Let $\cV \subset \cH$ be continuously embedded and dense in another Hilbert space $\cH$. One can identify $\cH$ with its (anti-)dual $\cH^*$ in a standard way and consider
\begin{equation}\label{Gelfand}
	\cH \ni f \equiv \langle f, \cdot \rangle_\cH \in \cH^*, \qquad \cV \subset \cH \equiv \cH^* \subset \cV^*.
\end{equation}
In the above Hilbert space triple, the operator $\widehat A$ in~\eqref{T.hat} naturally defines a maximal restriction in $\cH$ which is formally given by
\begin{equation}\label{res.A.hat}
	A := (\iota_{\cV\to \cH}^{*})^{-1} \widehat A (\iota_{\cV\to \cH})^{-1}.
\end{equation}
Here $\iota_{\cV\to \cH}$ denotes the bounded embedding $\cV \hookrightarrow \cH$ and its adjoint is the (bounded) restriction operator $\cH^* \hookrightarrow \cV^*$. Under suitable assumptions, $A$ is boundedly invertible in $\cH$.

\begin{theorem}[{\cite[Thm.\ 2.2]{Almog-2015-40}}] \label{thm:lm.1}
	In addition to the assumptions of Theorem~{\rm\ref{thm:lm.0}}, let $\mathcal{V} \subset \cH$ be continuously embedded and dense in another Hilbert space $\cH$ and assume that $\Phi_1$, $\Phi_2$ extend to bounded operators on $\cH$. Then the operator in $\cH$ defined by 
	\begin{equation}\label{lm.op}
		\begin{aligned}
			\Dom (A) & := \big\{ f\in \cV \, : \, \exists \eta_f \in\cH,    \forall g \in \cV,   \fra (f, g) = \langle \eta_f, g \rangle_\cH \big\}, \\
			Af & := \eta_f,
		\end{aligned}
	\end{equation}
	is boundedly invertible and its domain is dense in $\cV$ and $\cH$.
\end{theorem}

\subsection{Schur complement dominant operator matrices}

We recall useful claims from \cite{Gerhat-2024-286}. Employing suitable Gelfand triples, they allow us to introduce operator matrices with non-empty resolvent set in the product space $\cH := \cH_1 \oplus \cH_2$ of two complex Hilbert spaces $\cH_1$ and $\cH_2$.

\begin{asm-sec} \label{asm:schur.dom}
	\begin{enumerate}[\upshape (i), wide]
		\item \label{item.spaces} Let $\cD_S$, $\cD_2$, $\cD_{-S}$, and $\cD_{-2}$ be complex Hilbert spaces such that, with continuous embeddings having dense ranges, 
		\begin{equation}\label{eq:triplets}
			\cD_S \subset \cH_1 \subset \cD_{-S}, \qquad \cD_2 \subset \cH_2 \subset \cD_{-2}.
		\end{equation}
		\item \label{item.entries.1} Suppose that
		\begin{equation}\label{op.hat.cond.bdd}
			\begin{aligned}
				\widehat A & \in \cB(\cD_S, \cD_{-S}), & \quad \widehat B & \in \cB (\cD_2, \cD_{-S}), \\
				\widehat C & \in \cB (\cD_S, \cD_{-2}), & \quad \widehat D  & \in \cB (\cD_2, \cD_{-2}). \\
			\end{aligned}
		\end{equation}
	\end{enumerate}
\end{asm-sec}

Under Assumption~\ref{asm:schur.dom}, we define the operator matrix 
\begin{equation}
	\widehat \cA := \left(
	\begin{array}{cc}
		\widehat A & \widehat B \\
		\widehat C & \widehat D
	\end{array}
	\right) \in \cB (\cD, \cD_-), \qquad
	\cD  := \cD_S \oplus \cD_2, \qquad
	\cD_-  := \cD_{-S} \oplus \cD_{-2}.
\end{equation}
Its first Schur complement
\begin{equation}
	\label{eq:def.hat.S}
	\quad 
	\widehat{S_1} (\la) := \widehat A - \la- \widehat B (\widehat D-\la)^{-1} \widehat C \in \cB (\cD_S, \cD_{-S})
\end{equation}
is defined for spectral parameters
\begin{equation}\label{eq:hat.rho}
	\la \in \rho ( \wh D) := \big\{ z \in \C \, : \, (\widehat D- z)^{-1} \in \cB (\cD_{-2}, \cD_{2}) \big\}.
\end{equation}
Finally, the corresponding (unbounded) maximal operators acting in $\cH$ and $\cH_1$ are 
\begin{equation}\label{Aa.S.max}
	\cA := \widehat \cA\vert_{\Dom(\cA)}, \qquad S_1 (\la):= \widehat{S_1}(\la)\vert_{\Dom (S_1(\la))},
\end{equation}
on their respective domains
\begin{equation}\label{eq:Schur.max.dom}
	\begin{aligned}
		\Dom (\cA) & := \big\{ (f,g) \in \cD \, : \, \widehat\cA (f,g) \in \cH \big\},\\
		\Dom (S_1(\la)) & := \big\{ f \in \cD_S \, : \, \widehat{S_1}(\la)f \in \cH_1 \big\}. 
	\end{aligned}
\end{equation}
Then  the spectra of $\cA$ and $S(\cdot)$ are related by the following theorem; see \cite[Chap.~IX]{EE} for the definitions of the essential spectra $\se{k}$.

\begin{theorem}[{\cite[Cor.~3.4~(ii), Cor.~3.5, Cor.~3.6, Cor.~3.7]{Gerhat-2024-286}}]
	\label{thm:Schur.dom}
	Let Assumption~{\rm{\ref{asm:schur.dom}}} be satisfied and let $\widehat{ S_1 }(\cdot)$, $\rho( \wh D)$, $\cA$ and $S_1(\cdot)$ be as in~\eqref{eq:def.hat.S},~\eqref{eq:hat.rho},~\eqref{Aa.S.max} and \eqref{eq:Schur.max.dom}. Let $\Sigma \subset \rho(\wh D)$ be such that
	\begin{equation}\label{Schur.asm}
		\forall\la \in \Sigma, \ \exists z_\la \in \C \ : \  (\widehat{S_1}(\la) - z_\la)^{-1} \in \cB (\cD_{-S}, \cD_S).
	\end{equation}
	Then the following spectra of $\cA$ and $S_1(\cdot)$ are equivalent on $\Sigma$. In detail,
	\begin{equation}
		\begin{aligned}
			& \la \in \sigma(\cA) \,\,  &&  \iff \quad 0 \in \sigma ( S_1 (\la)), \\
			& \la \in \spp(\cA) \,\, && \iff \quad 0 \in \spp (S_1 (\la)), \\
			& \la \in \se{2}(\cA) \,\, &&  \iff \quad 0 \in \se{2} (S_1 (\la)), \qquad \la \in \Sigma.
		\end{aligned}
	\end{equation}
	Moreover, if there exists $\la \in \Sigma \neq \emptyset$ with $0 \in \rho( S_1 (\la))$, then $\Dom (\cA)$ is dense in both $\cD$ and $\cH$.
\end{theorem}

The condition \eqref{Schur.asm} can be verified using the following lemma.

\begin{lemma}\label{lem:hat.ext}
	Let $\cD \subset \cH \subset \cD_-$ be a triple of Hilbert spaces where the respective embeddings are continuous and have dense range. Let $\widehat T \in \cB (\cD, \cD_-)$ and define
	\begin{equation}
		T := \widehat T|_{\Dom (T)}, \qquad \Dom (T) := \big\{ f \in \cD \, : \, \widehat T f \in \cH \big\}.
	\end{equation}
	Assume there exists $\la_0 \in \C$ such that
	\begin{equation}
		(\widehat T-\la_0)^{-1} \in \cB(\cD_-,\cD).
	\end{equation}
	Then, for all $\la \in \C$
	\begin{equation}
		\la \in \rho(T) \quad \iff \quad \la \in \rho(\widehat T) \quad : \iff \quad (\widehat T-\la)^{-1} \in \cB(\cD_-,\cD).
	\end{equation}
\end{lemma}
\begin{proof}
Suppose that $(T-\lambda)^{-1}\in \cB(\cH,\Dom (T) )$. Since $\widehat T-\lambda_{0}$ is bijective, $T-\lambda_{0}$ is bijective and $(T-\lambda_{0})^{-1}\subset (\widehat T-\lambda_{0})^{-1}$. Using the first resolvent identity, we have
\begin{equation}
\begin{aligned}
( T-\lambda)^{-1} =& ( T-\lambda_{0})^{-1} + (\lambda-\lambda_{0})(T-\lambda)^{-1}(T-\lambda_{0})^{-1}\\
\subset & (\widehat T-\lambda_{0})^{-1} + (\lambda-\lambda_{0})(T-\lambda)^{-1}(\widehat T-\lambda_{0})^{-1}
\\
& \eqqcolon R_{\lambda} \in B(\cD_{-},\cD).
\end{aligned}
\end{equation}
where for the boundedness we have used that $\cD\subset \cH$ is boundedly embedded and $(T-\lambda)^{-1} \in B(\cH, \Dom (T))$. We then have
\begin{equation}\label{Identities}
R_{\lambda}(\widehat T -\lambda)|_{\Dom (T)} = I_{\Dom (T)},\qquad (\widehat T -\lambda)R_{\lambda}|_{\cH} = I_{\cH}.
\end{equation}
Now we show that $\Dom (T)$ is dense in $\cD$. Let $x\in \cD$, then $(\widehat T-\lambda_{0}) x\in \cD_{-}$. Since $\cH$ is dense range embedded in $\cD_{-}$, there exists a sequence $\{f_{n}\}_{n\in\N}\subset\cH$ such that
\begin{equation} 
f_{n} \longrightarrow (\widehat T-\lambda_{0}) x  \text{ in }\cD_{-}.
\end{equation}
Applying the continuity of $(\widehat T-\lambda_0)^{-1}:\cD_{-} \to \cD$, we obtain
\begin{equation} 
(\widehat T-\lambda_0)^{-1}f_{n} \longrightarrow  x  \text{ in }\cD.
\end{equation}
Note that $(\widehat T-\lambda_0)^{-1}f_{n}\in \Dom (T)$, hence this shows the density of $\Dom (T)$ in $\cD$.
Combine this density with the continuity and dense range of the embeddings, the identities in \eqref{Identities} can be extended from $\Dom (T)$ to $\cD$ and from $\cH$ to $\cD_{-}$, respectively. This implies that $(\widehat T -\lambda)^{-1} = R_{\lambda}$.

Conversely, suppose $(\widehat T-\lambda)^{-1}\in \cB(\cD_{-},D)$. Let $u\in \cH$, since 
\begin{equation}
(T-\lambda)^{-1}u = (\widehat T-\lambda)^{-1}u,	
\end{equation}
by the assumption and the continuity of the embeddings,
\begin{equation}
\Vert (T-\lambda)^{-1} u \Vert_{\cD} = \Vert (\widehat T-\lambda)^{-1} u \Vert_{\cD} \lesssim \Vert u \Vert_{\cD_{-}} \lesssim \Vert u \Vert_{\cH}.
\end{equation}
\end{proof}
\bibliography{../references}
\bibliographystyle{acm}	

\end{document}

%% file: commands.tex
\usepackage{amsmath,amsthm, amssymb}
\usepackage{mathrsfs}
\usepackage[shortlabels]{enumitem}
\usepackage{graphicx}
\usepackage[font=small]{caption}
\usepackage{xcolor}
\usepackage{url}

\newcommand{\N}{\mathbb{N}}
\newcommand{\R}{{\mathbb{R}}}
\newcommand{\C}{{\mathbb{C}}}

\newcommand{\dd}{{{\rm d}}}

\newcommand{\ie}{{\emph{i.e.}}}
\newcommand{\eg}{{\emph{e.g.}}}

\newcommand{\ov}{\overline}
\newcommand\wt{\widetilde}
\newcommand\wh{\widehat}

\newcommand{\la}{\lambda}

\newcommand{\eps}{\varepsilon}

\newcommand{\se}[1]{\sigma_{\rm e#1}}
\newcommand{\spp}{\sigma_{\rm p}}

\newcommand{\Dom}{{\operatorname{Dom}}}
\newcommand{\Ker}{{\operatorname{Ker}}}
\newcommand{\Ran}{{\operatorname{Ran}}}

\renewcommand{\Re}{\operatorname{Re}}
\renewcommand{\Im}{\operatorname{Im}}
\newcommand{\dist}{\operatorname{dist}}
\newcommand{\sgn}{\operatorname{sgn}}
\newcommand{\supp}{\operatorname{supp}}
\newcommand{\Tr}{\operatorname{Tr}}

\newcommand{\Num}{\operatorname{Num}}

\newcommand{\BigO}{\mathcal{O}}

\newcommand{\lspan}{{\operatorname{span}}}

\newcommand{\rad}{{\operatorname{rad}}}

\newcommand{\opT}{T}

\newcommand{\intR}{\int_{\R}}

\newcommand{\Lt}{{L^2}}

\newcommand{\CcR}{{C_c^{\infty}(\R)}}

\newcommand{\Dt}{-\partial_x^2}

\newcommand{\Ntp}{\partial_x}

\newcommand{\ls}{\lesssim 	}
\newcommand{\gs}{\gtrsim}

\theoremstyle{plain}

\newtheorem{theorem}{Theorem}[section]
\newtheorem{lemma}[theorem]{Lemma}

\newtheorem{proposition}[theorem]{Proposition}
\newtheorem{corollary}[theorem]{Corollary}

\theoremstyle{definition}
\newtheorem{example}[theorem]{Example}
\newtheorem{remark}[theorem]{Remark}
\newtheorem{asm-sec}[theorem]{Assumption}

\newcommand\cA{\mathcal A}
\newcommand\cB{\mathcal B}
\newcommand\cD{\mathcal D}
\newcommand\cE{\mathcal E}
\newcommand\cF{\mathcal F}

\newcommand\cH{\mathcal H}

\newcommand\cM{\mathcal M}

\newcommand\cS{\mathcal S}
\newcommand\cT{\mathcal T}
\newcommand\cV{\mathcal V}

\newcommand\cX{\mathcal X}

\newcommand\fra{\mathfrak a}

\newcommand\frs{\mathfrak s}

\newcommand\sL{\mathscr L}